\documentclass[a4paper,11pt]{amsart}

\usepackage[utf8]{inputenc}
\usepackage[T1]{fontenc}

\usepackage{graphicx}
\usepackage{caption}
\usepackage{subcaption}
\usepackage{booktabs}
\usepackage{xcolor}

\usepackage{amsmath}
\usepackage{amssymb}
\usepackage{mathtools}

\usepackage{systeme}
\usepackage[squaren,Gray]{SIunits}
\usepackage{bbm}
\usepackage{stmaryrd}
\usepackage[overload]{empheq}
\usepackage{enumitem}
\usepackage{nicefrac}

\usepackage{aliascnt}

\newtheorem{thm}{Theorem}[section]

\newaliascnt{lemma}{thm}
\newtheorem{lemma}[lemma]{Lemma}
\aliascntresetthe{lemma}

\newaliascnt{prop}{thm}
\newtheorem{prop}[prop]{Proposition}
\aliascntresetthe{prop}

\newaliascnt{cor}{thm}
\newtheorem{cor}[cor]{Corollary}
\aliascntresetthe{cor}

\theoremstyle{definition}
\newaliascnt{definition}{thm}
\newtheorem{definition}[definition]{Definition}
\aliascntresetthe{definition}

\theoremstyle{remark}
\newaliascnt{remark}{thm}
\newtheorem{remark}[remark]{Remark}
\aliascntresetthe{remark}

\usepackage{hyperref}
\usepackage{cleveref}

\crefname{thm}{theorem}{theorems}
\Crefname{thm}{Theorem}{Theorems}

\crefname{lemma}{lemma}{lemmas}
\Crefname{lemma}{Lemma}{Lemmas}

\crefname{prop}{proposition}{propositions}
\Crefname{prop}{Proposition}{Propositions}

\crefname{cor}{corollary}{corollaries}
\Crefname{cor}{Corollary}{Corollaries}

\crefname{remark}{remark}{remarks}
\Crefname{remark}{Remark}{Remarks}

\crefname{definition}{definition}{definitions}
\Crefname{definition}{Definition}{Definitions}

\newcommand{\e}{\varepsilon}
\renewcommand{\r}{\rho}

\newcommand{\R}{\mathbb{R}}

\newcommand{\dt}{\partial_t}

\newcommand{\dx}{\partial_x}
\newcommand{\dxx}{\partial_{xx}}
\newcommand{\dz}{\partial_z}
\newcommand{\dzz}{\partial_{zz}}
\newcommand{\norm}[1]{\left\lVert#1\right\rVert}
\newcommand{\s}{\sigma}
\newcommand{\sgn}{\text{sign}}
\newcommand{\thresh}{{\rm th}}
\newcommand{\init}{\infty}

\title[Self-Generated Gradients and Cell Division]{When self-generated gradients meet cell division:
waves and speed selection in a minimal model}

\author{Mete Demircigil}

\address{Department of Mathematics, University of Arizona,
Tucson, Arizona, USA}

\email{mete.demircigil@math.cnrs.fr}



\usepackage[
backend=bibtex,
citestyle=numeric-comp,
bibstyle=numeric,
sorting=nyt,
maxnames=12,
giveninits=true,
abbreviate=true,
isbn=false,
url=false,
doi=false,
eprint=false,
date=year
]{biblatex}

\AtEveryBibitem{%
\clearfield{month}
\clearfield{note}
\clearfield{number}
\clearlist{language}
}

\renewbibmacro{in:}{}

\begin{document}

\begin{abstract}
We study a minimal Go-or-Grow model for cell propagation driven by self-generated nutrient gradients, coupling diffusion, division, and directed migration through a discontinuous switching mechanism. We establish local well-posedness, characterize all positive bounded traveling waves, and derive an explicit minimal wave speed. Using neutral fractions, we show that minimal waves undergo a pulled-to-pushed transition as the migratory bias increases, while nonminimal waves remain pulled. Finally, we obtain a weak speed-selection result that singles out the minimal wave speed as the only possible limiting propagation speed.
\end{abstract}

\maketitle

\section{Introduction}

In this paper, we analyze a minimal model for the collective motion of cells propagating as a traveling front. The model is a one-dimensional parabolic system coupling diffusion, advection, reaction, and nutrient consumption. For $t\geq0$ and $x\in\R$, the model reads
\begin{subequations}\label{diffusivemodel:main}
\begin{empheq}[left={\empheqlbrace\,}]{align}
\label{diffusivemodel:a}
&\dt \r - \dxx \r + \dx ( \chi \sgn(\dx N) \mathbbm{1}_{N\leq N_\thresh  }\r) =  \mathbbm{1}_{N>N_\thresh} \r,  \\ 
&\dt N - D \dxx N =-  \r N. \label{diffusivemodel:b}
\end{empheq}
\end{subequations}
Here $\r(t,x)$ describes the density of a diffusing cell population that either grows or is advected depending on the local nutrient concentration $N(t,x)$.
More precisely, for a threshold value $N_\thresh$, cells grow at rate $1$ when $N>N_\thresh$, whereas when $N\leq N_\thresh$, they are advected with speed $\chi>0$ in the direction of the gradient $\dx N$.
This advection speed results from biases in individual cell trajectories, which are averaged at the macroscopic level. The chemical nutrient $N$ satisfies a reaction-diffusion equation with a simple consumption term $-\r N$, with the consumption rate per cell normalized to $1$. 

Throughout the paper, we work in a monotone setting in which $N$ is strictly increasing in space, with
\begin{align*}
\lim_{x\to -\infty}N(t,x)=0,
\qquad
\lim_{x\to +\infty}N(t,x)=N_\init>N_\thresh.
\end{align*}
The constant $N_\init$ represents the nutrient concentration in the absence of cells. Under these assumptions, the cell population propagates from left to right. We denote by $\bar{x}(t)$ the unique position of the threshold, defined by
\begin{align}
\label{def-xt}
N(t,\bar{x}(t))=N_\thresh.
\end{align}
Because the advection term in \eqref{diffusivemodel:a} is discontinuous across the position of the threshold $\bar{x}(t)$, the flux has to be interpreted with care.
In particular, for weak solutions that are continuous and sufficiently regular on both sides of $\bar{x}(t)$, flux continuity yields the following Rankine--Hugoniot-type jump condition:
\begin{align}
\label{C1discRel}
\dx\r\left(t,\bar{x}(t)^+\right)-\dx\r\left(t,\bar{x}(t)^-\right)=-\chi \r(t,\bar{x}(t)).
\end{align}
We consider nonnegative initial data $(\r_{\rm in},N_{\rm in})\in L^\infty(\R)^2$, with
\begin{align*}
\dx N_{\rm in}>0,
\qquad
\lim_{x\to -\infty}N_{\rm in}(x)=0,
\qquad
\lim_{x\to +\infty}N_{\rm in}(x)=N_\init>N_\thresh,
\end{align*}
which are compatible with the interface condition above. We further assume that $\r_{\rm in}$ is bounded from above by an exponentially decaying function as $x\to+\infty$.

\begin{figure}
\begin{center}
\begin{subfigure}{.5\textwidth}
  \centering
\includegraphics[width=\linewidth]{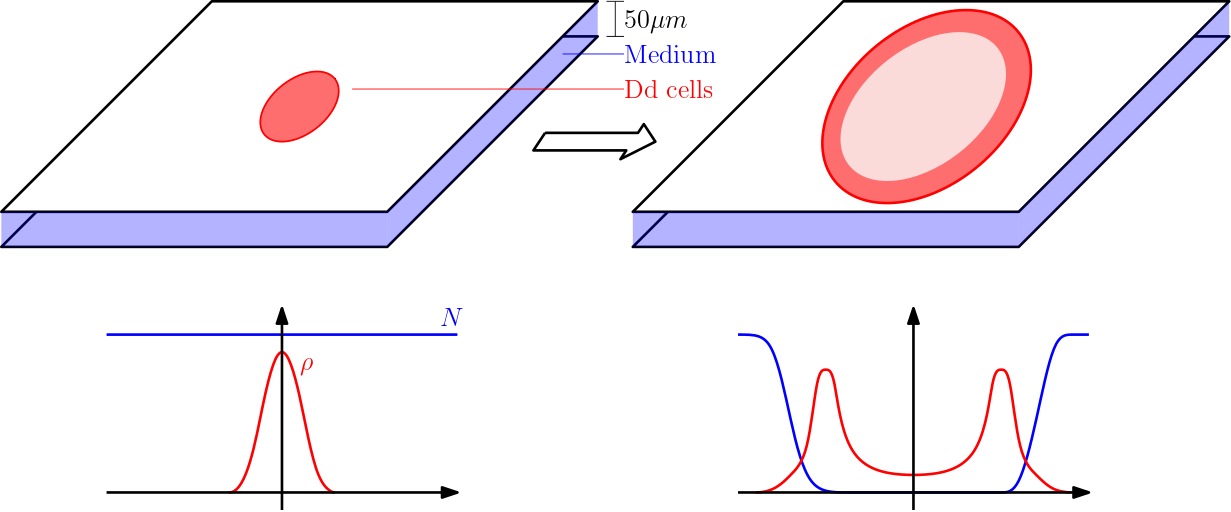}
  \caption{}
  \label{Fig-experience:sfig1}
\end{subfigure}%
\begin{subfigure}{.5\textwidth}
  \centering
\includegraphics[width=\linewidth]{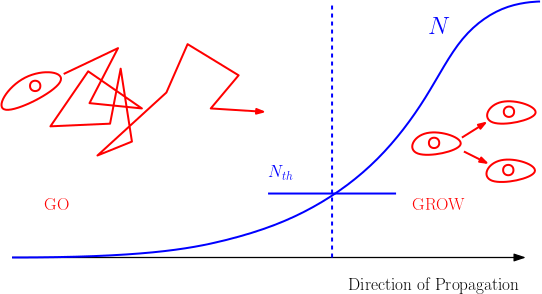}
  \caption{}
  \label{Fig-experience:sfig2}
\end{subfigure}
\caption{ {(A)}: A schematic representation of the experiment carried out in \cite{cochet2021}. Cells are confined between two narrowly spaced plates and quickly consume available oxygen, so that the colony experiences self-induced hypoxic conditions. This, in turn, triggers outward migration of the colony in the form of a ring expanding at constant speed over long periods of time. The bottom curves represent a cross section of the cell density and oxygen profiles. At initial time the oxygen concentration is homogeneous at the fixed value $N_\init$.
{(B)}: A cartoon  representation of the 'Go or Grow' hypothesis. Cells switch between two behaviors, depending on the level of oxygen. When oxygen concentration is above some threshold $N_\thresh$, cells divide and follow Brownian trajectories (this is the 'Grow' behavior). In contrast, when oxygen concentration drops below $N_\thresh$, cells stop dividing and follow a biased Brownian motion towards higher levels of oxygen (this is the 'Go' behavior).   
\label{Fig-experience} }
\end{center}
\end{figure}

System \eqref{diffusivemodel:main} was introduced in \cite{cochet2021} by the author and collaborators as a minimal model for collective cell migration driven by a self-generated gradient. In this study, the following collective behavior of \textit{Dictyostelium discoideum} cells (\textit{Dd} cells for short) was observed under hypoxic conditions: when a colony of \textit{Dd} cells is confined between two narrowly spaced plates, \textit{Dd} cells form a dense ring moving outwards. After a brief transient phase, the ring of cells moves at constant speed and constant density over the time course of the experiment; see \Cref{Fig-experience:sfig1}. 
The authors proposed that the rapid consumption of oxygen by \textit{Dd} cells exposes them to hypoxia, \textit{i.e.} a lack of oxygen, which in turn induces \textit{aerotaxis}, \textit{i.e.} a bias in individual cell trajectories towards higher oxygen concentrations, leading to a macroscopic outward motion.
We refer to \cite{cochet2021,biondo2021} (see also \cite{rieu2022}) for a biological discussion on the hypotheses under which the observed phenomenon may arise. 

The scientific approach in \cite{cochet2021} leading to this minimal model can be described as follows: (i) Experimentally, it was observed that cells exhibit various individual behaviors across the colony. (ii) In the model, two particular behaviors were retained: cell division and migration towards oxygen. (iii) It was postulated that the transition between the two behaviors depends on a single threshold; see \Cref{Fig-experience}. In particular, it is well known that \textit{Dd} cells lack sufficient energy to divide under oxygen-deprived conditions. The term ``Go or Grow'' was coined to describe this dichotomy, by analogy with a similar mechanism in models of glioma cells \cite{hatzikirou2012}. The latter is nevertheless of a different nature, as it describes a density-dependent rather than an oxygen-dependent switch between diffusion and cell division.

The mechanism of cell division in System \eqref{diffusivemodel:main} is reminiscent of standard reaction-diffusion models, among which the classical Fisher/Kolmogorov--Petrovsky--Piskunov model \cite{fisher1937,kolmogorov1937,aronson1975} is a prototype. In contrast to the F/KPP equation, where growth is limited by density-dependent saturation, here growth is limited through nutrient consumption: as cells divide, they consume $N$, which in turn limits further growth.

Interestingly, this minimal model is sufficient to describe the propagation of a wave of cells in a context of self-generated oxygen gradients. In fact, System \eqref{diffusivemodel:main} exhibits explicit traveling wave solutions and the minimal wave speed can be computed explicitly; see \cite{cochet2021} and also \Cref{Sect-TW}. The minimal wave speed is given by
\begin{align}
\label{speedformula}
\s^*=\left\{ \begin{array}{ll}
\chi+\nicefrac{1}{\chi} & \text{ if } \chi>1, \\
2 & \text{ if } \chi \leq 1.
\end{array} \right. 
\end{align}
Thus, there are two different regimes: in the regime $\chi\leq 1$, which we call the \textit{small bias} regime, the speed corresponds to the well-known F/KPP speed, $\s_{F/KPP}:=2$, whereas in the \textit{large bias} regime, $\chi>1$, the wave speed $\chi+\nicefrac{1}{\chi}$ is greater than $\s_{F/KPP}$. 

System \eqref{diffusivemodel:main} combines two distinct propagation phenomena: aerotaxis of cells triggered by the self-generated gradient, and expansion by division-diffusion, as described by the F/KPP model.
Biologically, it is therefore natural to ask how these two propagation phenomena combine with each other; see \cite{cochet2021,biondo2021}. Since the collective propagation speed $\s^*$ is always higher than the advection speed $\chi$, cell division contributes positively to the propagation of cells. 
Conversely, in the large bias regime $\chi>1$, aerotaxis also has a net positive effect on the propagation speed compared with expansion by division-diffusion alone. In that case, we refer to the wave as an aerotactic wave. By contrast, in the small bias regime $\chi\leq 1$, the propagation is driven by division-diffusion, and aerotaxis does not contribute to the wave speed; in that case, we refer to the propagation as a F/KPP wave.

A number of works, among which \cite{tweedy2020,insall2020,cremer2019,liu2019}, have investigated cell propagation under self-generated gradients and demonstrated its biological relevance as an efficient migration strategy. In particular, \cite{cremer2019} showed that the combination of \textit{chemotaxis}, namely the response to chemical gradients, and cell division can enhance population expansion, a mechanism closely related to the aerotactic effect considered here.

Other works have investigated chemotactic waves  in  \textit{Escherichia coli} bacteria; see in particular \cite{saragosti2011,calvez2019} and the references therein. These works consider models that are, to some extent, analogous to System \eqref{diffusivemodel:main}. A main difference, however, is that two chemical signals are required to sustain wave propagation in \cite{saragosti2011,calvez2019}, whereas here the single signal $N$ is sufficient. Moreover, cell division is negligible in the case of \textit{E. coli}, while it plays a key role here. We refer to \cite{sublet2021} for a discussion of different models of propagation through self-generated signaling gradients.

In parallel, the Keller--Segel model \cite{keller1970,keller1971} has been widely used to describe cells undergoing chemotaxis. In \cite{nadin2008}, a variant of this model was proposed by adding a density-dependent growth term. Keller--Segel models with growth have since been investigated in numerous works, including \cite{nadin2008,salako2017,salako2020,li2019,bramburger2021}. These models combine chemotaxis and cell division and exhibit traveling waves under suitable conditions on the parameters. However, chemotactic self-aggregation in these studies and aerotaxis in the present work lead to biases in opposite directions at the edge of the expansion front. In \cite{hamel2020,henderson2021}, the authors have investigated the case of negative chemotaxis, where the chemotactic bias acts in the same direction as the propagation induced by division-diffusion and therefore bears a closer similarity to the aerotactic advection term considered here. In \cite{henderson2021}, the author obtained bounds on the propagation speed. Since the model considered there is different from ours, these results cannot be directly compared with the explicit speed formula \eqref{speedformula}. Nevertheless, for small negative chemotaxis the propagation speed remains the F/KPP speed $\s^*=2$, whereas for large negative chemotaxis the speed increases, in qualitative agreement with our findings.

Following the initial circulation of the present work, the author and collaborators investigated related Go-or-Grow models from complementary viewpoints. In \cite{demircigil2025}, a Go-or-Grow rule was introduced directly at the level of a finite system of branching particles: the leading particles divide, while the remaining particles undergo biased motion. In the large-population limit, this rank-based rule becomes a nonlocal threshold condition on the amount of population lying ahead of a given position, playing a role analogous to the nutrient threshold $N\leq N_\thresh$ in System \eqref{diffusivemodel:main}. The resulting nonlocal reaction-advection-diffusion equation was subsequently studied in \cite{demircigil2026}, together with a related local Go-or-Grow model. An important feature of the nonlocal equation is that, when written in terms of the cumulative population, it reduces to a local equation enjoying a comparison principle. This additional structure allows for a more precise analysis of the propagation front. These models retain the sharp Go-or-Grow dichotomy of System \eqref{diffusivemodel:main}, while replacing the explicitly evolving nutrient field by switching mechanisms depending directly on the cell population.

Interestingly, Formula \eqref{speedformula} also arises in several classes of scalar propagation equations. A first class consists of monostable reaction-diffusion equations
\begin{align*}
\dt u-\dxx u=(u-A(u))(1+\chi^2A'(u)),
\end{align*}
for suitable increasing and convex nonlinearities $A$, with $A(0)=A'(0)=0$ and $A(1)=1$. The choice $A(u)=u^2$ gives the classical Hadeler--Rothe-type nonlinearity $u(1-u)(1+2\chi^2u)$; see \cite{hadeler1975,ben-jacob1985,giletti2021,AnHendersonRyzhik_convergence,AnHendersonRyzhik_Quantitative}. A second class is given by reaction-advection-diffusion equations
\begin{align*}
\dt u+\chi\dx(A(u))=\dxx u+u-A(u),
\end{align*}
which include the Burgers--F/KPP equation when $A(u)=u^2$; see \cite{an2021,demircigil2026,Patterson}. Finally, the same minimal wave speed appears for the related free-boundary problems introduced in \cite{berestycki2018} and further studied in \cite{berestycki2025},
\begin{empheq}[left={\empheqlbrace\,}]{align*}
\begin{array}{l}
\dt u - \dxx u =u \text{, for }x>\mu_t ,\\
u(t,\mu_t)=1 \text{ and } \dx u(t,\mu_t)=-\chi.
\end{array}
\end{empheq}
Under the corresponding structural assumptions, all three classes exhibit the same change of regime in the minimal wave speed at $\chi=1$ and the same minimal speed $\s^*$ given by \eqref{speedformula}.

There is also a direct structural analogy between the second class and System \eqref{diffusivemodel:main}. Indeed, under the monotonicity assumption $\dx N>0$, the equation for $\r$ can be written as
\begin{align*}
\dt\r+\chi\dx A(\r,N)=\dxx\r+\r-A(\r,N),
\qquad
A(\r,N):=\r\mathbbm{1}_{N\leq N_\thresh}.
\end{align*}
The essential difference is that, in the present model, the switching mechanism depends on the evolving nutrient field $N$, and is moreover discontinuous across the threshold $N=N_\thresh$, rather than being determined by a smooth scalar function of the cell density alone. Thus, although these scalar equations share the same wave-speed formula and provide useful mathematical analogies, the additional scalar structure and associated comparison principles available in those models are absent from System \eqref{diffusivemodel:main}.

A number of recent works on the scalar models described above \cite{giletti2021,an2021,AnHendersonRyzhik_convergence,AnHendersonRyzhik_Quantitative,demircigil2026,Patterson,berestycki2025} have gone beyond the determination of the minimal wave speed and investigated the precise asymptotic location of the propagation front, including Bramson-type logarithmic corrections to the leading-order front position $s^*t$. The question addressed here is somewhat different. Since System \eqref{diffusivemodel:main} was introduced directly from the biological setting of \cite{cochet2021}, our main concern is to determine which traveling wave is dynamically relevant for biologically meaningful initial data. In particular, we seek to identify whether the minimal speed $\s^*$ given by \eqref{speedformula} is the speed selected by the Cauchy problem. The coupling to $N$, however, prevents the direct use of the scalar comparison structures available in the models discussed above.

In \cite{garnier2012,roques2012}, the authors investigated the inside dynamics of traveling waves in reaction-diffusion equations and proposed a characterization of \textit{pushed} and \textit{pulled} waves. A pushed wave receives a significant contribution from the whole population to its net propagation, whereas a pulled wave is driven by growth and diffusion at the leading edge of the front, with negligible contribution from the bulk of the population. In particular, it was shown in \cite{garnier2012} that, for the monostable cubic reaction-diffusion equation with reaction term $f(u)=u(1-u)(1+2\chi^2u)$, a transition from a pulled to a pushed wave occurs at $\chi=1$.

The same dichotomy was observed numerically for System \eqref{diffusivemodel:main} in \cite{cochet2021}: in the small bias regime $\chi\leq 1$, the traveling wave is pulled, whereas in the large bias regime $\chi>1$, it is pushed.
One purpose of the present work is to give a rigorous counterpart to this observation through the analysis of the inside dynamics of the corresponding traveling waves.
Transitions from pulled to pushed waves have also been studied in \cite{pluess2011,gandhi2016,gandhi2019}; in those works, however, the transition is caused by structural modifications of the system, namely by a change of the reaction term from a monostable to a bistable structure. Here, as in the monostable cubic equation and the Burgers--F/KPP equation \cite{an2021}, the transition is instead induced by the relative size of the parameters, without changing the nature of the stable states.

Of note, the definitions of \textit{pulled} and \textit{pushed} waves can vary in the literature. The historical definition, proposed in \cite{stokes1976}; see also \cite{rothe1981,vansaarloos2003}, is based on whether the minimal speed $\s^*$ is equal to the speed of the linearized front around the steady state $0$ in the pulled case, or strictly greater than this speed in the pushed case. The definition proposed in \cite{garnier2012,roques2012} is instead based on the inside dynamics of the traveling waves. In this paper, we follow the latter definition, and more precisely study the inside dynamics of $\r$, which we believe is particularly well suited for a system of equations such as System \eqref{diffusivemodel:main}. However, both definitions coincide in our case, since linearizing System \eqref{diffusivemodel:main} around its leading edge yields a constant chemical nutrient field $N\equiv N_\init>N_\thresh$ and the linear F/KPP equation for $\r$, which gives rise to a front traveling at speed $\s=2$.

\section*{Results and Strategies of Proof}

The present article first addresses the local-in-time well-posedness analysis of System \eqref{diffusivemodel:main}, and then an asymptotic analysis including the computation of traveling wave solutions, the study of their inside dynamics, and a weak characterization of the spreading speed in the Cauchy problem.
Throughout the article we work in the setting where $N$ is increasing in space and cells propagate from left to right. Regarding well-posedness, we give a non-optimal criterion on the initial datum $N_{\rm in}$ under which monotonicity of $N$ is preserved locally in time and globally in space. For the rest of the article, however, the monotonicity of $N$ remains a restrictive assumption of our analysis.

In \Cref{Sect-Exist}, we prove a local-in-time existence and uniqueness result for System \eqref{diffusivemodel:main}, under the assumption that $N$ is initially monotone.
Because of the discontinuity of the advection coefficient, which is coupled to the signaling gradient, a direct application of Banach's Fixed Point Theorem does not seem straightforward. To circumvent this difficulty, we use the monotonicity of $N$ and the definition of the threshold position $\bar{x}(t)$; see \eqref{def-xt}. This allows us to apply Banach's Fixed Point Theorem to the curve $\bar{x}(\cdot)$ itself. Our strategy relies on an endpoint estimate of $N$ in $W^{3,\infty}$, which, to the best of our knowledge, lies beyond the range of standard textbook estimates. This estimate is obtained by carefully handling the discontinuity at the interface. The reduction to the dynamics of a curve coupled with a PDE is reminiscent of one-dimensional free-boundary problems; see, for instance, Chapter~3 in \cite{fasano2008} on the Stefan problem, \cite{lee2017} in the context of front propagation, and \cite{mirrahimi2015} in the context of mutation-selection dynamics in evolutionary biology.

In \Cref{Sect-TW}, we exhibit all positive and bounded traveling wave solutions of System \eqref{diffusivemodel:main} that propagate from left to right, \textit{i.e.} all stationary solutions in the frame $(t,z)=(t,x-\s t)$, with $\s\geq0$. This completes the preliminary analysis performed in \cite{cochet2021}. In this setting,  \eqref{diffusivemodel:a} reduces to a second-order differential equation with piecewise constant coefficients. All traveling wave profiles of $\r$ are obtained by concatenating a constant profile on the left with an exponentially decreasing profile on the right. In both the small and large bias regimes, there exists a minimal speed $\s^*$ for traveling waves, given by Formula \eqref{speedformula}. For each $\s\geq \s^*$, there exists an associated wave profile $(\r^\s,N^\s)$, whose exponential decay at $z=+\infty$ is at least as slow as the decay of the profile associated to the minimal speed $\s^*$. This observation will be used in the asymptotic analysis of \Cref{Sect-Asympt}. 

In \Cref{Sect-Inside}, we investigate the inside dynamics of the traveling waves. We introduce the formalism of neutral fractions \cite{garnier2012,roques2012} and extend it to System \eqref{diffusivemodel:main}. The methodology consists in following the relative proportion $\nu$ of a neutral subpopulation within the traveling wave $\rho^*$. This gives rise to a linear parabolic equation,
\begin{align*}
\dt \nu + L \nu =0.
\end{align*}
For the minimal wave in the large bias regime $\chi>1$, the elliptic operator $L$ is self-adjoint in a weighted $L^2$-space and has the following spectral properties: $0$ is a simple eigenvalue, whose eigenspace is spanned by the constants, and $L$ has a spectral gap. This implies exponential convergence of neutral fractions towards a constant, which constitutes the signature of a pushed wave according to \cite{garnier2012,roques2012}. In contrast, the minimal wave in the small bias regime $\chi\leq 1$, as well as every nonminimal wave $\s>\s^*$, is pulled. More precisely, for bounded initial neutral fractions that vanish at the leading edge, we show that
\begin{align*}
    \nu(t,\cdot) \longrightarrow 0
\end{align*}
uniformly on every interval of the form $[a,+\infty)$. The proof relies on weighted $L^2$-energy estimates together with a truncation argument.

Finally, in \Cref{Sect-Asympt}, we give a weak description of the asymptotic behavior of solutions to System \eqref{diffusivemodel:main}. If we define the instantaneous spreading speed to be $\dot{\bar{x}}(t)$, we show that, for initial conditions bounded above by a multiple of $\r^{\s^*}$ and under the technical assumption that $\dot{\bar{x}}\in L^\infty(\R_+)$,
\begin{align*}
 \liminf_{t\to +\infty} \dot{\bar{x}}(t) \leq \s^*, \text{ and }
 \limsup_{t\to +\infty} \dot{\bar{x}}(t) \geq \s^*.
\end{align*} 
In other words, for an initial cell profile whose exponential decay at the leading edge is at least as fast as that of $\r^{\s^*}$, the instantaneous spreading speed cannot remain, for all large times, strictly above or strictly below the minimal wave speed $\s^*$. This provides a weak speed-selection result and identifies the minimal wave speed as the biologically relevant candidate for the asymptotic propagation speed, in particular for initial cell profiles whose support is bounded from above. Stronger convergence of the solution towards the minimal traveling wave remains an open problem, one of the main difficulties being the lack of a suitable comparison principle.

\section{Local-in-Time Well-Posedness}
\label{Sect-Exist}

In this section, we establish local-in-time existence and uniqueness for System \eqref{diffusivemodel:main}, under certain conditions. The main difficulty in proving such a result stems from the singular advection term $\dx\left(\chi \sgn(\dx N)\mathbbm{1}_{N\leq N_\thresh}\r\right)$ in \eqref{diffusivemodel:a}. Therefore, we work in the framework where $N$ is increasing and $N(t,\cdot)=N_\thresh$ admits a unique solution $\bar{x}(t)$. In this framework, System \eqref{diffusivemodel:main} is equivalent to the following interface problem:
\begin{subequations}\label{ApproxSys:main}
\begin{empheq}[left={\empheqlbrace\,}]{align}
\label{ApproxSys:a}
&\dt \r - \dxx \r + \dx ( \chi  \mathbbm{1}_{x\leq \bar{x}(t)}\r) =  \mathbbm{1}_{x>\bar{x}(t)} \r,  \\ 
&\dt N - D\dxx N =-   \r N, \label{ApproxSys:b}\\ 
&N(t,\bar{x}(t))=N_\thresh. \label{ApproxSys:c}
\end{empheq}
\end{subequations}
However, monotonicity of $N$ is not preserved in general: it is possible to exhibit an initial configuration $(\r_{\rm in},N_{\rm in})$ where $N_{\rm in}$ is increasing but nearly constant and $\r_{\rm in}$ is sufficiently localized, so that the concentration $N(t,x)$ is no longer monotone after some time $t>0$, simply because of strong depletion around a spatial location. Thus, even under assumptions such as $N_{\rm in}$ being increasing, System \eqref{diffusivemodel:main} is not always equivalent to System \eqref{ApproxSys:main}. 
To circumvent this issue, we proceed in two steps. We introduce the two following properties for $N$:
\begin{itemize}
\item[\hspace{1cm}$\mathbf{(P_1):}$] $N:I\times\R\to \R$ and for all $t\in I$, $N(t,\cdot)$ is increasing.
\item[\hspace{1cm}$\mathbf{(P_2):}$] $N:I\times\R\to \R$ and for all $t\in I$, there exists a unique solution $\bar{x}(t)$ such that $N(t,\bar{x}(t))=N_\thresh$.
\end{itemize}
The organization of this section reflects this two-step strategy. In \Cref{Sect-Exist-Interface}, we prove \Cref{ThmExistence}, which establishes existence and uniqueness for the interface problem \eqref{ApproxSys:main}, together with Property $\mathbf{(P_2)}$. In \Cref{Sect-Exist-Monotonicity}, we prove \Cref{cor-existence}, which gives a sufficient, non-optimal criterion on the initial data $(\r_{\rm in},N_{\rm in})$ ensuring that the function $N$ given by \Cref{ThmExistence} also satisfies Property $\mathbf{(P_1)}$ for small time. Thus, under this criterion, the solution of the interface problem is in fact the unique solution to System \eqref{diffusivemodel:main}.

\subsection{The Interface Problem}
\label{Sect-Exist-Interface}

We first study the interface problem \eqref{ApproxSys:main} independently of the question of whether monotonicity of $N$ is preserved. In this formulation, the threshold position $\bar{x}(t)$ is treated as an unknown curve coupled to the parabolic system, and the objective is to construct a solution for which Property $\mathbf{(P_2)}$ holds.

The proof is based on a fixed-point argument on the interface position. Given a curve $t\mapsto \bar{x}(t)$, we first solve Equations \eqref{ApproxSys:a}--\eqref{ApproxSys:b} and denote the corresponding solution by $(\r[\bar{x}],N[\bar{x}])$. We then define a new curve $\bar{y}(\cdot)$ as the solution of the threshold equation
\begin{align*}
N[\bar{x}](t,\bar{y}(t))=N_\thresh.
\end{align*}
At least formally at this stage, $\bar{y}=\bar{y}[\bar{x}]$ is expected to solve the ODE
\begin{align}
\label{ODEy}
\left\{
\begin{array}{ll}
\dot{\bar{y}}(t)=-\frac{\dt N[\bar{x}](t,\bar{y}(t))}{\dx N[\bar{x}](t,\bar{y}(t))},\\
\bar{y}(0)=\bar{x}_{\rm in}.
\end{array} 
\right..
\end{align}
The fixed-point problem then consists in finding a curve $\bar{x}$ such that $\bar{y}[\bar{x}]=\bar{x}$, or equivalently in solving the following Cauchy problem:
\begin{align*}
\left\{
\begin{array}{ll}
\dot{\bar{x}}(t)=-\frac{\dt N[\bar{x}](t,\bar{x}(t))}{\dx N[\bar{x}](t,\bar{x}(t))},\\
\bar{x}(0)=\bar{x}_{\rm in}.
\end{array} 
\right..
\end{align*}
This reduction to the dynamics of an interface curve is reminiscent of one-dimensional free-boundary problems; see the aforementioned works \cite{fasano2008,lee2017,mirrahimi2015}.

To make this construction rigorous, one needs enough regularity on $N[\bar{x}]$ to ensure well-posedness of the ODE for $\bar{y}$. More precisely, $\dt N[\bar{x}]$ and $(\dx N[\bar{x}])^{-1}$ should be locally Lipschitz in space, with $\dx N[\bar{x}](t,\bar{y}(t))$ uniformly bounded away from $0$. Since parabolic regularity suggests that $\dt N$ has the same spatial regularity as $\dxx N$, this leads to the requirement that $\dxx N[\bar{x}]$ be locally Lipschitz in space. In other words, the relevant regularity is $W^{3,\infty}$ in space. This is a delicate point: the discontinuity induced by the interface puts the estimate at an endpoint of the usual parabolic regularity theory. We shall nevertheless prove that this endpoint regularity holds.

Finally, the contraction argument differs slightly from the standard Picard--Lindelöf argument for ODEs. Because the construction requires control of $\dot{\bar{x}}$, the fixed-point space must also control the time derivative of the curve. However, the map $\bar{x}\mapsto\bar{y}$ is, at best, Lipschitz continuous in the $W^{1,\infty}$-norm, and not a contraction. To recover a contraction property, we work instead in a $W^{1,p}$-norm, with $p<\infty$. The local-in-time integration in the ODE then provides the small factor in $T$ needed to make the map $\bar{x}\mapsto\bar{y}$ contractive.

We introduce some notation and preliminary estimates before stating the main theorem of this subsection.

We denote by $e^{t\mu\dxx}$ the heat semigroup on the real line, with diffusion coefficient $\mu>0$. In what follows, only the cases $\mu=1$ and $\mu=D$ will be used. For $t>0$,
\begin{align*}
e^{t\mu\dxx}f (x) = \frac{1}{\sqrt{4 \pi \mu t}} \int_{\R} e^{-\frac{(x-y)^2}{4\mu t}}  f(y)dy.
\end{align*}
As a consequence of Young's convolution inequality, for $1\leq p \leq q \leq \infty$, the following estimates hold:
\begin{align}
&\norm{e^{t\mu \dxx } f}_{q} \leq Ct^{-\frac{1}{2}\left(\frac{1}{p}-\frac{1}{q}\right) }\norm{f}_p, \label{Young} \\
& \norm{ \dx  e^{t\mu \dxx } f}_{q} \leq Ct^{-\frac{1}{2}\left(\frac{1}{p}-\frac{1}{q}+1\right) }\norm{f}_p. \label{dxYoung}
\end{align}
We also use Hölder spaces $C^{k,\alpha}(\R)$, with $\alpha \in (0,1)$, equipped with the norm
\begin{align*}
\norm{f}_{C^{k,\alpha}(\R)}
=
\sum_{i=0}^k\norm{\dx^i f}_\infty+[\dx^k f]_\alpha,
\qquad
[f]_\alpha:=\sup_{x\neq y}\frac{|f(x)-f(y)|}{|x-y|^\alpha}.
\end{align*}
Using the standard interpolation estimate
\begin{align*}
[f]_\alpha \leq C \norm{f}_\infty^{1-\alpha} \norm{\dx f}_\infty ^\alpha,
\end{align*}
we get on the one hand the following standard bounds:
\begin{align}
\label{HolderHeat}
&\norm{e^{t\mu \dxx} f}_{C^{0,\alpha}(\R)} \leq C\left(1+t^{-\frac{\alpha}{2}}\right)\norm{f}_\infty,\\
\label{dxHolderHeat}
&\norm{e^{t\mu \dxx} f}_{C^{1,\alpha}(\R)} \leq C\left(1+t^{-\frac{1+\alpha}{2}}\right)\norm{f}_\infty,
\end{align}
where the constant $C$ depends on $\mu,\alpha$.
On the other hand, we have also the bounds:
\begin{align}
\label{dxReverseHolderHeat}
&\norm{ e^{t\mu \dxx} f}_{W^{1,\infty}(\R)} \leq C\left(1+t^{\frac{\alpha-1}{2}}\right)\norm{ f}_{C^{0,\alpha}(\R)},\\
\label{dxxReverseHolderHeat}
&\norm{ e^{t \mu\dxx} f}_{W^{2,\infty}(\R)} \leq C\left(1+t^{\frac{\alpha}{2}-1}\right)\norm{ f}_{C^{0,\alpha}(\R)}.
\end{align}
The singular parts of these estimates follow from the following standard computation. For $(t,x)\in \R_+\times \R$ and $k\geq 1$,
\begin{align*}
&\left|\dx^k e^{t\mu \dxx} f (x)\right| \\
=& \left|\frac{1}{\sqrt{4 \pi \mu t}} \int_{\R} \dx^k \left(e^{-\frac{(x-y)^2}{4\mu t}}\right)  (f(y)-f(x))dy\right| \\
\leq& \frac{[f]_\alpha}{\sqrt{4 \pi \mu t}}\int_{\R} \left|\dz^k \left( e^{-\frac{z^2}{4\mu t}}\right)\right| |z|^\alpha\, dz\\
\leq& Ct^{\frac{\alpha-k}{2}}[f]_\alpha.
\end{align*}

We now introduce the space in which the fixed-point argument on the interface position will be carried out. For $A>0$, $T>0$, and $p\in(4,\infty)$, set
\begin{align}
\label{defBA}
B(A):=\left\{ \bar{x}(\cdot)\in  W^{1,p}([0,T])  \middle |  \bar{x}(0)=\bar{x}_{\rm in} \text{ and } \norm{\bar{x}}_{W^{1,p}}\leq A\right\},
\end{align}
equipped with the norm
\begin{align*}
\norm{y}_{W^{1,p}}=\norm{y}_{p}+\norm{\dot y}_{p}.
\end{align*}
Given a curve $\bar x \in B(A)$, we consider System \eqref{ApproxSys:main} in the moving frame $(t,z)=(t,x-\bar x (t))$:
\begin{align*}
&\tilde{\r}(t,z):=\r(t,\bar{x}(t)+z),\\
&\tilde{N}(t,z):=N(t,\bar{x}(t)+z).
\end{align*}
Then $\tilde{\r}$ and $\tilde{N}$ satisfy
\begin{subequations}\label{movingmodel:main}
\begin{empheq}[left={\empheqlbrace\,}]{align}
&\dt \tilde \r -\dot {\bar x} (t)\dz \tilde \r- \dzz \tilde \r + \dz (  \chi\mathbbm{1}_{z\leq 0  }\tilde \r) =  \mathbbm{1}_{z>0} \tilde \r, \label{movingmodel:a} \\
&\dt \tilde N-\dot {\bar x} (t)\dz \tilde  N - D\dzz \tilde N =-   \tilde\r \tilde N. \label{movingmodel:b}
\end{empheq}
\end{subequations}
Throughout the proof of \Cref{ThmExistence}, $\tilde{\r}$ and $\tilde{N}$ denote the solutions in the moving frame. 

As observed in the introduction, continuity of the flux in Equation \eqref{movingmodel:a} leads, by a Rankine--Hugoniot-type argument, to the jump condition
\begin{align*}
\dz\tilde{\r}(t,0^+)-\dz \tilde{\r}(t,0^-)=-\chi \tilde{\r}(t,0).
\end{align*}
We absorb this jump by introducing the factorization
\begin{align*}
\tilde{\r} &= vU, \\
U(z) &=
\left\{
\begin{array}{ll}
1 & \text{if } z\leq 0,\\
e^{-\chi z} & \text{if } z>0.
\end{array}
\right.
\end{align*}
The function $U$ is chosen precisely so that this factorization cancels the jump condition at $z=0$. The resulting unknown $v$ has higher spatial regularity than $\tilde \r$. More precisely, $v$ will be $C^{1,\alpha}$ in space. A direct computation gives the equation
\begin{subequations}\label{vEqn:main}
\begin{empheq}[left={\empheqlbrace\,}]{align}
\label{vEqn:a}
&\dt v -\dzz v -\beta(t,z)\dz v -\gamma(t,z)v =0,\\
&\beta(t,z):=\dot{\bar x}(t)-\chi \mathbbm{1}_{z\leq 0} -2\chi\mathbbm{1}_{z> 0},\label{vEqn:b}\\
&\gamma(t,z):=\chi \left(\chi+\frac{1}{\chi}-\dot{\bar x}(t) \right)\mathbbm{1}_{z> 0}. \label{vEqn:c}
\end{empheq}
\end{subequations}
The particular choice of $U$ is not essential for the local well-posedness argument. Other weights could be used, provided that the factorization cancels the jump condition at $z=0$. We use this specific $U$ for simplicity; as will be seen in \Cref{Sect-TW}, it corresponds to the traveling wave profile in the case $\chi\geq 1$.

We can now state the local well-posedness theorem for the interface problem \eqref{ApproxSys:main}.

\begin{thm}[Local well-posedness of the interface problem]
\label{ThmExistence}
Let $\alpha \in \left(0,\nicefrac{1}{2}\right)$ and suppose that
\begin{align*}
\nicefrac{\r_{\rm in}}{U}\in C^{1,\alpha}(\R),
\qquad
N_{\rm in}\in W^{3,\infty}(\R), \hspace{1.7cm}\\
\dx N_{\rm in}>0, 
\qquad
\lim_{x\to -\infty} N_{\rm in}(x)=0,
\qquad
\lim_{x\to +\infty} N_{\rm in}(x)=N_\infty. 
\end{align*}
For any $p\in\left(4,+\infty\right)$, there exist $A>0$ and $T>0$ such that the following holds.

For any curve $\bar{x}\in B(A)$, where $B(A)$ is given by \eqref{defBA}, there exists a unique solution $(\tilde{\r}[\bar{x}],\tilde{N}[\bar{x}])$ to System \eqref{movingmodel:main}. This solution satisfies
\begin{itemize}
\item $v[\bar{x}]=\nicefrac{\tilde{\r}[\bar{x}]}{U}\in L^\infty([0,T],C^{1,\alpha}(\R))$,
\item $\tilde{N}[\bar{x}]\in L^\infty([0,T], W^{3,\infty}(\R))$.
\end{itemize}
Moreover, the corresponding function $N[\bar{x}]$ in the static frame admits, for each $t\in[0,T]$, a unique threshold position $\bar{y}[\bar{x}](t)$ such that
\begin{align*}
N[\bar{x}](t,\bar{y}[\bar{x}](t))=N_\thresh.
\end{align*}

Finally, there exists a unique fixed point $\bar{x}=\bar{y}[\bar{x}]$ in $B(A)$. For this unique curve $\bar{x}$, the corresponding solution satisfies
\begin{align*}
\tilde{N}[\bar{x}](t,0)=N_\thresh,
\qquad t\in[0,T],
\end{align*}
or equivalently, in the static frame,
\begin{align*}
N[\bar{x}](t,\bar{x}(t))=N_\thresh,
\qquad t\in[0,T].
\end{align*}
Therefore, $(\r[\bar{x}],N[\bar{x}])$ in the static frame is the unique solution to the interface problem \eqref{ApproxSys:main}.
\end{thm}

\begin{proof}
We divide the proof into four steps. Let
\begin{align*}
\alpha'\in\left(\frac{2}{p},1-\frac{2}{p}\right),
\end{align*}
which is possible since $p>4$.

\begin{enumerate}[label=\textit{Step \arabic*.}, leftmargin=*]
\item For a fixed curve $\bar{x}\in B(A)$, we construct the solution of the parabolic equations in the moving frame. More precisely:
\begin{enumerate}
\item We construct the unique mild solution
\begin{align*}
v\in L^\infty([0,T],C^{1,\alpha}(\R))
\end{align*}
to Equation \eqref{vEqn:a}, and show that the map
\begin{align*}
\bar x \in B(A)\longmapsto v[\bar x]\in L^\infty([0,T],C^{1,\alpha}(\R))
\end{align*}
is Lipschitz continuous.

\item Given $v$, and hence $\tilde{\r}=vU$, we construct the unique mild solution
\begin{align*}
\tilde N \in L^\infty([0,T],C^{2,\alpha'}(\R))
\end{align*}
to Equation \eqref{movingmodel:b}, and show that the map
\begin{align*}
\bar x \in B(A)\longmapsto \tilde N[\bar x]\in L^\infty([0,T],C^{2,\alpha'}(\R))
\end{align*}
is Lipschitz continuous.
\end{enumerate}

\item We prove the enhanced regularity
\begin{align*}
\tilde N\in L^\infty([0,T], W^{3,\infty}(\R)).
\end{align*}
This improves the regularity obtained in Step~1 and is crucial for the definition of the interface dynamics. The estimate relies on explicit computations near the interface. Remark~\ref{rk-optimality-N} explains why this regularity is critical.

\item We show that, for each $t\in[0,T]$, the equation
\begin{align*}
N(t,\cdot)=N_\thresh
\end{align*}
in the static frame admits a unique solution, denoted by $\bar{y}(t)$. Moreover, $\bar{y}$ satisfies ODE \eqref{ODEy}, and the regularity obtained in Step~2 yields well-posedness of this ODE.

\item Finally, for $A>0$ large enough and $T>0$ small enough, we prove that the map $\bar{x}\mapsto \bar{y}$ maps $B(A)$ into itself and is a contraction. Banach's fixed-point theorem then gives the desired solution of the interface problem.
\end{enumerate}

\textit{Conventions.} Throughout the proof, $C$ denotes constants that may depend on the parameters of the problem and on the initial data, but not on $A$ or $T$. To simplify the presentation of the estimates, we assume without loss of generality that
\begin{align*}
A,\quad
\norm{\nicefrac{\r_{\rm in}}{U}}_{C^{1,\alpha}},\quad
\norm{N_{\rm in}}_{W^{3,\infty}},\quad
\norm{N_{\rm in}}_{C^{2,\alpha'}}
\end{align*}
are larger than $1$. This allows us to use elementary bounds such as $1+A\leq 2A$ without further comment. We also assume that $T<1$, so that, for instance, $|t-s|^{-c}\leq |t-s|^{-d}$ for $t,s\in[0,T]$ whenever $0<c<d$.

We follow the steps of the proof. A lot of the technical details are postponed to the end of the proof.

\medskip

\textit{Step 1: Solving the moving-frame system for a prescribed interface $\bar{x}(\cdot)$.}

\medskip

\textit{Step 1. (a): Existence and uniqueness of a mild solution $v$ to \eqref{vEqn:a}.} 

Consider the affine map $F_{\bar{x}}:L^\infty\left([0,T],C^{1,\alpha}(\R)\right) \to L^\infty\left([0,T],C^{1,\alpha}(\R)\right)$, for $t\in [0,T]$:
\begin{align}
\label{vMap}
F_{\bar{x}}[u](t) = e^{t \dzz } v_{\rm in}+\int_0^t e^{(t-s) \dzz }\left( \beta(s)\dz u(s)+ \gamma(s)u(s) \right)\, ds,
\end{align}
where $\beta,\gamma$ are given by \eqref{vEqn:b}, \eqref{vEqn:c} and depend on $\bar{x}$.

\begin{prop}[The solution map $\bar{x}\mapsto v{[\bar{x}]}$ ]
\label{pStep1a}
Let $\alpha \in \left( 0, \nicefrac{1}{2}\right), p \in \left(\nicefrac{2}{1-\alpha},+\infty\right)$ and $v_{\rm in}\in C^{1,\alpha}(\R)$. 
Fix $A>\max\{|\bar{x}_{\rm in}|,1\}$. Then there exists $T\in (0,1)$ such that:
\begin{enumerate}[label=(\roman*)]
\item  For every $\bar{x}\in B(A)$, the functional $F_{\bar{x}}$ admits a unique fixed point denoted by ${v}[\bar{x}]$. This fixed point is the unique solution ${v}[\bar{x}]\in L^\infty\left([0,T],C^{1,\alpha}(\R)\right)$ to Equation \eqref{vEqn:a}.
\item The solution satisfies the bound
\begin{align}
\label{vbd}
\norm{ v[\bar{x}]}_{L^\infty([0,T],C^{1,\alpha}(\R))}
\leq
2\norm{v_{\rm in}}_{C^{1,\alpha}(\R)}.
\end{align}
\item There exists a function $\Theta_\rho:(0,1)\to\R_+$, depending on $A$, $v_{\rm in}$ and the parameters of the problem, such that $\Theta_\rho(T)\to0$ as $T\to0$ and, for all $\bar{x}_1,\bar{x}_2\in B(A)$,
\begin{align}
\label{vMap-Lip}
\norm{v[\bar{x}_1]-v[\bar{x}_2]}_{L^\infty([0,T],C^{1,\alpha}(\R))}
\leq
\Theta_\rho(T)
\norm{\dot{\bar{x}}_1-\dot{\bar{x}}_2}_{p}.
\end{align}
\end{enumerate}
\end{prop}

In the present setting, we use the initial condition
\begin{align*}
v_{\rm in}=\nicefrac{\rho_{\rm in}}{U}.
\end{align*}
By the assumptions of \Cref{ThmExistence}, we have $v_{\rm in}\in C^{1,\alpha}(\R)$. Moreover, \Cref{ThmExistence} assumes $p>4$, and therefore in particular $p>\nicefrac{2}{1-\alpha}$, since $\alpha < \nicefrac{1}{2}$. Hence \Cref{pStep1a} applies.

The proof of \Cref{pStep1a} relies only on the heat-kernel estimates stated above and on Hölder's inequality in time. We postpone it to the end of the proof of \Cref{ThmExistence}.

\medskip

\textit{ Step 1. (b): Existence and uniqueness of a mild solution $\tilde N$ to \eqref{movingmodel:b}.}

Consider the map $ G_{\bar{x}}: L^\infty([0,T],C^{2,\alpha'}(\R)) \to L^\infty([0,T],C^{2,\alpha'}(\R))$, with $\alpha' \in (\frac{2}{p},1-\frac{2}{p})$. For $t\in [0,T]$ and $n \in L^\infty([0,T],C^{2,\alpha'}(\R))$:
\begin{align}
\label{NMap}
G_{\bar{x}}[n](t)= e^{t D\dzz } N_{\rm in}+\int_0^t e^{(t-s)D \dzz }\left( \dot{\bar x}(s) \dz n(s)-  v[\bar{x}](s)U n(s) \right)ds.
\end{align}

\begin{prop}[The solution map $\bar x\mapsto \tilde N{[\bar x]}$]
\label{pStep1b}
Let $\alpha \in \left(0,\nicefrac{1}{2}\right)$, $p\in(4,+\infty)$, and
\begin{align*}
\alpha'\in\left(\frac{2}{p},1-\frac{2}{p}\right).
\end{align*}
Assume that $N_{\rm in}\in W^{3,\infty}(\R)$ and $v_{\rm in}\in C^{1,\alpha}(\R)$. Fix $A>\max\{|\bar{x}_{\rm in}|,1\}$. Then, up to choosing $T\in(0,1)$ smaller than in \Cref{pStep1a}, the following holds:
\begin{enumerate}[label=(\roman*)]
\item For every $\bar{x}\in B(A)$, the functional $G_{\bar{x}}$ admits a unique fixed point, denoted by $\tilde N[\bar{x}]$. This fixed point is the unique solution
\begin{align*}
\tilde N[\bar{x}]\in L^\infty([0,T],C^{2,\alpha'}(\R))
\end{align*}
to Equation \eqref{movingmodel:b}.

\item The solution satisfies the bound
\begin{align}
\label{Nbd}
\norm{\tilde N[\bar{x}]}_{L^\infty([0,T],C^{2,\alpha'}(\R))}
\leq
2\norm{N_{\rm in}}_{W^{3,\infty}(\R)}.
\end{align}

\item There exists a function $\Theta_N:(0,1)\to\R_+$, depending on $A$, $v_{\rm in}$, $N_{\rm in}$ and the parameters of the problem, such that $\Theta_N(T)\to0$ as $T\to0$ and, for all $\bar{x}_1,\bar{x}_2\in B(A)$,
\begin{align}
\label{NMap-Lip}
\norm{\tilde{N}[\bar{x}_1]-\tilde{N}[\bar{x}_2]}_{L^\infty([0,T],C^{2,\alpha'}(\R))}
\leq
\Theta_N(T)
\norm{\dot{\bar{x}}_1-\dot{\bar{x}}_2}_{p}.
\end{align}
\end{enumerate}
\end{prop}

In the present setting, the assumptions of \Cref{pStep1b} follow from those of \Cref{ThmExistence} and from the choice $v_{\rm in}=\nicefrac{\r_{\rm in}}{U}$. Since $p>4$, the interval
\begin{align*}
\left(\frac{2}{p},1-\frac{2}{p}\right)
\end{align*}
is nonempty, so $\alpha'$ can indeed be chosen.

The proof of \Cref{pStep1b} relies on the same heat-kernel estimates as in Step~1.(a), together with the bounds \eqref{vbd} and \eqref{vMap-Lip}. We postpone it to the end of the proof of \Cref{ThmExistence}.

\medskip

\textit{Step 2: Enhanced regularity of $\tilde N=\tilde N[\bar{x}]\in L^\infty([0,T], W^{3,\infty}(\R))$.}

First, let us recall that in Step 3. our goal is to define the interface $\bar{y}$ of the level set through the ODE
\begin{align*}
\dot{\bar{y}}(t)
=
-\frac{\dt N(t,\bar{y}(t))}{\dx N(t,\bar{y}(t))},
\end{align*}
and we need the right-hand side to be Lipschitz continuous in the spatial variable. Since
\begin{align*} \dt N = D\dxx N-\r N,
\end{align*}
this requires control of $\dxx N \in W^{1,\infty}$ in space, which we can obtain through the regularity $\tilde N\in L^\infty([0,T], W^{3,\infty}(\R))$. Now, $\tilde{N}$ satisfies
\begin{align*}
\dt \tilde{N}-D\dzz \tilde{N}-\dot{\bar{x}}(t)\dz \tilde{N}=-  vU\tilde{N}.
\end{align*}
The right-hand side term is Lipschitz continuous, but from this it cannot be deduced in general that $\tilde{N}$ has $W^{3,\infty}$ regularity in space. However, in our case this result remains true, as we can single out the $C^1$-discontinuity of $U$ at $z=0$, then prove that this explicit contribution enjoys the endpoint $W^{3,\infty}$ regularity. Finally, we refer the reader to Remark \ref{rk-optimality-N} at the end of this proof, where we give an argument why any higher regularity is not to be expected.

From \textit{Step 1. (b)}, we have we have the representation:
\begin{align}
\label{repN}
\tilde N(t)=&e^{t D\dzz } N_{\rm in}+\int_0^t e^{(t-s) D\dzz }\left( \dot{\bar x}(s) \dz \tilde N(s)\right)\, ds \\
&-  \int_0^t e^{(t-s) D\dzz }\left(v(s)U\tilde N(s) \right)ds.\nonumber
\end{align}
The regularity of the two first terms is obtained by the standard estimates of the heat semigroup and the difficulty stems from the last term. In order to do so, we will decompose the term $\dz \left(v(s)U\tilde N (s) \right)$ as the sum of a $C^{0,\alpha}$ function and a discontinuous function. In fact for $s\in [0,T], z\in \R$:
\begin{align}
&\dz \left(vU \tilde N\right)(s,z)\\=&\underbrace{\dz \left(vU \tilde N\right)(s,z) - \llbracket \dz\left(vU \tilde N\right)(s) \rrbracket_{z=0}\mathbbm{1}_{z\geq0}}_{g(s,z):=} +\underbrace{\llbracket \dz\left(vU \tilde N\right)(s) \rrbracket_{z=0}}_{h(s):=}\mathbbm{1}_{z\geq0},
\label{C1DiscDecomp}
\end{align}
where $\llbracket f \rrbracket_{z_0} = \lim_{z\to z_0^+}f(z)-\lim_{z\to z_0^-}f(z)$. Here $\llbracket \dz\left(vU \tilde N\right)(s) \rrbracket_{z=0}$ is well-defined, since $v(s),\tilde N(s)\in C^{1,\alpha}(\R)$ and $U\in C^1(\R_+)\cap C^1(\R_-)$.

Let us now invoke the lemma that will enable us to conclude:
\begin{lemma}[Recovery of $W^{3,\infty}$ regularity for $\tilde N$]\label{lStep2}
Under the conditions of \Cref{ThmExistence}, we have that:
\begin{enumerate}[label=(\roman*)]
\item $e^{t D\dzz } N_{\rm in} \in L^\infty([0,T], W^{3,\infty}(\R))$.
\item $\int_0^t e^{(t-s) D\dzz }\left( \dot{\bar x}(s) \dz \tilde N(s)\right)\, ds \in L^\infty([0,T], W^{3,\infty}(\R))$ and in particular:
\begin{align}
\norm{\partial_{zzz}\int_0^t e^{(t-s) D\dzz } \dot{\bar x}(s) \dz \tilde N(s)ds}_\infty
\leq C T^{\frac{p\alpha'-2}{2p}} \norm{\dot{\bar{x}}}_p \norm{\tilde N}_{L^\infty([0,T],C^{2,\alpha'}(\R))}.
\end{align}
\item For $g$ given by \eqref{C1DiscDecomp}, $g\in L^\infty([0,T], C^{0,\hat\alpha}(\R))$, where $\hat{\alpha}=\min\{\alpha,\alpha'\}$, and as a consequence 
\begin{align*}
\int_0^t e^{(t-s) D\dzz }g(s)\, ds \in L^\infty([0,T], W^{2,\infty}(\R)).
\end{align*}
\item For $h$ given by \eqref{C1DiscDecomp}, $h\in L^\infty(0,T)$ and
\begin{align}
\norm{\dzz \int_0^t  e^{(t-s) D\dzz }h(s)\mathbbm{1}_{z\geq0}ds}_\infty\leq 
C \norm{h}_{\infty}.
\end{align}
\end{enumerate}
\end{lemma}

The proof of \Cref{lStep2} relies on the reverse Hölder heat estimate \eqref{dxxReverseHolderHeat}, a simple extension argument for piecewise Hölder functions, and an explicit heat-kernel computation for the discontinuous contribution $h(s)\mathbbm{1}_{z\geq0}$. We postpone it to the end of the proof of \Cref{ThmExistence} and now show how to conclude using this lemma.

As a direct consequence of Items~(iii) and~(iv) of \Cref{lStep2}, we have that
\begin{align*}
 \int_0^t e^{(t-s) D\dzz }\left(g(s)+h(s)\mathbbm{1}_{z\geq0} \right)ds 
 \in L^\infty([0,T], W^{2,\infty}(\R)).
\end{align*}
Using next the identity \eqref{C1DiscDecomp}, which we rewrite
\begin{align*}
\dz\left( vU\tilde{N}\right) = g+h \mathbbm{1}_{z\geq0},
\end{align*}
this shows that
\begin{align*}
 \int_0^t e^{(t-s) D\dzz }\left(v(s)U\tilde N(s) \right)ds
 \in L^\infty([0,T], W^{3,\infty}(\R)).
\end{align*}
This combined with \eqref{repN} and Items~(i) and~(ii) of \Cref{lStep2} implies the desired result:
\begin{align*}
\tilde N\in L^\infty([0,T], W^{3,\infty}(\R)).
\end{align*}

\medskip

\textit{Step 3: Definition of the map $\bar{x}\mapsto\bar{y}$.}

We now return to the initial frame $(t,x)$. The following lemma shows that we can define the curve $\bar y[\bar x]$ such that for $t\in [0,T]$,
\begin{align*}
N[\bar{x}](t,\bar y[\bar{x}](t))=N_\thresh,
\end{align*}
which defines the map $\bar{x}\mapsto\bar{y}$.
\begin{lemma}[Persistence of the threshold interface]
\label{lStep3}
Assume that $N_{\rm in}$ is nondecreasing and that $N_{\rm in}(\bar{x}_{\rm in})=N_\thresh$. Assume furthermore that there exist $\zeta>0$, $\underline m>0$ such that
\begin{align*}
\dx N_{\rm in}(x)\geq \underline m
\qquad \text{for all } x\in[\bar{x}_{\rm in}-\zeta,\bar{x}_{\rm in}+\zeta].
\end{align*}
Then, for $\bar{x}\in B(A)$, there exists $T>0$ such that the solution $N[\bar{x}]$ to \eqref{movingmodel:b}, considered in the stationary frame, satisfies for all $t\in[0,T]$,
\begin{align*}
N[\bar{x}](t,x)<N_\thresh \quad \text{for } x\leq \bar{x}_{\rm in}-\zeta,
\qquad
N[\bar{x}](t,x)>N_\thresh \quad \text{for } x\geq \bar{x}_{\rm in}+\zeta,
\end{align*}
and
\begin{align*}
\dx N[\bar{x}](t,x)\geq \frac{\underline m}{8}
\qquad \text{for all } x\in[\bar{x}_{\rm in}-\zeta,\bar{x}_{\rm in}+\zeta].
\end{align*}
In particular, for each $t\in[0,T]$, the equation
\begin{align*}
N[\bar{x}](t,x)=N_\thresh
\end{align*}
has a unique solution $\bar y(t)\in[\bar{x}_{\rm in}-\zeta,\bar{x}_{\rm in}+\zeta]$.
\end{lemma}
The proof is postponed to the end of the proof of \Cref{ThmExistence}.

We next show that $\bar y$ solves the expected ODE. Since
\begin{align*}
N=N[\bar{x}]\in L^\infty([0,T],W^{3,\infty}(\R)),
\qquad
\r=\r[\bar{x}]\in L^\infty([0,T],W^{1,\infty}(\R)),
\end{align*}
and
\begin{align*}
\dt N=D\dxx N-\r N,
\end{align*}
we have
\begin{align*}
\dt N\in L^\infty([0,T],W^{1,\infty}(\R)).
\end{align*}
Moreover, \Cref{lStep3} gives
\begin{align*}
\dx N(t,\bar y(t))\geq \frac{\underline m}{8}.
\end{align*}
Therefore,
\begin{align*}
\mathcal F(t,x):=-\frac{\dt N(t,x)}{\dx N(t,x)}
\end{align*}
is bounded and uniformly Lipschitz continuous in the space variable on $[0,T]\times[\bar{x}_{\rm in}-\zeta,\bar{x}_{\rm in}+\zeta]$. It follows from the Cauchy--Lipschitz theorem that $\bar y\in W^{1,\infty}([0,T])$ and satisfies
\begin{align*}
\left\{
\begin{array}{ll}
\dot{\bar y}(t)=\mathcal F(t,\bar y(t))\\
\bar y(0)=\bar x_{\rm in}.
\end{array}
\right.
\end{align*}
Conversely, this ODE preserves the relation $N(t,\bar y(t))=N_\thresh$, since this relation holds at $t=0$.

It remains to check that $\bar y$ belongs to the fixed-point ball $B(A)$. Using the lower bound on $\dx N$ and the equation for $N$, we obtain
\begin{align*}
|\dot{\bar y}(t)|
&\leq
\frac{8}{\underline m}
\left(
D\norm{\dxx N}_{\infty}
+\norm{\r}_{\infty}\norm{N}_{\infty}
\right).
\end{align*}
Together with Bounds \eqref{vbd}--\eqref{Nbd}, this yields
\begin{align*}
\norm{\dot{\bar y}}_{\infty}
\leq
C\norm{N_{\rm in}}_{W^{3,\infty}(\R)}
\norm{\nicefrac{\r_{\rm in}}{U}}_{C^{1,\alpha}(\R)},
\end{align*}
where we recall that we assume $\norm{N_{\rm in}}_{W^{3,\infty}(\R)}
, \norm{\nicefrac{\r_{\rm in}}{U}}_{C^{1,\alpha}(\R)}>1$. 
Thus, using $T<1$ and up to choosing
\begin{align*}
A > 2\left( |\bar{x}_{\rm in}| + C\norm{N_{\rm in}}_{W^{3,\infty}(\R)}
\norm{\nicefrac{\r_{\rm in}}{U}}_{C^{1,\alpha}(\R)}\right),
\end{align*}
we have that
\begin{align*}
\norm{\bar y[\bar{x}]}_{W^{1,p}((0,T))} \leq   \norm{{\bar{x}}}_\infty+ \norm{\dot{\bar{x}}}_\infty  \leq \left(|\bar{x}_{\rm in}| +T \norm{\dot{\bar{x}}}_\infty \right) + \norm{\dot{\bar{x}}}_\infty \leq A,
\end{align*}
and therefore the map $\bar x\mapsto\bar y$ maps $B(A)$ into itself.

\medskip

\textit{Step 4: Unique fixed point of the map $\bar x\in B(A)\mapsto \bar y\in B(A)$.}

Let $\bar x_1,\bar x_2\in B(A)$, and denote by $\bar y_1,\bar y_2\in B(A)$ the corresponding curves constructed in Step 3. We also write $N_i=N[\bar x_i]$, $\rho_i=\rho[\bar x_i]$, and
\begin{align}
\label{def.Fi}
\mathcal F_i(t,x):=-\frac{\dt N_i(t,x)}{\dx N_i(t,x)}.
\end{align}
Then
\begin{align*}
\dot{\bar y}_i(t)=\mathcal F_i(t,\bar y_i(t)).
\end{align*}

\begin{lemma}[Stability of the vector field $\mathcal{F}$]
\label{lStep4}
There exist $T>0$, $K>0$ and a function $\Theta$, with $\Theta(T)\to0$ as $T\to0$ such that:

For any $\bar x_1,\bar x_2\in B(A)$, let $\r_i,N_i,\mathcal{F}_i$ as defined just above for $i=1,2$,
\begin{align}
\label{lStep4-1}
\norm{\dx \mathcal F_i}_{L^\infty([0,T]\times[\bar{x}_{\rm in}-\zeta,\bar{x}_{\rm in}+\zeta])}\leq K,
\end{align}
and, for all $t\in[0,T]$ and all $x\in[\bar{x}_{\rm in}-\zeta,\bar{x}_{\rm in}+\zeta]$,
\begin{align}
\label{lStep4-2}
|\mathcal F_1(t,x)-\mathcal F_2(t,x)|
\leq
\Theta(T)\norm{\dot{\bar x}_1-\dot{\bar x}_2}_{p}.
\end{align}
\end{lemma}

The proof of this lemma is postponed to the end of the proof of \Cref{ThmExistence}. Using this lemma, we now show how $\bar{x} \in B(A)\mapsto \bar{y} \in B(A)$ is a contraction for $T>0$ small enough.

First, let us observe that by Step 3, $\bar y_i(t)\in[\bar{x}_{\rm in}-\zeta,\bar{x}_{\rm in}+\zeta]$. From this we obtain
\begin{align*}
&|\dot{\bar y}_1(t)-\dot{\bar y}_2(t)|\\
 =&
\left|\mathcal F_1(t,\bar y_1(t))-\mathcal F_2(t,\bar y_2(t))\right|  \\
\leq&
\left|\mathcal F_1(t,\bar y_1(t))-\mathcal F_1(t,\bar y_2(t))\right| +
\left|\mathcal F_1(t,\bar y_2(t))-\mathcal F_2(t,\bar y_2(t))\right| \\
\leq&
K\left|\bar y_1(t)-\bar y_2(t)\right| +
\Theta(T)\norm{\dot{\bar x}_1-\dot{\bar x}_2}_{p},
\end{align*}
where we used respectively \eqref{lStep4-1} and \eqref{lStep4-2}.
Since $\bar y_1(0)=\bar y_2(0)=\bar x_{\rm in}$, Grönwall's lemma gives
\begin{align*}
\norm{\bar y_1-\bar y_2}_{\infty}
\leq
\Theta(T)\frac{e^{KT}-1}{K}
\norm{\dot{\bar x}_1-\dot{\bar x}_2}_{p}.
\end{align*}
Bootstrapping the penultimate estimate, we obtain
\begin{align*}
\norm{\dot{\bar y}_1-\dot{\bar y}_2}_{\infty}
\leq \Theta(T)e^{KT}
\norm{\dot{\bar x}_1-\dot{\bar x}_2}_{p}.
\end{align*}
Moreover,
\begin{align*}
&\norm{\bar y_1-\bar y_2}_{W^{1,p}(0,T)}\\
\leq&
T^{\frac1p}
\norm{\bar y_1-\bar y_2}_{W^{1,\infty}(0,T)} \\
\leq & T^{\frac1p}\Theta(T) \left(\frac{e^{KT}-1}{K}+ e^{KT}\right)\norm{{\bar x}_1-{\bar x}_2}_{W^{1,p}(0,T)}.
\end{align*}
Thus, for $T>0$ small enough, the map
\begin{align*}
\bar x\in B(A)\mapsto \bar y\in B(A)
\end{align*}
is a contraction in $W^{1,p}(0,T)$ and it admits a unique fixed point, which concludes the proof of \Cref{ThmExistence}, up to the proofs of \Cref{pStep1a}, \Cref{pStep1b}, \Cref{lStep2}, \Cref{lStep3}, \Cref{lStep4}, which are now given below.


\begin{proof}[Proof of \Cref{pStep1a}. The solution map $\bar x\mapsto v{[\bar x]}$.]
\leavevmode\\

Fix $\bar x\in B(A)$ and set
\begin{align*}
X_T:=L^\infty([0,T],C^{1,\alpha}(\R)).
\end{align*}
Recall that, for $u\in X_T$, the map $F_{\bar x}$ is defined by
\begin{align}
\label{vMapbis} 
F_{\bar x}[u](t)
&=
e^{t\dzz}\left(\nicefrac{\rho_{\rm in}}{U}\right)
+
\int_0^t e^{(t-s)\dzz}
\left(
\beta(s)\dz u(s)
+
\gamma(s)u(s)
\right) ds,
\end{align}
where
\begin{align*}
\beta(t)=\dot{\bar{x}}(t)-\chi \mathbbm{1}_{z\leq 0} -2\chi \mathbbm{1}_{z>0},
\qquad
\gamma(t)=\chi \left(\chi +\frac{1}{\chi}- \dot{\bar{x}}(t) \right)\mathbbm{1}_{z>0}.
\end{align*}

We start by showing that $F_{\bar x}$ maps into $X_T$.

As $\nicefrac{\r_{\rm in}}{U} \in C^{1,\alpha}(\R)$, we have that $e^{t \dzz } \left(\nicefrac{\r_{\rm in}}{U}\right) \in X_T$.

Now, we treat the second term in \eqref{vMapbis} and show that it is in $X_T$. Using Bound \eqref{dxHolderHeat}, we have 
\begin{align*}
&\norm{\int_0^t e^{(t-s) \dzz }\left( \beta(s)\dz u(s) \right)ds}_{C^{1,\alpha}(\R)}\\
&\leq \int_0^t \norm{e^{(t-s) \dzz }\left( \beta(s)\dz u(s) \right)}_{C^{1,\alpha}(\R)}ds  \\
&\leq C\int_0^t \left(1+(t-s)^{-\frac{1+\alpha}{2}}\right)\norm{ \beta(s)\dz u(s) }_{\infty}ds\\
&\leq C\norm{ u }_{L^\infty\left([0,T],W^{1}(\R)\right)}\int_0^t \left(1+(t-s)^{-\frac{1+\alpha}{2}}\right)\left(1+\left|\dot{\bar{x}}(s)\right|\right)ds\\
&\leq C\norm{ u }_{X_T}\int_0^t (t-s)^{-\frac{1+\alpha}{2}}\left(1+\left|\dot{\bar{x}}(s)\right|\right)ds\\
&\leq C\norm{ u }_{X_T}\norm{ (t-\cdot)^{-\frac{1+\alpha}{2}}(1+|\dot{\bar{x}}(\cdot)|)}_{L^1([0,t])}\\
&\leq C\norm{ u }_{X_T}\norm{(t-\cdot)^{-\frac{1+\alpha}{2}}}_{L^{\frac{p}{p-1}}([0,t])}\norm{(1+|\dot{\bar{x}}(\cdot)|)}_{L^p([0,t])}\\
&\leq C\norm{ u }_{X_T}t^{\frac{(1-\alpha)p-2}{2p}}(1+\norm{\dot{\bar{x}}}_{p})\\
&\leq C\norm{ u }_{X_T}T^{\frac{(1-\alpha)p-2}{2p}}A,
\end{align*}
where we used the convention that $T<1$ and thus $(t-s)^{-\frac{1+\alpha}{2}}>1$, Hölder's inequality, $p>\frac{2}{1-\alpha}$ in order to guarantee the integrability of $s\mapsto (t-s) ^{-\frac{(1+\alpha)p}{2(p-1)}}$ and the convention that $A>1$. 

The remaining bound is obtained in a similar manner and we obtain
\begin{align*}
\norm{\int_0^t e^{(t-s) \dzz }\left( \gamma(s) u(s) \right)ds}_{X_T}\leq C\norm{ u }_{L^\infty\left([0,T]\times\R\right)}T^{\frac{(1-\alpha)p-2}{2p}}A.
\end{align*}

Therefore, for every $u\in X_T$, we have $F_{\bar x}[u]\in X_T$. Hence $F_{\bar x}$ maps $X_T$ into itself. 

We now show that $F_{\bar x}$ is a contraction on $X_T$. Since $F_{\bar x}$ is affine, the homogeneous part satisfies, by the preceding estimates,
\begin{align*}
\norm{F_{\bar x}[u_1]-F_{\bar x}[u_2]}_{X_T}
\leq
CA T^{\frac{(1-\alpha)p-2}{2p}}
\norm{u_1-u_2}_{X_T}.
\end{align*}
Since $p>\frac{2}{1-\alpha}$, the exponent is positive. Thus, up to reducing $T>0$, we may assume that
\begin{align*}
CA T^{\frac{(1-\alpha)p-2}{2p}}\leq \frac12.
\end{align*}
Hence $F_{\bar x}$ is a contraction on $X_T$.

By Banach's fixed-point theorem, $F_{\bar x}$ admits a unique fixed point $v[\bar x]\in X_T$. Moreover,
\begin{align*}
\norm{v[\bar x]}_{X_T}
&=
\norm{F_{\bar x}[v[\bar x]]}_{X_T} \\
&\leq
\norm{F_{\bar x}[v[\bar x]]-F_{\bar x}[0]}_{X_T}
+
\norm{F_{\bar x}[0]}_{X_T} \\
&\leq
\frac12\norm{v[\bar x]}_{X_T}
+
\norm{e^{t\dzz}v_{\rm in}}_{X_T}.
\end{align*}
Therefore,
\begin{align*}
\norm{v[\bar x]}_{X_T}
\leq
2\norm{v_{\rm in}}_{C^{1,\alpha}(\R)}.
\end{align*}
This proves items (i) and (ii).

It remains to prove item (iii), namely the Lipschitz dependence of the fixed point on $\bar x$. Let $\bar x_1,\bar x_2\in B(A)$, and set
\begin{align*}
v_i:=v[\bar x_i],
\qquad
F_i:=F_{\bar x_i},
\qquad
w:=v_1-v_2.
\end{align*}
Since $v_i$ is the fixed point of $F_i$, we have $v_i=F_i[v_i]$. Therefore,
\begin{align*}
w(t)
&=F_1[v_1](t)-F_2[v_2](t) \\
&=
\int_0^t e^{(t-s)\dzz}
\left(
\beta_1(s)\dz w(s)
+
\gamma_1(s)w(s)
\right)ds \\
&\quad+
\int_0^t e^{(t-s)\dzz}
\left(
(\beta_1(s)-\beta_2(s))\dz v_2(s)
+
(\gamma_1(s)-\gamma_2(s))v_2(s)
\right)ds.
\end{align*}
Since
\begin{align*}
F_1[w](t)
&=
e^{t\dzz}\left(\nicefrac{\r_{\rm in}}{U}\right)
+
\int_0^t e^{(t-s)\dzz}
\left(
\beta_1(s)\dz w(s)
+
\gamma_1(s)w(s)
\right)ds,
\end{align*}
we may rewrite this as
\begin{align*}
w(t)
&=
F_1[w](t)-F_1[0](t) \\
&\quad+
\int_0^t e^{(t-s)\dzz}
\left(
(\beta_1(s)-\beta_2(s))\dz v_2(s)
+
(\gamma_1(s)-\gamma_2(s))v_2(s)
\right)ds.
\end{align*}
By the choice of $T$ in the first part of the proof, $F_1$ has Lipschitz constant at most $\nicefrac12$. Thus
\begin{align}
\label{vLipAbsorb}
\norm{F_1[w]-F_1[0]}_{X_T}
\leq
\frac12\norm{w}_{X_T}.
\end{align}

It remains to estimate the second term. By the definitions of $\beta_i$ and $\gamma_i$, we have
\begin{align*}
\beta_1(s)-\beta_2(s)
=
\dot{\bar x}_1(s)-\dot{\bar x}_2(s) \quad \text{and} \quad
\gamma_1(s)-\gamma_2(s)
=
\chi\left(\dot{\bar x}_2(s)-\dot{\bar x}_1(s)\right)\mathbbm{1}_{z>0}.
\end{align*}
Therefore, for $t\in[0,T]$,
\begin{align*}
&\norm{
\int_0^t e^{(t-s)\dzz}
\left(
(\beta_1(s)-\beta_2(s))\dz v_2(s)
+
(\gamma_1(s)-\gamma_2(s))v_2(s)
\right)ds
}_{C^{1,\alpha}(\R)}
\\
&=
\norm{
\int_0^t e^{(t-s)\dzz}
\left(\dot{\bar x}_1(s)-\dot{\bar x}_2(s)\right)
\left(
\dz v_2(s)+\chi\mathbbm{1}_{z\geq0}v_2(s)
\right)ds
}_{C^{1,\alpha}(\R)}.
\end{align*}
Using the heat estimate \eqref{dxHolderHeat}, we obtain
\begin{align*}
&\norm{
\int_0^t e^{(t-s)\dzz}
\left(\dot{\bar x}_1(s)-\dot{\bar x}_2(s)\right)
\left(
\dz v_2(s)+\chi\mathbbm{1}_{z\geq0}v_2(s)
\right)ds
}_{C^{1,\alpha}(\R)}
\\
&\leq
C
\int_0^t
(t-s)^{-\frac{1+\alpha}{2}}
|\dot{\bar x}_1(s)-\dot{\bar x}_2(s)|
\norm{\dz v_2(s)+\chi\mathbbm{1}_{z\geq0}v_2(s)}_{\infty}
ds.
\end{align*}
Since $v_2\in L^\infty([0,T],C^{1,\alpha}(\R))$, we have
\begin{align*}
\norm{\dz v_2(s)+\chi\mathbbm{1}_{z\geq0}v_2(s)}_{\infty}
\leq
C\norm{v_2(s)}_{C^{1,\alpha}(\R)}.
\end{align*}
Using the uniform bound \eqref{vbd}, this gives
\begin{align*}
&\norm{
\int_0^t e^{(t-s)\dzz}
\left(
(\beta_1(s)-\beta_2(s))\dz v_2(s)
+
(\gamma_1(s)-\gamma_2(s))v_2(s)
\right)ds
}_{C^{1,\alpha}(\R)}
\\
&\leq
C
\norm{\nicefrac{\r_{\rm in}}{U}}_{C^{1,\alpha}(\R)}
\int_0^t
(t-s)^{-\frac{1+\alpha}{2}}
|\dot{\bar x}_1(s)-\dot{\bar x}_2(s)|
ds.
\end{align*}
Applying Hölder's inequality with exponents $p$ and $\frac{p}{p-1}$, we obtain
\begin{align*}
&\int_0^t
(t-s)^{-\frac{1+\alpha}{2}}
|\dot{\bar x}_1(s)-\dot{\bar x}_2(s)|
ds  \\
&\leq
\left(
\int_0^t
(t-s)^{-\frac{(1+\alpha)p}{2(p-1)}}ds
\right)^{\frac{p-1}{p}}
\norm{\dot{\bar x}_1-\dot{\bar x}_2}_{p} \\
&\leq
C t^{\frac{p(1-\alpha)-2}{2p}}
\norm{\dot{\bar x}_1-\dot{\bar x}_2}_{p}.
\end{align*}
Hence, for all $t\in[0,T]$,
\begin{align*}
&\norm{
\int_0^t e^{(t-s)\dzz}
\left(
(\beta_1(s)-\beta_2(s))\dz v_2(s)
+
(\gamma_1(s)-\gamma_2(s))v_2(s)
\right)ds
}_{C^{1,\alpha}(\R)}
\\
&\leq
C
\norm{\nicefrac{\r_{\rm in}}{U}}_{C^{1,\alpha}(\R)}
T^{\frac{p(1-\alpha)-2}{2p}}
\norm{\dot{\bar x}_1-\dot{\bar x}_2}_{p}.
\end{align*}
Taking the supremum over $t\in[0,T]$, and combining this estimate with \eqref{vLipAbsorb}, we get
\begin{align*}
\norm{w}_{X_T}
&\leq
\frac12\norm{w}_{X_T}
+
C
\norm{\nicefrac{\r_{\rm in}}{U}}_{C^{1,\alpha}(\R)}
T^{\frac{p(1-\alpha)-2}{2p}}
\norm{\dot{\bar x}_1-\dot{\bar x}_2}_{p}.
\end{align*}
Absorbing the first term on the right-hand side yields
\begin{align*}
\norm{v[\bar x_1]-v[\bar x_2]}_{X_T}
\leq
C
\norm{\nicefrac{\r_{\rm in}}{U}}_{C^{1,\alpha}(\R)}
T^{\frac{p(1-\alpha)-2}{2p}}
\norm{\dot{\bar x}_1-\dot{\bar x}_2}_{p}.
\end{align*}
Thus item (iii) holds with
\begin{align*}
\Theta_\rho(T)
:=
C
\norm{\nicefrac{\r_{\rm in}}{U}}_{C^{1,\alpha}(\R)}
T^{\frac{p(1-\alpha)-2}{2p}}.
\end{align*}
Since $p>\frac{2}{1-\alpha}$, we have $\Theta_\rho(T)\to0$ as $T\to0$. This proves item (iii) and concludes the proof of \Cref{pStep1a}.
\end{proof}


\begin{proof}[Proof of \Cref{pStep1b}. The solution map $\bar x\mapsto \tilde N{[\bar x]}$.]
\leavevmode\\
Fix $\bar x\in B(A)$ and set
\begin{align*}
Y_T:=L^\infty([0,T],C^{2,\alpha'}(\R)).
\end{align*}
Recall that, for $n\in Y_T$, the map $G_{\bar x}$ is defined by
\begin{align}
\label{NMapbis}
G_{\bar{x}}[n](t)= e^{t D\dzz } N_{\rm in}+\int_0^t e^{(t-s)D \dzz }\left( \dot{\bar x}(s) \dz n(s)-  v[\bar{x}](s)U n(s) \right)ds.
\end{align}
We prove that, for $T>0$ small enough, $G_{\bar x}$ maps $Y_T$ into itself and is a contraction on $Y_T$. Since the argument is the same as in the proof of \Cref{pStep1a}, we only detail the two terms that need to be estimated.

First, using the heat estimate for the Hölder seminorm of the first derivative \eqref{dxHolderHeat}, we get
\begin{align*}
&\left[\dzz \int_0^t e^{D(t-s)\dzz}  v(s)Un(s)ds\right]_{\alpha'}\\
=&\left[\dz \int_0^t e^{D(t-s)\dzz}  \dz(v(s)Un(s))ds\right]_{\alpha'}\\
\leq& C\int_0^t (t-s)^{-\frac{1+\alpha'}{2}}\norm{\dz\left(v(s)Un(s) \right)}_\infty ds
\end{align*}
Since $U\in W^{1,\infty}(\R)$, we have
\begin{align*}
\norm{\dz(v(s)Un(s))}_{\infty}
\leq
C\norm{v(s)}_{W^{1,\infty}(\R)}
\norm{n(s)}_{W^{1,\infty}(\R)}.
\end{align*}
Thus, we obtain
\begin{align*}
&\left[\dzz \int_0^t e^{D(t-s)\dzz}  v(s)Un(s)ds\right]_{\alpha'}\\
\leq &CT^{\frac{1-\alpha'}{2}}\norm{v}_{X_T}\norm{n}_{Y_T}
\\
\leq &CT^{\frac{1-\alpha'}{2}}\norm{\nicefrac{\r_{\rm in}}{U}}_{C^{1,\alpha}(\R)}\norm{n}_{Y_T}.
\end{align*}

Second, for the transport term, we obtain similarly 
\begin{align*}
&\left[\dzz \int_0^t e^{D(t-s)\dzz}\dot{\bar{x}}(s)\dz n(s)ds\right]_{\alpha'}\\
&\leq C\int_0^t (t-s)^{-\frac{1+\alpha'}{2}}\left|\dot{\bar{x}}(s)\right|\norm{\dzz n(s) }_\infty ds\\
&\leq C\norm{n}_{Y_T}\int_0^t
(t-s)^{-\frac{1+\alpha'}{2}}
|\dot{\bar x}(s)|
ds.
\end{align*}
Applying Hölder's inequality with exponents $p$ and $\nicefrac{p}{p-1}$ yields
\begin{align*}
\int_0^t
(t-s)^{-\frac{1+\alpha'}{2}}
|\dot{\bar x}(s)|
ds
\leq
C
t^{\frac{p(1-\alpha')-2}{2p}}
\norm{\dot{\bar x}}_{p} \leq CT^{\frac{p(1-\alpha')-2}{2p}}A.
\end{align*}
The remaining components of the $C^{2,\alpha'}$ norm are estimated in the same way, and are easier. Therefore $G_{\bar x}$ maps $Y_T$ into itself. Moreover, the same estimates applied to $G_{\bar x}[n_1]-G_{\bar x}[n_2]$ show that
\begin{align*}
\norm{G_{\bar x}[n_1]-G_{\bar x}[n_2]}_{Y_T}
\leq
C
\left(
T^{\frac{1-\alpha'}{2}}
+
A T^{\frac{p(1-\alpha')-2}{2p}}
\right)
\norm{n_1-n_2}_{Y_T}.
\end{align*}
Since $\alpha'<1-\nicefrac2p$, the exponent $
\nicefrac{p(1-\alpha')-2}{2p}$
is positive. Hence, up to reducing $T>0$, we may assume that the Lipschitz constant of $G_{\bar x}$ is at most $\nicefrac12$. By Banach's fixed-point theorem, $G_{\bar x}$ admits a unique fixed point, denoted by $\tilde N[\bar x]\in Y_T$. This fixed point is the unique solution to \eqref{movingmodel:b} and as before, since the Lipschitz constant of $G_{\bar x}$ is at most $\nicefrac12$, we also have that
\begin{align*}
\norm{\tilde N[\bar x]}_{Y_T}
\leq
2\norm{N_{\rm in}}_{C^{2,\alpha'}(\R)}
\leq
2\norm{N_{\rm in}}_{W^{3,\infty}(\R)}.
\end{align*}
This proves items (i) and (ii).

It remains to prove item (iii), namely the Lipschitz dependence of the fixed point on $\bar x$. Let $\bar x_1,\bar x_2\in B(A)$, and set
\begin{align*}
\tilde{N}_i:=\tilde{N}[\bar x_i],
\qquad
G_i:=G_{\bar x_i},
\qquad
P:=\tilde{N}_1-\tilde{N}_2,
\end{align*}
and recall from the proof of \Cref{pStep1a} that $w=v[\bar{x}_1]-v[\bar{x}_2]$. We may rewrite $P$ as
\begin{align*}
P(t)
=&
G_1[P](t)-G_1[0](t) \\
&+
\int_0^t e^{(t-s)D\dzz}
\left(
(\dot{\bar x}_1(s)-\dot{\bar x}_2(s))\dz \tilde N_2(s)
-
w(s)U\tilde N_2(s)
\right) ds.
\end{align*}
By the same arguments as before, we have the following bounds:
\begin{align*}
\norm{G_1[P]-G_1[0]}_{Y_T}&\leq
\frac{1}{2}\norm{P}_{Y_T}\\
\norm{\int_0^t e^{(t-s)D\dzz} (\dot{\bar{x_1}} - \dot{\bar{x_2}})\dz \tilde N_2 (s)ds}_{Y_T}&\leq C T^{\frac{p(1-\alpha')-2}{2p}}\norm{N_{\rm in}}_{W^{3,\infty}(\R)}\norm{\dot{\bar{x_2}}-\dot{\bar{x_1}}}_p\\
\norm{\int_0^t e^{(t-s)D\dzz}w(s)U\tilde N_2(s) ds}_{Y_T} &\leq C T^{\frac{1-\alpha'}{2}}\norm{w}_{X_T}\norm{N_{\rm in}}_{W^{3,\infty}(\R)},
\end{align*}
where, as in the proof of \Cref{pStep1a}, $X_T:=L^\infty([0,T],C^{1,\alpha}(\R))$.
By recalling \eqref{vMap-Lip} from \Cref{pStep1a}, we have
\begin{align*}
\norm{w}_{X_T} \leq \Theta_\rho(T) \norm{\dot{\bar{x}}_1-\dot{\bar{x}}_2}_p.
\end{align*}
Hence, the same argument as before gives
\begin{align*}
\norm{P}_{Y_T}\leq C\left( T^{\frac{p(1-\alpha')-2}{2p}}+T^{\frac{1-\alpha'}{2}} \Theta_\rho(T)\right)\norm{N_{\rm in}}_{W^{3,\infty}(\R)}\norm{\dot{\bar{x}}_1-\dot{\bar{x}}_2}_p.
\end{align*}
Thus item (iii) holds with
\begin{align*}
\Theta_N(T)
:=
C\left( T^{\frac{p(1-\alpha')-2}{2p}}+T^{\frac{1-\alpha'}{2}} \Theta_\rho(T)\right)\norm{N_{\rm in}}_{W^{3,\infty}(\R)}.
\end{align*}
Since $\alpha'<1-\frac2p$, the exponent $\nicefrac{p(1-\alpha')-2}{2p}$ is positive. Together with $\Theta_\rho(T)\to0$, this implies $\Theta_N(T)\to0$ as $T\to0$. This proves item (iii) and concludes the proof of \Cref{pStep1b}.
\end{proof}


 \begin{proof}[Proof of \Cref{lStep2}. Recovery of $W^{3,\infty}$ regularity for $\tilde N$.]
  \leavevmode\\

Item (i) is a direct consequence of the assumption that $N_{\rm in}\in W^{3,\infty}(\R)$.

Item (ii) is a consequence of the following computations, where in particular we use Bound \eqref{dxxReverseHolderHeat} and the Hölder inequality:
\begin{align*}
&\norm{\partial_{zzz}\int_0^t e^{(t-s) D\dzz } \dot{\bar x}(s) \dz \tilde N(s)ds}_\infty\\
  \leq &\int_0^t \norm{ \dzz \left(e^{(t-s)D \dzz } \dot{\bar x}(s) \dzz \tilde{N}(s)\right)}_\infty\, ds\\
  \leq &C\int_0^t \left( 1+(t-s)^{\frac{\alpha'}{2}-1} \right) \norm{  \dot{\bar x}(s) \dzz \tilde{N}(s)}_{C^{0,\alpha'}(\R)}\, ds\\
  \leq &C\norm{\tilde N}_{L^\infty([0,T],C^{2,\alpha'}(\R))}\int_0^t  (t-s)^{\frac{\alpha'}{2}-1}  \left|\dot{\bar{x}}(s) \right|\, ds\\
  \leq &C\norm{\tilde N}_{L^\infty([0,T],C^{2,\alpha'}(\R))}\left(
\int_0^t
(t-s)^{\left(\frac{\alpha'}2-1\right)\frac{p}{p-1}}
ds
\right)^{\frac{p-1}{p}}
\norm{\dot{\bar x}}_{p}\\
\leq & C\norm{\tilde N}_{L^\infty([0,T],C^{2,\alpha'}(\R))} T^{\frac{p\alpha'-2}{2p}}  \norm{\dot{\bar{x}}}_p,
\end{align*}
where we used $T<1$, so that $1+(t-s)^{\frac{\alpha'}{2}-1} \leq C(t-s)^{\frac{\alpha'}{2}-1}$, and the integrability of $s\mapsto (t-s)^{\left(\frac{\alpha'}2-1\right)\frac{p}{p-1}} $, which is a consequence of the condition $\alpha'>\frac{2}{p}$.

Item (iii) is a consequence of the decomposition \eqref{C1DiscDecomp}, which we recall
\begin{align*}
&f(s,z):=\dz \left(vU\tilde{N} \right)(s,z),\\
&g(s,z):=f(s,z)-\llbracket f(s) \rrbracket_{z=0}\mathbbm{1}_{z\geq0}, \\
& h(s):= \llbracket f(s) \rrbracket_{z=0},
\end{align*}
and the following Claim:

\smallskip
\noindent\emph{Claim.}
Let $f\in L^\infty([0,T],C^{0,\alpha}(\R_+))\cap L^\infty([0,T],C^{0,\alpha}(\R_-))$, where we understand $C^{0,\alpha}(\R_\pm)$ as a normed space, equipped with the norm $\norm{\cdot}_\infty +[\cdot]_{\alpha;\R_\pm}$. Then, we have that $g:=f-\llbracket f\rrbracket_{z=0}\mathbbm{1}_{z\geq 0} \in L^\infty([0,T],C^{0,\alpha}(\R))$ and:
\begin{align*}
\norm{g}_{L^\infty([0,T],C^{0,\alpha}(\R))}\leq 2^{1-\alpha} \max\left\{\sup_{t\in [0,T]}[f(t)]_{\alpha;\R_+},\sup_{t\in [0,T]}[f(t)]_{\alpha;\R_-}\right\}+3\norm{f}_{\infty}.
\end{align*}

Since  $f=\dz\left(vU\tilde{N}\right)\in L^\infty([0,T],C^{0,\hat\alpha}(\R_+))\cap L^\infty([0,T],C^{0,\hat\alpha}(\R_-))$, where $\hat{\alpha}:=\min(\alpha,\alpha')$, we then get that $g\in L^\infty([0,T],C^{0,\hat\alpha}(\R))$. Using the same arguments as before, we then immediately get that
\begin{align*}
\int_0^t e^{(t-s) D\dzz }g(s)\, ds \in L^\infty([0,T], W^{2,\infty}(\R)).
\end{align*}
This establishes Item (iii) up to the proof of the Claim:\\

\noindent\emph{Proof of the claim.}
By construction we have that for $t\in[0,T]$, $g(t)\in C^0(\R)$ and $\norm{g}_{\infty}\leq 3 \norm{f}_{\infty}$, since $ \left|\llbracket f\rrbracket_{z=0} \right|\leq 2 \norm{f}_\infty$. Let $t\in[0,T]$ and $x,y\in \R$ and suppose that $x<0<y$:
\begin{align}
\label{HolderLemmaAux}
&\frac{|g(t,y)-g(t,x)|}{|y-x|^\alpha}\\ \nonumber
=& \frac{\left|\left(f(t,y)-\llbracket f(t)\rrbracket_{z=0}\right)-f(t,x)\right|}{|y-x|^\alpha} \\  \nonumber
\leq& \frac{|f(t,y)-f(t,0^+)|}{|y-x|^\alpha} + \frac{|f(t,0^-)-f(t,x)|}{|y-x|^\alpha} \\
\nonumber
\leq& 2\left(\frac{1}{2}\cdot\frac{|y|^\alpha}{|y-x|^\alpha}[f(t)]_{\alpha;\R_+} + \frac{1}{2}\cdot\frac{|x|^\alpha}{|y-x|^\alpha}[f(t)]_{\alpha;\R_-}\right)\\
\nonumber
\leq& 2\left(\frac{|y|+|x|}{2|y-x|} \right)^\alpha \max\left\{\sup_{t\in [0,T]}[f(t)]_{\alpha;\R_+},\sup_{t\in [0,T]}[f(t)]_{\alpha;\R_-}\right\}\\
\nonumber
=& 2^{1-\alpha} \max\left\{\sup_{t\in [0,T]}[f(t)]_{\alpha;\R_+},\sup_{t\in [0,T]}[f(t)]_{\alpha;\R_-}\right\},
\end{align}
where we have used the concavity of $z\mapsto z^\alpha$ and the fact that $x<0<y$.\\
If $x,y<0$ (resp. $x,y>0$), then the left-hand side of \eqref{HolderLemmaAux} is simply bounded by $ [f(t)]_{\alpha;\R_-}$ (resp. $[f(t)]_{\alpha;\R_+} $).\\

It remains to prove Item (iv). Clearly $h =\llbracket f \rrbracket_{z=0} \in L^\infty(0,T)$, as $f=\dz\left(vU\tilde{N}\right) \in L^\infty([0,T],C^{0,\tilde\alpha}(\R_+))\cap L^\infty([0,T],C^{0,\tilde\alpha}(\R_-))$. Then, the bound is treated through explicit computations:
\begin{align*}
&\int_0^t \dzz e^{(t-s) D\dzz }h(s)\mathbbm{1}_{z\geq0}ds\\
=& \int_0^t \dz e^{(t-s)D \dzz }h(s)\delta_{0}ds\\
=& -\int_0^t \frac{z h(s)}{4\pi^\frac{1}{2}D^\frac{3}{2}(t-s)^{\frac32}}e^{-\frac{z^2}{4D(t-s)}}ds\\
=& -\sgn(z)\int_{\frac{|z|}{\sqrt{t}}}^\infty \frac{h\left(t-\frac{Dz^2}{u^2}\right)}{2\pi^\frac{1}{2}D}e^{-u^2}du\hspace{1cm} \text{, where }u=\frac{|z|}{\sqrt{D (t-s)}}.
\end{align*}
Hence, by the integrability of $e^{-u^2}$:
\begin{align*}
\norm{\int_0^t \dzz e^{(t-s) D\dzz }h(s)\mathbbm{1}_{z\geq0}ds}_\infty\leq C\norm{h}_{\infty},
\end{align*}
which concludes the proof of \Cref{lStep2}.
\end{proof}

\begin{proof}[Proof of \Cref{lStep3}. Persistence of the threshold interface.]
\leavevmode\\
Let $\bar x\in B(A)$ and write $\rho=\rho[\bar x]$ and $N=N[\bar x]$ in the stationary frame $(t,x)$, where they satisfy System (\eqref{ApproxSys:main}). Then, using the representation formula,
\begin{align*}
N(t)
=
e^{tD\dxx}N_{\rm in}
-
\int_0^t e^{(t-s)D\dxx}\rho(s)N(s)ds.
\end{align*}
By the bounds obtained in Steps 1 and 2, there exists a constant $C$, independent of $\bar x$ and $T$, such that
\begin{align*}
\norm{N}_{L^\infty([0,T]\times\R)}
+
\norm{\rho}_{L^\infty([0,T]\times\R)}
\leq C.
\end{align*}

By assumption, we have that $\dx N_{\rm in}\geq \underline{m}$ on the interval $[ \bar{x}_{\rm in}- \zeta,\bar{x}_{\rm in}+\zeta ]$. Therefore by setting $\e=\zeta \underline{m}$ and the global monotonicity of $N_{\rm in}$,
\begin{align*}
\text{for }x<\bar{x}_{\rm in}-\zeta, N_{\rm in}(x)\leq N_\thresh-\e, \quad \text{ and }\quad\text{ for }
x>\bar{x}_{\rm in}+\zeta, N_{\rm in}(x)\geq N_\thresh+\e.
\end{align*}

We first prove the persistence of the inequalities outside the interval. Since $N_{\rm in}\in W^{3,\infty}(\R)$, it is uniformly continuous, and therefore
\begin{align*}
\norm{e^{tD\dxx}N_{\rm in}-N_{\rm in}}_{\infty}\to0
\qquad \text{as } t\to0.
\end{align*}
Thus, up to reducing $T>0$, we may assume that
\begin{align*}
\norm{e^{tD\dxx}N_{\rm in}-N_{\rm in}}_{\infty}
\leq \frac{\e}{4}
\qquad \text{for all } t\in[0,T].
\end{align*}
Moreover, using the uniform bounds on $\rho$ and $N$, we may also assume that
\begin{align*}
\left\|
\int_0^t e^{(t-s)D\dxx}\rho(s)N(s)ds
\right\|_{\infty}
\leq
\frac{\e}{4}
\qquad \text{for all } t\in[0,T].
\end{align*}
Therefore, for $x\leq \bar x_{\rm in}-\zeta$,
\begin{align*}
N(t,x)
&\leq
N_{\rm in}(x)
+
\frac{\e}{4}
+
\frac{\e}{4}
\leq
N_\thresh-\frac{\e}{2}
<
N_\thresh.
\end{align*}
Similarly, for $x\geq \bar x_{\rm in}+\zeta$,
\begin{align*}
N(t,x)
&\geq
N_{\rm in}(x)
-
\frac{\e}{4}
-
\frac{\e}{4}
\geq
N_\thresh+\frac{\e}{2}
>
N_\thresh.
\end{align*}

We now show that there exists $T>0$, such that for $t\in [0,T], x\in [\bar{x}_{\rm in}-\zeta,\bar{x}_{\rm in}+\zeta]$, we have that $\dx N(t,x) \geq \frac{\underline{m}}{8}$. 
The reasoning is again similar. On the one hand:
\begin{align*}
 \norm{\dx \left( -\int_0^t e^{(t-s)D\dxx}   \r(s)N(s)ds \right) }_{\infty}
 &\leq \int_0^t \frac{C}{\sqrt{t-s}} \norm{\r(s)}_\infty\norm{N(s)}_\infty ds.
\end{align*}
As before we can choose $T>0$ such that the right handside becomes smaller than $\frac{\underline{m}}{8}$ (here the constant $C$ does not depend on $\underline{m}$).\\
On the other hand, by choosing $T>0$ small enough such that for every $t\in[0,T]$:
\begin{align*}
\frac{1}{\sqrt{4\pi D t}}\int_0^{\zeta}e^{-\frac{y^2}{4Dt}}dy\geq \frac{1}{4}.
\end{align*}
In that fashion for $x\in [\bar{x}_{\rm in}-\zeta,\bar{x}_{\rm in}+\zeta]$:
\begin{align*}
\left(e^{tD\dxx}\dx N_{\rm in}\right)(x)&= \frac{1}{\sqrt{4\pi D t}} \int_\R  e^{-\frac{(x-y)^2}{4Dt}}\dx N_{\rm in}(y)dy\\
 & \geq \frac{1}{\sqrt{4\pi D t}} \int_{\bar{x}_{\rm in}-\zeta}^{\bar{x}_{\rm in}+\zeta}  e^{-\frac{(x-y)^2}{4Dt}}\dx N_{\rm in}(y)dy\\
 & \geq \frac{\underline{m}}{\sqrt{4\pi D t}} \int_{\bar{x}_{\rm in}-\zeta}^{\bar{x}_{\rm in}+\zeta}  e^{-\frac{(x-y)^2}{4Dt}}dy\\
 & \geq \frac{\underline{m}}{\sqrt{4\pi D t}} \int_{-\zeta+\bar{x}_{\rm in}-x}^{\zeta+\bar{x}_{\rm in}-x}  e^{-\frac{y^2}{4Dt}} dy\\
  &\geq \frac{\underline{m}}{4},
\end{align*}
by noticing that either $\left[0, \zeta\right]\subset \left[ -\zeta+\bar{x}_{\rm in}-x,\zeta+\bar{x}_{\rm in}-x\right]$, when $x\leq \bar{x}_{\rm in}$, or alternatively that $\left[- \zeta,0\right]\subset
\left[-\zeta+\bar{x}_{\rm in}-x,\zeta+\bar{x}_{\rm in}-x\right]$, when $x\geq \bar{x}_{\rm in}$.
From this, we conclude that:
\begin{align*}
\dx N(t,x) = \left(e^{tD\dxx}\dx N_{\rm in}\right)(x)- \left[
\dx\int_0^t e^{(t-s)D\dxx}\rho(s)N(s)ds
\right](x) \geq \frac{\underline{m}}{4} - \frac{\underline{m}}{8} =\frac{\underline{m}}{8}.
\end{align*}

Finally, for every $t\in [0,T]$, we therefore have that for $x\leq \bar{x}_{\rm in}-\zeta$, $N(t,x)<N_\thresh$ and for $x\geq \bar{x}_{\rm in}+\zeta$, $N(t,x)>N_\thresh$. Moreover, $N(t,\cdot)$ is increasing on the interval $[\bar{x}_{\rm in}-\zeta,\bar{x}_{\rm in}+\zeta]$ and thus the equation $N(t,\cdot)=N_\thresh$ has a unique solution, denoted $\bar{y}(t)=\bar{y}[\bar{x}](t)\in [\bar{x}_{\rm in}-\zeta,\bar{x}_{\rm in}+\zeta]$.

\end{proof}

\begin{proof}[Proof of \Cref{lStep4}. Stability of the vector field $\mathcal F$.]
\leavevmode\\
Let
\begin{align*}
I:=[\bar{x}_{\rm in}-\zeta,\bar{x}_{\rm in}+\zeta].
\end{align*}
By \Cref{lStep3}, for $i=1,2$,
\begin{align*}
\dx N_i(t,x)\geq \frac{\underline m}{8}
\qquad
\text{for all } (t,x)\in[0,T]\times I.
\end{align*}

We first prove \eqref{lStep4-1}. Since
\begin{align*}
\mathcal F_i(t,x)
=
-\frac{\dt N_i(t,x)}{\dx N_i(t,x)},
\end{align*}
we have
\begin{align*}
\dx \mathcal F_i
=
-\frac{\partial_{tx}N_i\,\dx N_i-\dt N_i\,\dxx N_i}{(\dx N_i)^2}.
\end{align*}
Using the lower bound on $\dx N_i$ and the equation
\begin{align*}
\dt N_i=D\dxx N_i-\rho_iN_i,
\end{align*}
we get
\begin{align*}
|\dx \mathcal F_i|
&\leq
C
\left(
|\partial_{tx}N_i||\dx N_i|
+
|\dt N_i||\dxx N_i|
\right) \\
&\leq
C
\left(
|D\partial_{xxx}N_i-\dx(\rho_iN_i)|\,|\dx N_i|
+
|D\dxx N_i-\rho_iN_i|\,|\dxx N_i|
\right).
\end{align*}
From \Cref{pStep1a} and \Cref{pStep1b}, we have
\begin{align*}
\norm{N_i}_{L^\infty([0,T],W^{3,\infty}(\R))}
 \leq 2 \norm{N_{\rm in}}_{W^{3,\infty}(\R)}
 \quad 
 \text{ and }
 \quad
\norm{\rho_i}_{L^\infty([0,T],W^{1,\infty}(\R))}
 \leq 2 \norm{\nicefrac{\r_{\rm in}}{U}}_{C^{1,\alpha}(\R)}.
\end{align*}
Hence there exists $K>0$, independent of $T$ and of $\bar x_i$, such that
\begin{align*}
\norm{\dx\mathcal F_i}_{L^\infty([0,T]\times I)}
\leq K.
\end{align*}
This proves \eqref{lStep4-1}.

We now prove \eqref{lStep4-2}. For $(t,x)\in[0,T]\times I$, using again the lower bounds on $\dx N_1$ and $\dx N_2$, we have
\begin{align*}
&|\mathcal F_1(t,x)-\mathcal F_2(t,x)|
\\
=&
\left|
\frac{\dt N_2(t,x)}{\dx N_2(t,x)}
-
\frac{\dt N_1(t,x)}{\dx N_1(t,x)}
\right| \\
\leq&
C
\left(
|\dt N_1(t,x)|\,|\dx N_1(t,x)-\dx N_2(t,x)|
+
|\dt N_1(t,x)-\dt N_2(t,x)|
\right).
\end{align*}
Since
\begin{align*}
\dt N_i=D\dxx N_i-\rho_iN_i,
\end{align*}
we obtain
\begin{align*}
&|\mathcal F_1(t,x)-\mathcal F_2(t,x)|\\
\leq &
C\left(
\norm{D\dxx N_1-\rho_1N_1}_{\infty}
\norm{\dx N_1-\dx N_2}_{\infty}
+
\norm{\dxx N_1-\dxx N_2}_{\infty}
\right.\\
&\left.
+
\norm{\rho_1-\rho_2}_{\infty}\norm{N_1}_{\infty}
+
\norm{\rho_2}_{\infty}\norm{N_1-N_2}_{\infty}
\right).
\end{align*}

It remains to estimate the differences in the stationary frame. We do this explicitly for the second derivative of $N$ and the other terms are treated in a similar and easier way. For any $y\in\R$,
\begin{align*}
&|\dxx N_1(t,y)-\dxx N_2(t,y)|\\
&=
|\dzz \tilde N_1(t,y-\bar x_1(t))
-
\dzz \tilde N_2(t,y-\bar x_2(t))|\\
&\leq
|\dzz \tilde N_1(t,y-\bar x_1(t))
-
\dzz \tilde N_1(t,y-\bar x_2(t))|\\
&\quad+
|\dzz \tilde N_1(t,y-\bar x_2(t))
-
\dzz \tilde N_2(t,y-\bar x_2(t))|\\
&\leq
\norm{\partial_{zzz}\tilde N_1}_{\infty}
|\bar x_1(t)-\bar x_2(t)|
+
\norm{\dzz \tilde N_1-\dzz\tilde N_2}_{\infty}.
\end{align*}
Using
\begin{align*}
\norm{\bar x_1-\bar x_2}_{\infty}
\leq
T^{1-\frac1p}
\norm{\dot{\bar x}_1-\dot{\bar x}_2}_{p}
\end{align*}
and \eqref{NMap-Lip}, we get
\begin{align*}
\norm{\dxx N_1-\dxx N_2}_{\infty}
\leq
\Theta(T)
\norm{\dot{\bar x}_1-\dot{\bar x}_2}_{p},
\end{align*}
for a function $\Theta$ such that $\Theta(T)\to0$ as $T\to0$.

Similarly, using \eqref{vMap-Lip} and \eqref{NMap-Lip}, together with the same shift estimate, we obtain
\begin{align*}
\norm{\rho_1-\rho_2}_{\infty}
+
\norm{N_1-N_2}_{\infty}
+
\norm{\dx N_1-\dx N_2}_{\infty}
\leq
\Theta(T)
\norm{\dot{\bar x}_1-\dot{\bar x}_2}_{p}.
\end{align*}
Combining these estimates with the uniform bounds on $N_i$ and $\rho_i$, we conclude that, for all
$(t,x)\in[0,T]\times I$,
\begin{align*}
|\mathcal F_1(t,x)-\mathcal F_2(t,x)|
\leq
\Theta(T)
\norm{\dot{\bar x}_1-\dot{\bar x}_2}_{p}.
\end{align*}
This proves \eqref{lStep4-2} and concludes the proof of \Cref{lStep4} and thus the proof of \Cref{ThmExistence}.
\end{proof}

\end{proof}

\begin{remark}
\label{rk-optimality-N}
Step 2 naturally raises the question whether one can improve the conclusion to
\begin{align*}
\tilde N\in L^\infty([0,T],C^{3,\alpha''}(\R))
\end{align*}
for some $\alpha''\in(0,1)$. The obstruction is already present in the explicit discontinuous contribution
\begin{align*}
H(t,z):=\int_0^t \dzz e^{(t-s)D\dzz}\mathbbm{1}_{z\geq0}ds.
\end{align*}
Indeed, using the heat kernel, one obtains, for $z\neq0$,
\begin{align*}
H(t,z)
=
-C_D\,\sgn(z)
\int_{\frac{|z|}{\sqrt{Dt}}}^{+\infty} e^{-c u^2}du,
\end{align*}
for some constants $C_D,c>0$ depending only on $D$. In particular, $H(t,\cdot)$ is bounded uniformly in $t$, but it has a nontrivial jump at $z=0$ for every $t>0$. Therefore
\begin{align*}
H\notin L^\infty([0,T],C^{0,\alpha''}(\R))
\qquad
\text{for every } \alpha''\in(0,1).
\end{align*}
Consequently, the estimate
\begin{align*}
\tilde N\in L^\infty([0,T],W^{3,\infty}(\R))
\end{align*}
obtained in Step 2 is critical in general: one cannot expect a global $C^{3,\alpha''}$ bound across the interface without additional compatibility
conditions cancelling the jump.
\end{remark}

\subsection{A Sufficient Criterion for Preservation of Monotonicity}
\label{Sect-Exist-Monotonicity}

\begin{cor}[Preservation of monotonicity]
\label{cor-existence}\hspace{0.1cm}\\
Suppose that in addition to the assumptions of \Cref{ThmExistence}, the initial conditions $(\r_{\rm in},N_{\rm in})$ satisfy the following conditions:
\begin{enumerate}
\item $\nicefrac{\dx {N}_{\rm in}}{{N}_{\rm in}}, \nicefrac{\dx {\r}_{\rm in}}{{\r}_{\rm in}}\in BUC(\R)$,
\item there exist $\nu, \eta>0$, such that $\liminf_{x\to -\infty} \nicefrac{\dx {N}_{\rm in}}{{N}_{\rm in}}\geq \nu $ and\\ $\limsup_{x\to +\infty}\nicefrac{\dx {\r}_{\rm in}}{{\r}_{\rm in}}\leq -\eta$,
\item $N_\init-{N}_{\rm in}$ is square-integrable at $x=+\infty$.
\end{enumerate}
Then $(\r,N)$, the solution given by \Cref{ThmExistence}, satisfies the Property $\mathbf{(P_1)}$ locally in time, namely
\begin{align*}
N(t,\cdot)\text{ is increasing for all }t\in[0,T].
\end{align*}
Consequently, by the persistence of the threshold interface established in \Cref{lStep3}, $(\rho,N)$ satisfies $\mathbf{(P_2)}$ locally in time, that is, for each $t\in[0,T]$ there exists a unique $\bar x(t)$ such that
\begin{align*}
N(t,\bar x(t))=N_\thresh.
\end{align*}
Hence $(\rho,N)$ is the unique solution to System \eqref{diffusivemodel:main}.
\end{cor}

Let us briefly comment on the assumptions of \Cref{cor-existence}. The condition on $N_{\rm in}$ at $x=-\infty$ means that, as one moves to the right from $-\infty$, $N_{\rm in}$ grows at least exponentially. The assumption that $N_\init-N_{\rm in}$ is square-integrable near $+\infty$ is more restrictive than merely requiring $\lim_{x\to+\infty}N_{\rm in}(x)=N_\init$. However, we will see in the next section that the traveling wave solution satisfies
\begin{align*}
\tilde N(z)=N_\init+p(z),
\end{align*}
where $p$ is dominated by an exponentially decreasing function as $z\to+\infty$, and is therefore square-integrable near $+\infty$.

Concerning the condition on $\r_{\rm in}$, recall that $\tilde \r=vU$, where $U$ decreases exponentially as $z\to+\infty$ and $v$ is bounded. Hence
\begin{align*}
\frac{\dx \r_{\rm in}}{\r_{\rm in}}
=
\frac{\dx v_{\rm in}}{v_{\rm in}}
+
\frac{\dx U}{U}.
\end{align*}
Thus the additional condition translates into
\begin{align*}
\limsup_{x\to+\infty}
\frac{\dx v_{\rm in}}{v_{\rm in}}
\leq
\chi-\eta.
\end{align*}
Since boundedness of $v_{\rm in}$ prevents $\frac{\dx v_{\rm in}}{v_{\rm in}}$ from remaining uniformly positive at $+\infty$, this assumption can be interpreted as a restriction on the oscillations of the logarithmic derivative $\frac{\dx v_{\rm in}}{v_{\rm in}}$.

\begin{proof}

\textit{Step 1: Bounded logarithmic derivative of $\tilde N$.}

We start by showing that
\begin{align*}
w:=\frac{\dz \tilde N}{\tilde N} \in C([0,T];L^\infty(\R)).
\end{align*}
Since $\tilde N$ remains strictly positive locally in time by the maximum principle, we may divide \eqref{movingmodel:b} by $\tilde{N}$ and obtain
\begin{align*}
\frac{\dt \tilde{N}}{\tilde{N}} = D (\dz w + w^2)+\dot{\bar{x}} w -  vU.
\end{align*}
Then by observing that $\dz\left(\nicefrac{\dt \tilde{N}}{\tilde{N}}\right)=\partial_{zt}\log \tilde{N} = \dt\left(\nicefrac{\dz \tilde{N}}{\tilde{N}}\right)$ and by differentiating the preceding equation, $w$ satisfies the following equation:
\begin{align*}
\dt w = D (\dzz w + \dz\left( w^2\right))+\dot{\bar{x}} \dz w -  \dz \left(vU\right).
\end{align*}
This leads to the following representation formula for $w$:
\begin{align*}
w(t)= e^{tD\dzz}w_{\rm in}+\int_0^t e^{(t-s)D\dzz} \dz\left(D w^2(s)+ \dot{\bar{x}}(s) w(s)  -  v(s)U \right)ds,
\end{align*}
where $w_{\rm in}:=\nicefrac{\dz N_{\rm in}}{N_{\rm in}}$.

By arguments similar to the ones exposed in the proof of \Cref{ThmExistence}, we show that for $R>0$ big enough there exists $T>0$ such that the right handside defines a contraction from $L^\infty([0,T];B(R))$ into itself, where $B(R)=\{ f\in L^\infty(\R),\norm{f}_\infty\leq R\}$. We merely treat the term $w \in L^\infty([0,T];B(R)) \mapsto \int_0^t e^{(t-s)D\dzz} D \dz \left( w^2(s)\right) ds \in  L^\infty([0,T];B(R))$. Let $w_1,w_2 \in  L^\infty([0,T]\times\R)$, we have:
\begin{align*}
&\norm{\int_0^t e^{(t-s)D\dzz} \left( \dz \left(w_1^2\right)(s) - \dz \left(w^2_2\right)(s) \right) ds}_\infty \\
\leq & C\int_0^t \norm{ w_1^2(s)- w_2^2(s)}_\infty \frac{ds}{\sqrt{t-s}} \\
\leq & C\int_0^t \norm{ w_1(s)- w_2(s)}_\infty\norm{ w_1(s)+ w_2(s)}_\infty \frac{ds}{\sqrt{t-s}} \\
\leq & CR\sqrt{ T} \norm{w_1-w_2}_{\infty}.
\end{align*}
Hence $w=\nicefrac{\dz \tilde N}{\tilde N}\in L^\infty([0,T]\times\R)$. 

Furthermore, we show that $w\in C([0,T];L^\infty(\R))$.
Since $w_{\rm in}\in BUC(\R)$ and the heat semigroup is strongly continuous on $BUC(\R)$, we have that $ e^{tD\dxx}w_{\rm in} \in C([0,T];L^\infty(\R)) $. It remains to treat the other terms. We only consider the term involving $\dz(w^2)$, the other terms being similar. 
\begin{align*}
I(t):=\int_0^t e^{(t-s)D\dzz}D\dz(w^2)(s)ds.
\end{align*}
First, by the heat-kernel estimate,
\begin{align*}
\norm{I(t)}_\infty
\leq
C\norm{w}_{L^\infty([0,T]\times\R)}^2
\int_0^t (t-s)^{-\frac12}ds
\leq
C\sqrt t\,\norm{w}_{L^\infty([0,T]\times\R)}^2.
\end{align*}
Hence $I(t)\to0$ in $L^\infty(\R)$ as $t\to0$.

Now let $t\in(0,T)$ and $h>0$ such that $t+h\leq T$. We write
\begin{align*}
I(t+h)-I(t)
&=
\int_t^{t+h} e^{(t+h-s)D\dzz}D\dz(w^2)(s)ds \\
&\quad+
\left(e^{hD\dzz}-\text{Id}\right)I(t).
\end{align*}
The first term is bounded by
\begin{align*}
C\norm{w}_{L^\infty([0,T]\times\R)}^2
\int_t^{t+h}(t+h-s)^{-\frac12}ds
\leq
C\sqrt h\,\norm{w}_{L^\infty([0,T]\times\R)}^2,
\end{align*}
and therefore tends to $0$ as $h\to0$.

For the second term, we note that $I(t)\in BUC(\R)$. Indeed, for every $t>0$, the heat kernel representation and the integrability of
$(t-s)^{-1/2}$ imply that $I(t)$ is bounded and uniformly continuous in $z$. Therefore, using again the strong continuity of the heat semigroup on $BUC(\R)$, we obtain
\begin{align*}
\norm{\left(e^{hD\dzz}-\text{Id}\right)I(t)}_\infty
\longrightarrow 0
\qquad\text{as }h\to0.
\end{align*}
Thus $I$ is right-continuous at every $t\in[0,T)$. The left-continuity is obtained in the same way. Consequently,
\begin{align*}
I\in C([0,T];L^\infty(\R)).
\end{align*}
Repeating the same argument for the remaining terms yields
\begin{align*}
w\in C([0,T];L^\infty(\R)).
\end{align*}

\textit{Step 2: Bounded logarithmic derivative of $v$.}

Next, we show that 
\begin{align*}
q:=\frac{\dz v}{v} \in C([0,T];L^\infty(\R)).
\end{align*}
$q$ satisfies:
\begin{align*}
\dt q =\dzz q + \dz (q^2) +\dz (\beta q)  +\dz\gamma.
\end{align*}
The proof is similar to above and we omit the details, except the bound $\int_0^te^{(t-s)\dzz}\dz\gamma ds\in L^\infty([0,T]\times \R)$. Let us observe that\\ $\dz\gamma=\chi\left(\chi+\frac{1}{\chi}-\dot{\bar{x}}(t)\right)\delta_0$. Applying Hölder's inequality then leads to
\begin{align*}
&\norm{\int_0^te^{(t-s)\dzz}\dz\gamma ds}_\infty\\
 =& \norm{\chi \int_0^t \frac{\chi+\frac{1}{\chi}-\dot{\bar{x}}(s)}{\sqrt{4\pi(t-s)}}e^{-\frac{z^2}{4(t-s)}}ds}_\infty \\ \leq& CT^{\frac12-\frac1p} \norm{\chi+\frac{1}{\chi}-\dot{\bar{x}}}_{p} .
\end{align*}

\textit{Step 3: Preservation of monotonicity of $\tilde N$.}

We establish that $\tilde{N}$ is nondecreasing. By assumption $\limsup_{z\to +\infty} \nicefrac{\dz \tilde{\r}_{\rm in}}{\tilde{\r}_{\rm in}}<-\eta<0$ and therefore there exists an $R>0$ such that for every $z>R$, $\nicefrac{\dz v_{\rm in}(z)}{v_{\rm in}(z)} + \nicefrac{U'(z)}{U(z)}<-\eta$. Then, by continuity of $t\mapsto \nicefrac{\dz v(t,\cdot)}{v(t,\cdot)}$ in $L^\infty(\R)$, there exists $T>0$, such that for every $t\in [0,T]$ and $z>R$, $\nicefrac{\dz \tilde \r(t,z)}{\tilde \r(t,z)} = \nicefrac{\dz v(t,z)}{v(t,z)} + \nicefrac{U'(z)}{U(z)}<-\nicefrac{\eta}{2} $. In particular $\dz \tilde \r (t,z) <0$ for $(t,z)\in [0,T]\times [R,+\infty)$.\\
Now, by the assumption
\begin{align*}
\liminf_{z\to-\infty} w_{\rm in}(z)
=
\liminf_{z\to-\infty}\frac{\dz N_{\rm in}(z)}{N_{\rm in}(z)}
\geq \nu,
\end{align*}
there exists $R'<R$ such that for all $z\leq R', w_{\rm in}(z)\geq \nicefrac{\nu}{2}$. Furthermore, since $\dz N_{\rm in}>0$, and thus $w_{\rm in}>0$, and  $w_{\rm in}\in BUC(\R)$, $w_{\rm in}$ is lower bounded by a positive constant on $[R',R]$. Up to reducing and renaming $\nu$, we may therefore assume that for all $z\leq R, w_{\rm in}(z)\geq \nicefrac{\nu}{2}$. By the same argument, we then have for $T>0$ small enough that for $t\in[0,T],z\in(-\infty,R], w(t,z)\geq \nicefrac{\nu}{4}>0$. \\
Therefore it remains to show that on the interval $[R,+\infty)$, we also have that $\dz \tilde N\geq 0$. For $h>0$, let $f(t,z): = \tilde N(t,z+h)-\tilde N(t,z)$ and $g(t,z):=-  \tilde \r(t,z+h)\tilde N(t,z+h)+  \tilde  \r(t,z)\tilde N(t,z)$. Then we have that:
\begin{align}
\label{heat-fh-gh}
\dt f -\dot{\bar{x}}\dz f- D\dzz f = g.
\end{align} 
Furthermore for $(t,z) \in [0,T]\times (-\infty,R)$, we have that $f\geq 0$ by the preceding and for $z>R$, if $f(t,z)=\tilde N(t,z+h)-\tilde N(t,z)<0$, then we must have that $g(t,z)=-  \tilde\r(t,z+h)\tilde N(t,z+h)+ \tilde \r(t,z)\tilde N(t,z) \geq 0$, since $\tilde \r(t,z+h)\leq \tilde \r(t,z)$. Therefore we have $ f_-g\geq 0$, where $(\cdot)_-=-\min(0,\cdot)$. 

By the lower bound on $w$ and the upper bound on the logarithmic derivative of $\tilde{\r}$, we have that $\tilde N$ is dominated at $z=-\infty$ by $e^{\nicefrac{\nu z}{4}}$ and $\tilde \r$ is dominated at $z=+\infty$ by $e^{-\nicefrac{\eta z}{2}}$. From this we deduce that $g(t,\cdot)\in L^2(\R)$. By the assumption that $N_\init-N_{\rm in}$ is square-integrable at $x=+\infty$, we also have that $f(0,\cdot)\in L^2(\R)$. Therefore $f(t,\cdot)\in L^2(\R)$ as solution to \eqref{heat-fh-gh} and the following computations are justified:
\begin{align*}
&\frac{d}{dt}\left(\frac{1}{2}\int_\R (f_-)^2\right) \\
=& -\int_\R f_- \dt f  \text{, since }f_-\dt \left(f_-\right) = -f_- \dt f \\
=& -\dot{\bar{x}}(t) \int_\R  f_- \dz f -D \int_\R  f_- \dzz f - \int f_- g \\
=& \dot{\bar{x}}(t) \int_\R  f_- \dz (f_-) -D \int_\R  f_- \dzz f - \int f_- g\text{, since }f_-\dz \left(f_-\right) = -f_- \dz f \\
\leq &\frac{\dot{\bar{x}}(t)}{2} \int_\R  \dz ((f_-)^2)+ D \int_\R (\dz f_-)\dz f \text{, since } f_-g \geq 0 \\
=& -D\int_\R (\dz f_-)^2  \text{, since }\dz f_-\dz f = -(\dz f_-) ^ 2 \\
\leq& 0.
\end{align*}
But, by assumption $f_-(0,\cdot) \equiv 0$. Hence $f_-(t,\cdot) \equiv 0$ and $\tilde N(t,\cdot)$ is nondecreasing.
\end{proof}

\section{The Traveling Waves}
\label{Sect-TW}

In this section, we investigate the existence of waves for System\eqref{diffusivemodel:main}, \textit{i.e.} solutions of the form $( \r(t,x),N(t,x))=(\tilde \r(x-\s t),\tilde N(x-\s t))$, for a wave speed $\s $ to be determined. Setting $z=x-\s t$, we find that any traveling wave solution must satisfy the following equations:
\begin{subequations}\label{diffusivewave:main}
\begin{empheq}[left={\empheqlbrace\,}]{align}\label{diffusivewave:a}
&-\s  \tilde\r' - \tilde\r'' +  ( \chi \sgn(\dz N) \mathbbm{1}_{\tilde N\leq N_\thresh  }\tilde\r)' = \mathbbm{1}_{\tilde N>N_\thresh} \tilde\r,  \\ 
&-\s  \tilde N' - D  \tilde N'' =-  \tilde\r \tilde N. \label{diffusivewave:b}
\end{empheq}
\end{subequations}
To simplify the notation, we henceforth drop the diacritical $\tilde{} $. Under the assumption that $N$ is increasing, \eqref{diffusivewave:a} reduces to a second-order linear ordinary differential equation with piecewise-constant coefficients. By translation invariance, we suppose that $N(0)=N_\thresh$. Together with the $C^1$-jump relation \eqref{C1discRel}, which follows from the continuity of the flux, we obtain the following problem:
\begin{align}
\label{waverho}
\left\{\begin{array}{ll}
-\s\r'- \r''+\chi \r' =0, & \text{ for }z<0, \\
-\s\r'-\r'' =\r, & \text{ for }z>0,
\end{array} \right.
\hspace{.5cm} \text{ and } \hspace{.5cm} \r'(0^+)-\r'(0^-)=-\chi\r(0).
\end{align}
We solve this problem explicitly and thus deduce all bounded and nonnegative traveling wave profiles for $\r$. Moreover, there exists a minimal speed $\s^*$, such that for every $\s\in[\s^*,+\infty)$, there exists a unique (up to a multiplicative factor) traveling wave profile $\r^\s$. In a second step, given the profile $\r^\s$, we construct a corresponding traveling wave profile $N^\s$ and the condition $N(0)=N_\thresh$ will fix the multiplicative factor of $\r$, thus leading to a unique traveling wave profile $(\r^\s,N^\s)$ for each $\s \geq \s^*$.

Let us introduce some notation, before moving on to the statement of \Cref{thmparwave}. Define the Fisher/Kolmogorov–Petrovsky–Piskunov speed $\s_\text{F/KPP}:=2$. Note that $\chi+\frac{1}{\chi}\geq \s_\text{F/KPP}$, with equality if and only if $\chi=1$. Furthermore set for $\s\geq 2$, $\mu_\pm(\s):=\frac{\s\pm\sqrt{\s^2-4}}{2}$. We then have the following inequality for $\s>\s_\text{F/KPP}=2$:
$$0<\mu_-(\s)<\mu_-\left(\s_\text{F/KPP}\right)=1=\mu_+\left(\s_\text{F/KPP}\right)<\mu_+(\s) .
$$
In addition, the function $\sigma \mapsto \mu_+(\s)$ (resp. $\sigma \mapsto \mu_-(\s)$) is increasing (resp. decreasing).

\begin{thm}
\label{thmparwave}
Fix $N_\init>N_\thresh$ and under the assumption that $N$ is increasing, there exists a minimal speed $\s^*$ independent of $N_\init$, such that a bounded and nonnegative traveling wave profile $(\r^\s(z),N^\s(z))$ exists if and only if $\s\geq \s^*$. Given $\s\geq \s^*$, the traveling wave profile $(\r^\s(z),N^\s(z))$ is unique. Moreover, the minimal speed $\s^*$ is given by Formula \eqref{speedformula} and depends on the value of $\chi$:
\begin{itemize}
\item[\textendash] if $\chi>1$, then $\s^*=\chi+\frac{1}{\chi}$,
\item[\textendash] if $\chi \leq 1$, then $\s^*=\s_\text{F/KPP}=2$.
\end{itemize}
The traveling wave profile $(\r^\s(z),N^\s(z))$ admits the following limits at $z=\pm \infty$:
\begin{align*}
&\lim_{z\to -\infty} \r^\s(z) = \r^\s(-\infty), & \lim_{z\to +\infty} \r^\s(z) = 0, \\
&\lim_{z\to -\infty} N^\s(z) = 0, & \lim_{z\to +\infty} N^\s(z) = N_\init,
\end{align*}
where the value $\r^\s(-\infty)$ depends on $N_\init,\chi,\s$ and $D$. Moreover, the profile $\r^\s$ is a constant on the negative halfline 
and a decreasing function on the positive halfline. More precisely, there exist $C^\s>0,D^\s\in \R$ such that for $z\geq 0$,
\begin{itemize}
\item[\textendash] for $\s>\s^*,  \r^\s(z)= C^\s e^{-\mu_-(\s)z}+ D^\s e^{-\mu_+(\s)z}$,
\item[\textendash] for $\chi>1$, $\s=\s^*=\chi+\frac{1}{\chi},  \r^{\s^*}(z)= C^{\s^*} e^{-\mu_+({\s^*})z}$ and $\mu_+({\s^*})=\chi$,
\item[\textendash] for $\chi\leq 1, \s=\s_\text{F/KPP}, \r^{\s_\text{F/KPP}}(z)= C^{\s_\text{F/KPP}}( (1-\chi)z+1) e^{-z}$.
\end{itemize}
In addition, let $\mu>0$, with $\mu\neq \frac{\s}{D}$, such that $\r(z)\leq Ce^{-\mu z}$ for a constant $C>0$, then there exists another constant $C'>0$, such that for $z\in \R$,
\begin{align}
\label{NasymptoticProperty}
|N(z)-N_\init|\leq C'\left(e^{-\frac{\s}{D}z}+e^{-\mu z} \right).
\end{align}
\end{thm}

Note that the behavior of $\r^\s$ on the positive halfline and Property \eqref{NasymptoticProperty} will be crucial to the proofs in \Cref{Sect-Inside} and \Cref{Sect-Asympt}. 

\begin{proof}
\textit{Step 1: No traveling wave exists for $\sigma\leq \chi$.}

Integrating \eqref{waverho} over the whole line yields $(\s-\chi )\r^\s(-\infty)=\int_{\R_+} \r^\s(z)dz$ (as we will see just below, $\r^\s$ is integrable at $z=+\infty$). Therefore by nonnegativity of the left-hand side, we find that $\s>\chi$.

\medskip 

\textit{Step 2: Solving ODE \eqref{waverho} on the negative halfline.}

Consider \eqref{waverho} for $z<0$. Its characteristic polynomial is $ X^2+(\s-\chi)X$ and has roots $0$ and ${\chi-\s}$. There exist two constants $A^\s,B^\s\in\R$, such that $\r^\s(z)=A^\s+B^\s e^{{(\chi -\s)}z}$, for $z<0$.  Since $\s>\chi$, the term $e^{{(\chi -\s)}z}$ is unbounded on $\R_-$, which leads to $B^\s =0$. 

\medskip 

\textit{Step 3: Solving ODE \eqref{waverho} on the positive halfline.}

\medskip 

\textit{Step 3. (a): No traveling wave exists for $\sigma<2=\sigma_{\text{F/KPP}}$.}

Consider \eqref{waverho} for $z>0$. Its characteristic polynomial is $ X^2+\s X+1$ and its discriminant is $\s^2-4$. If the discriminant is negative, the roots are complex and $\r^\s$ would be a linear combination of two oscillating functions, which is incompatible with the nonnegativity condition. Hence $\s^2\geq 4$, or, by positivity of $\s$, $\s\geq 2$. 

\medskip 

\textit{Step 3. (b): The traveling waves for $\sigma>2=\sigma_{\text{F/KPP}}$.} 

Let us show that in the small bias case, traveling waves exist for all $\sigma>2=\sigma_{\text{F/KPP}}$, whereas in the large bias case, traveling waves exist if and only if $\sigma\geq \chi+\frac{1}{\chi}$.

Suppose that $\s>2=\s_\text{F/KPP}$. The roots of the characteristic polynomial are then $-\mu_\pm(\s)$ and there exist two constants $C^\sigma,D^\s\in \R$ such that $\r^\s(z)=C^\s e^{-\mu_-z} + D^\s e^{-\mu_+ z} $. By the continuity at $z=0$ of $\r^\s$, we obtain equality $A^\s=C^\s +D^\s$ and by the $C^1$-jump relation \eqref{C1discRel}, we obtain equality $-\chi A^\s=-\mu_-C^\s-\mu_+D^\s$, leading to the following system:
\begin{align*}
\left\{
\begin{array}{ll}
A^\s & = C^\s +D^\s \\
-\chi A^\s & = -\mu_-C^\s +-\mu_+D^\s 
\end{array}
\right.
\iff
\left\{
\begin{array}{ll}
C^\s & = A^\s \frac{\mu_+-\chi}{\mu_+-\mu_-}, \\
D^\s & = A^\s \frac{\chi-\mu_-}{\mu_+-\mu_-} .
\end{array}
\right.
\end{align*}
Since $\mu_+-\mu_-= \sqrt{\s^2-4}$, we obtain
\begin{align*}
\r^\s(z)&=\frac{A^\s}{\sqrt{\s^2-4}}\left( ( \mu_+ -\chi ) e^{-\mu_- z}+(\chi-\mu_- )e^{-\mu_+ z} \right) \\
&=\frac{A^\s e^{-\mu_+ z}}{\sqrt{\s^2-4}}\left(  ( \mu_+ -\chi ) e^{(\mu_+ -\mu_- )z} +(\chi-\mu_- )\right).
\end{align*}
From the latter expression, we observe that $\rho^\s$ is nonnegative for all $z\geq 0$ if and only if the monotone function $z\mapsto ( \mu_+ -\chi ) e^{(\mu_+ -\mu_- )z}+ (\chi-\mu_- )$ is nonnegative. Its value at $z=0$ is $\mu_+-\mu_-$ which is positive and its sign at $z=+\infty$ is determined by the sign of $\mu_+ -\chi$. Thus $\rho^\s(z)$ is positive for $z\geq 0$, if and only if, $  \mu_+(\s) \geq \chi  $. \\
In the small bias case $\chi\leq 1$, this inequality always holds, as $\mu_+(\s)\geq \mu_+(\s_{\text{F/KPP}})=1 \geq \chi$. \\
For the large bias case $\chi>1$, solving the equation:
\begin{align*}
\mu_+({\sigma})=\chi \iff \frac{\s+\sqrt{\s^2-4}}{2}=\chi \iff \s = \chi+\frac{1}{\chi}.
\end{align*}
Since $\s\mapsto \mu_+(\s)$  is increasing, the inequality  $  \mu_+(\s)\geq \chi  $ holds, if and only if  $\s\geq \chi+\frac{1}{\chi}$.  

\medskip 

\textit{Step 3. (c): The traveling waves for $\sigma=2=\sigma_{\text{F/KPP}}$.}

Suppose $\s=2=\s_\text{F/KPP}$, then there exist two constants $C^\s,D^\s\in \R$ such that $\r^\s(z)=(C^\s z+D^\s)e^{-z } $. By similar arguments to those in Step 3b), we have $D^\s=A^\s$, $C^\s=({1-\chi })A^\s$, which leads to
\begin{align*}
\r^\s(z)=A^\s\left((1-\chi)z+1\right) e^{-z } .
\end{align*} 
To satisfy the nonnegativity condition, we must have $\chi\leq 1 $. Thus,$\s=2=\s_\text{F/KPP}$, is an admissible traveling wave speed, if and only if $\chi \leq 1 $. In particular, in the large bias case $\chi>1$, there does not exist a traveling wave at speed $\s=2=\s_\text{F/KPP}$.

We conclude from Step 3 that bounded, nonnegative traveling waves can exist only when $\s\geq \s^*$, where $\s^*$ is given by Formula \eqref{speedformula}. Combining Step 2 and Step 3, we also find that $\rho^\s$ is a constant on the negative halfline and is a decreasing function on the positive halfline. The precise expression on the positive are those derived from Step 3b) and Step 3c). In particular, note that in the case of large bias 
and $\s=\s^*$, we have that $ \mu_+({\s^*})=\chi$, which leads to $\r^{{\s^*}}(z)=A^\s e^{-\chi z}$, for $z\geq 0$.

\medskip 

\textit{Step 4: Construction of the corresponding profile $N^\s$ and determination of $A^\s$.}

It remains to show that for every wave profile $\r^\s$, we have a corresponding wave profile $N^\s$, which satisfies the condition $N^\s(0)=N_\thresh$. To do so, notice that by linearity of the equation for $\r$, the profile $\r^\s$ is defined up to the multiplicative constant $A^\s$, which remains yet to be determined. For a given constant $A^\s>0$, denote $\r^\s_{A^\s}$ its corresponding profile. One readily shows that \eqref{diffusivewave:b} with prescribed boundary condition $N^\s(+\infty)=N_\init$ has a unique bounded and nonnegative solution, which we denote $N^\s_{A^\s}$. Moreover, the map $A^\s\mapsto N^\s_{A^\s}(0)$ is continuous and decreasing and its range is $(0,N_\init)$. Therefore there exists a unique $A^\s$, such that $N^\s_{A^\s}(0)=N_\thresh$. Thus, for each fixed $\s\geq \s^*$, we obtain a unique traveling wave profile  $(\r^\s_{A^\s},N^\s_{A^\s})$. Note that the constant $A^\s$ is not computed explicitly.

\medskip 

\textit{Step 5: The behavior of $N^\s$ at $z=+\infty$. Establishing Property \eqref{NasymptoticProperty}.}

Let us finally prove Estimate \eqref{NasymptoticProperty}. Integrating \eqref{diffusivewave:b} on the interval $(z,+\infty)$ and noticing that $N'(+\infty)=0$, we find
\begin{align*}
&\s \left( N(z)-N_\init\right) + D  N'(z) = -  \int_z^{+\infty}\r(y)N(y)dy \\
\iff & \frac{\s}{D} \left( N(z)-N_\init\right) +   \left( N(z)-N_\init\right) ' = - \frac{1}{D} \int_z^{+\infty}\r(y)N(y)dy .
\end{align*}
Hence,
\begin{align*}
\left|\left( e^{\frac{\s}{D}z}(N(z)-N_\init) \right)'\right|\leq \frac{e^{\frac{\s}{D}z}}{D}\int_z^{+\infty} Ce^{-\mu y} N_\init dy \leq 
\frac{CN_\init}{\mu D} e^{\left(\frac{\s}{D}-\mu\right)z}.
\end{align*}
Integrating this inequality on the interval $(0,z)$ for a certain $C'>0$, we find
\begin{align*}
\left|N(z)-N_\init\right| \leq C \left(e^{-\frac{\s}{D}z}+e^{-\mu z} \right).
\end{align*}
\end{proof}

\Cref{thmparwave} shows that there are infinitely many traveling waves. In \Cref{Sect-Asympt}, we obtain a weak speed-selection result identifying $\s^*$ as the only possible limiting propagation speed for biologically relevant initial data.

\section{Inside Dynamics of Traveling Waves}
\label{Sect-Inside}

We elaborate on the properties of the traveling waves, following the lines of \cite{garnier2012}, as well as \cite{roques2012} to a small extent. We establish that in the large bias case, the wave is \textit{pushed}, and in the small bias case (or when $\s>\s^*$), the wave is \textit{pulled}, according to the definition proposed in \cite{garnier2012} (see the discussion in the Introduction on the ambiguity of \textit{pushed} and \textit{pulled} waves). To do so, we are using the formalism of neutral fractions. The aim is to study the behaviour of partitions of the traveling wave profile (see \cite{garnier2012} and  \cite{cochet2021} for the biological relevance of this decomposition).

\begin{definition}[Neutral Fractions]
Define $L:=- \dzz - \beta \dz $, where:
\begin{align}
\label{defBeta}
\beta(z)=\sigma-\chi \mathbbm{1}_{z\leq 0}+2\frac{\dz \r^\sigma}{\r^\sigma}
\end{align}
A neutral fraction $\nu$ (of the traveling wave $\r^\s$) is a solution to the following equation:
\begin{align}
\label{neutralfraction}
\left\{\begin{array}{l}
\dt \nu +L\nu=0 \\
\nu(0,\cdot)=\nu_{\rm in}
\end{array}\right. .
\end{align}
\end{definition}
It is clear that any constant is a neutral fraction, as stationary solutions to \eqref{neutralfraction}. The interest in \eqref{neutralfraction} stems from the following observation. Suppose we have neutral fractions $(\nu_i)_{i=1}^k\geq 0$ that satisfy $\sum_{i=1}^k \nu^0_i = 1$. It amounts to marking each part $\nu^0_i \r^\s$ of the population with neutral labels, \textit{i.e.} that do not interfer with the dynamics. The neutral fractions $(\nu_i(t)\r^\s)_{i=1}^k$ then describe the evolution over time of the distribution of these labels. Because of this interpretation, it is natural to suppose that $\nu^0$ takes its values in $[0,1]$, but such a restriction is of no relevance for the subsequent analysis.
Of note, describing $\nu_i$ or $\nu_i \r^\sigma$ is equivalent, only the expression of the operator $L$ changes. 

We now give a definition of \textit{pushed} and \textit{pulled} waves following the point of view given by \cite{garnier2012}. For the sake of clarity, this definition remains somewhat vague and the technicalities are discussed in \Cref{subsection-pushed} and \Cref{subsection-pulled}. 

\begin{definition}[Pushed and Pulled Waves]
\label{defPushedPulled}
A traveling wave $\r^\s$ will be called:
\begin{itemize}
\item pushed, if any neutral fraction $\nu$, \textit{i.e.} a solution to the evolution Problem \ref{neutralfraction}, converges to a constant function.
\item pulled, if any neutral fraction $\nu$, whose initial condition $\nu^0$ decreases sufficiently fast at $z=+\infty$ (see \Cref{ThmPulled} and \Cref{cor-pulled} for a more precise definition), converges to zero.
\end{itemize}
\end{definition}

\subsection{Pushed Front Dynamics in the Large Bias Case}
\label{subsection-pushed}

In this subsection, we develop arguments similar to those in
\cite{gallay2005,garnier2012}; see also \cite{roques2012}. The evolution
of neutral fractions in the large-bias, minimal-speed regime, \textit{i.e.} $\chi>1$ and $\s=\s^*$, is characterized by \Cref{thmpushed}.

Consider the operator $L$ in the space $L^2(e^V dz)$, where
\begin{align}
\label{defV}
V' = \beta,
\qquad
V(0)=0.
\end{align}
The condition $V(0)=0$ fixes the additive constant in $V$. 
On its natural domain, $L$ is self-adjoint, its kernel is spanned by
the constant functions, and the remainder of its spectrum is contained
in $[\gamma,+\infty)$, where
\begin{align}
\label{def.gamma}
\gamma:=\frac14\min\left\{\s^2-4,\frac1{\chi^2}\right\}>0.
\end{align}
This yields the following convergence result.

\begin{thm}[Large bias minimal speed wave is pushed]
\label{thmpushed}
Suppose that $\chi>1$ and that $\s=\s^*$. Let $\nu$ a solution to \eqref{neutralfraction} with $\nu_{\rm in} \in L^2(e^V dz)$, where $V$ is defined by \eqref{defV}. Then we have the following convergence result:
\begin{align}
\label{expdecay}
\norm{\nu(t)-\langle \nu_{\rm in}\rangle}_{L^2(e^Vdz)} \leq \norm{\nu_{\rm in}}_{L^2(e^V dz)} e^{-\gamma t},
\end{align}
with $\langle \nu_{\rm in}\rangle=\frac{\int \nu_{\rm in}  e^V dz}{\int  e^V dz}$ and $\gamma:=\frac{1}{4}\min \left(\s^2-4,\frac{1}{\chi ^2}\right)>0$. 

Furthermore, for any compact set $K\subset\R$, there exists a constant $C_K>0$, such that for $t>0$,
\begin{align}
\label{linfty-expdecay}
\sup_{z\in K}\left|\nu(t,z)-\langle \nu_{\rm in}\rangle\right|\leq C_K\norm{\nu_{\rm in}}_{L^2(e^V dz)}\left(1+t^{-\frac{1}{2}}\right)e^{-\gamma t}.
\end{align}
\end{thm}

\Cref{thmpushed} states that every (bounded) neutral fraction converges to a constant, and therefore the large-bias minimal-speed traveling wave
is pushed in the sense of \Cref{defPushedPulled}. Indeed, in the large bias case when $\s=\s^*$, $\beta(z)>0$ for $z<0$, and $\beta(z)<0$ for $z>0$. Consequently, $e^V\in L^1(\R)$ and thus every bounded function is in $L^2(e^Vdz)$.

In other words, independently of its initial datum, each bounded neutral
fraction becomes, asymptotically and uniformly on compact subsets of
space, a constant multiple of the total population. This limiting
multiple depends on the initial datum $\nu_{\rm in}$. This is in stark contrast
with the case $\s=\s_{\mathrm{F/KPP}}$, or more generally $\s>\s^*$,
considered in \Cref{cor-pulled}, where neutral fractions generally go
extinct unless they are represented in the leading edge.

\begin{proof}
\textit{Step 1: The operator $L$ is self-adjoint, closed, and nonnegative.}
    In order to simplify the spectral study of $L$ on $L^2(e^Vdz)$ with domain $D(L)=H^2(e^V dz)$, we introduce the pullback of $L$ to the space $L^2(dz)$, \textit{i.e.} for 
\begin{align*}
    f \in  D(\mathcal L):=\left\{ f \in H^1( dz) \middle| f'-\frac{\beta}{2} f \in H^1(dz)\right\},
\end{align*}
    we have
\begin{align*}
\mathcal Lf:=e^{\frac{V}{2}}L\left(e^{-\frac{V}{2}}f \right)=
-f''
+\frac{\beta'}{2}f
+\frac{\beta^2}{4}f,
\end{align*}
in the sense of distributions.

Define
\begin{align*}
\mathcal Df
:=
f'-\frac{\beta}{2}f,
\qquad
D(\mathcal D)=H^1(\R).
\end{align*}
Since $\beta\in L^\infty(\R)$, the operator $\mathcal D$ is densely
defined and closed. Its adjoint is $\mathcal D^*g
=
-g'-\frac{\beta}{2}g$ with $D(\mathcal D^*)=H^1(\R,dz)$. Consequently,
\begin{align*}
D(\mathcal D^*\mathcal D)
=&
\left\{
f\in D(\mathcal D):
\mathcal D f \in D(\mathcal D^*)
\right\}\\
=&
\left\{
f\in H^1(\R):
f'-\frac{\beta}{2}f\in H^1(\R)
\right\}
=
D(\mathcal L).
\end{align*}
Moreover, for $f\in D(\mathcal L)$,
\begin{align*}
\mathcal D^*\mathcal Df
=
-\left(f'-\frac{\beta}{2}f\right)'
-\frac{\beta}{2}
\left(f'-\frac{\beta}{2}f\right)
=
-f''
+\frac{\beta'}{2}f
+\frac{\beta^2}{4}f=
\mathcal Lf.
\end{align*}
Hence
\begin{align}
\label{L-factorization}
\mathcal L=\mathcal D^*\mathcal D.
\end{align}
It follows from the standard properties of operators of the form
$\mathcal D^*\mathcal D$ that $\mathcal L$ is self-adjoint, closed, and nonnegative. In
particular,
\begin{align}
\label{L-dissipativity}
\langle f,\mathcal Lf\rangle_{L^2(\R)}
=\norm{\mathcal Df}_2^2
=
\int_\R
\left|f'-\frac{\beta}{2}f\right|^2dz.
\end{align}
Since the map $f\mapsto e^{V/2}f$
is unitary from $L^2(e^V,dz)$ onto $L^2(dz)$, the operators $L$ and
$\mathcal L$ are unitarily equivalent. It follows that $L$ is also
self-adjoint, closed, and nonnegative.
\medskip

\textit{Step 2: Spectral structure of $\mathcal{L}$. }

Let us recall from \eqref{def.gamma}
that $\gamma=\nicefrac{1}{4}\min\left(\s^2-4,\nicefrac{1}{\chi^2} \right)$. On the one hand, 
\begin{align*}
    \s_\text{ess}(\mathcal L) =  [\gamma,+\infty).
\end{align*}
On the other hand,
\begin{align*}
\s(\mathcal L)\cap (-\infty,\gamma) = \{0\},
\end{align*}
and $0$ is a simple eigenvalue.

\medskip

\textit{Step 2. (a): The essential spectrum of $\mathcal{L}$. }

Using the definition of $\beta$ given by \eqref{defBeta} and \Cref{thmparwave} in large bias minimal speed wave case, we have that
\begin{align*}
\beta(z)= \left\{
\begin{array}{ll}
\frac{1}{\chi}& \text{, for }z<0, \\
 -\left(\chi-\frac{1}{\chi}\right)& \text{, for }z> 0.
\end{array}
\right.
\end{align*}
Consequently, $\beta'=-\chi\delta_0$ in the sense of distributions.
Hence, away from the origin, $\mathcal L$ acts as
\begin{align*}
    {\mathcal L}f=-f''+\frac{\beta^2}{4}f =\left\{
\begin{array}{ll}
    -f''+\nicefrac{1}{4\chi^2}f & \text{, for }z\leq 0,  \\
     -f''+\nicefrac{(\s^2-4)}{4}f& \text{, for }z> 0,
\end{array}
    \right.
\end{align*}
where we used $\left(\chi-\nicefrac1\chi\right)^2=\s^2-4$.
By \cite[Theorem~9.11]{teschl2009}, the essential spectrum of ${\mathcal L}$ is the union of the essential spectra of any self-adjoint
realizations of these operators on the two halflines. Since
\begin{align*}
&\s_{\rm ess}\left(-\dzz+\frac{1}{4\chi^2}\right)
=
\left[\frac{1}{4\chi^2},+\infty\right),\\
&\s_{\rm ess}\left(-\dzz+\frac{\s^2-4}{4}\right)
=
\left[\frac{\s^2-4}{4},+\infty\right),
\end{align*}
we obtain
\begin{align*}
\s_{\rm ess}({\mathcal L})
=
\left[\frac{1}{4\chi^2},+\infty\right)
\cup
\left[\frac{\s^2-4}{4},+\infty\right)
=
[\gamma,+\infty).
\end{align*}

\medskip

\textit{Step 2. (b): The spectrum of $\mathcal L$ below $\gamma$.}

By Step~2(a), every point of
$\s(\mathcal L)\cap(-\infty,\gamma)$ belongs to the discrete spectrum
of $\mathcal L$ and is therefore an eigenvalue of finite multiplicity. Let $\lambda <\gamma$ be an eigenvalue for $\mathcal L$ in the space $L^2(dz)$.

For $z>0$, the characteristic polynomial of the ordinary differential equation $(\mathcal L-\lambda)f=0$ is $P(\mu)=- \mu^2+\nicefrac{\s^2}{4}-1-\lambda$, whose roots are $\mu_\pm=\pm\frac{1}{2} \sqrt{\s^2-4(1+\lambda)}$, which is well-defined, since $\lambda<\nicefrac{(\s^2-4)}{4}$. This gives rise to two exponential modes $e^{\mu_\pm z}$, but it is clear that only $e^{\mu_- z}\in L^2(\R_+,dz)$.

For $z<0$, a similar argument shows that the only exponential mode
belonging to $L^2(\R_-,dz)$ is $e^{\nu_+z}$, with $\nu_+=\frac{1}{2}\sqrt{\nicefrac{1}{\chi^2}-4\lambda}$, which is well-defined, since $\lambda<\nicefrac{1}{4\chi^2}$.

Since $f_\lambda\in D(\mathcal L)\subset H^1(\R)$, it is continuous at
the origin. Consequently, up to multiplication by a constant, an
eigenfunction associated with $\lambda$ must be of the form
\begin{align*}
f_\lambda(z)= \left\{ \begin{array}{ll}
e^{\nu_+ z} & \text{ if }z<0 \\
e^{\mu_- z} & \text{ if }z\geq0 
\end{array} \right..
\end{align*} 
Moreover, the condition $f_\lambda \in D(\mathcal L)$ implies that 
\begin{align*}
    f_\lambda' - \frac{\beta}2 f_\lambda \in H^1(\R)
\end{align*}
and hence this function is continuous at the origin. We therefore have
\begin{align*}
&\left(f_\lambda' - \frac{\beta}2f_\lambda\right)(0^+)=\left(f_\lambda' - \frac{\beta}2f_\lambda\right)(0^-) \\\iff& \mu_-+\frac{\s}{2} = \nu_++\frac{\s-\chi}{2} \\
\iff& \chi= \sqrt{\s^2-4(1+\lambda)} + \sqrt{\frac{1}{\chi^2}-4\lambda}
\end{align*}
The right-hand side is strictly decreasing in $\lambda$, and thus the equation admits at most one root. Through a quick check we can see that $\lambda=0$ is a root. The corresponding eigenfunction is $e^{V/2}$, which is the pullback of
the constant eigenfunction $1$ of $L$.  Moreover, the preceding calculation shows that an eigenfunction is uniquely determined, up to
a multiplicative constant, by its value at the origin. Hence the eigenvalue $0$ is simple, and
\begin{align*}
\s(\mathcal L)\cap(-\infty,\gamma)=\{0\}.
\end{align*}

\medskip

\textit{Step 3: Convergence of the neutral fractions.}

Let $\mathcal P$ denote the orthogonal projection onto
$\ker(\mathcal L)=\operatorname{span}{e^{V/2}}$. It is given by
\begin{align*}
\mathcal P u :=  \frac{\int u e^{\nicefrac{V}{2}}}{\int e^V}  e^{\nicefrac{V}{2}}.
\end{align*}
By Step~2, the spectrum of $\mathcal L$ restricted to
$\ker(\mathcal L)^\perp$ is contained in
$[\gamma,+\infty)$. The spectral theorem \cite[Section 7.4]{brezis2011} therefore yields
\begin{align*}
\norm{e^{-t\mathcal L}(I-\mathcal P)f}_2
&\leq
e^{-\gamma t}\norm{(I-\mathcal P)f}_2,\\
\norm{\mathcal Le^{-t\mathcal L}(I-\mathcal P)f}_2
&\leq
e^{-\gamma t}\left(\gamma+\frac{1}{t}\right)\norm{(I-\mathcal P)f}_2.
\end{align*}

Set
\begin{align*}
w(t):=e^{V/2}\nu(t),
\qquad
w_{\rm in}:=e^{V/2}\nu_{\rm in}.
\end{align*}
Then $w$ satisfies $\dt w+\mathcal Lw=0$,
and therefore
\begin{align*}
w(t)=e^{-t\mathcal L}w_{\rm in} = \mathcal P w_{\rm in} + e^{-t\mathcal L} (I-\mathcal P)w_{\rm in} .
\end{align*}
Thus,
\begin{align*}
\norm{w(t)-\mathcal P w_{\rm in}}_2
&\leq
e^{-\gamma t}\norm{w_{\rm in}}_2.
\end{align*}
But using that $\mathcal P w_{\rm in} =\langle \nu_{\rm in}\rangle e^{\nicefrac{V}{2}}$, this bound may be rewritten as
\begin{align*}
\norm{\nu(t)-\langle \nu_{\rm in}\rangle}_{L^2(e^Vdz)}
&\leq
e^{-\gamma t}\norm{\nu_{\rm in}}_{L^2(e^Vdz)},
\end{align*}
which establishes \eqref{expdecay}.

Finally, we prove the local uniform estimate. Set
\begin{align*}
r(t)
:=
e^{-t\mathcal L}(I-\mathcal P)w_{\rm in} = w(t)-\mathcal Pw_{\rm in}.
\end{align*}
Using \eqref{L-dissipativity}, we have
\begin{align*}
    &\int_\R \left| \dz r(t) -\frac{\beta}{2} r(t)\right|^2
\\=& \langle r(t), \mathcal L r(t) \rangle 
\\ \leq&  \norm{r(t)}_2 \norm{\mathcal L r(t)}_2 
\\ \leq & e^{-2\gamma t}\left(\gamma+\frac{1}{t}\right)\norm{r(0)}_2^2.
\end{align*}
Consequently,
\begin{align*}
    \norm{\dz r(t) -\frac{\beta}{2} r(t)}_2 \leq Ce^{-\gamma t}\left(1+\frac{1}{\sqrt t}\right)\norm{r(0)}_2.
\end{align*}
Next, since $\beta \in L^\infty(\R)$,
\begin{align*}
   \norm{\dz r(t) }_2\leq \norm{\dz r(t) -\frac{\beta}{2} r(t)}_2 +\frac{\norm{\beta}_\infty}{2}\norm{ r(t) }_2 \leq C e^{-\gamma t}\left(1+\frac{1}{\sqrt{t}}\right)\norm{r(0)}_2.
\end{align*}
By the embedding $H^1(\R) \hookrightarrow L^\infty(\R)$, we obtain
\begin{align*}
    \norm{ r(t) }_\infty\leq C e^{-\gamma t}\left(1+\frac{1}{\sqrt{t}}\right)\norm{r(0)}_2.
\end{align*}
Then, using that $r(t)=e^{\nicefrac{V}{2}}(\nu(t)-\langle\nu_{\rm in}\rangle)$ and $\norm{r(0)}_2\leq \norm{\nu_{\rm in}}_{L^2(e^Vdz)} $, we obtain the bound
\begin{align*}
    \norm{e^{\nicefrac{V}{2}}(\nu(t)-\langle\nu_{\rm in}\rangle)}_\infty  \leq Ce^{-\gamma t}\left(1+\frac{1}{\sqrt{t}}\right)\norm{\nu_{\rm in}}_{L^2(e^Vdz)}.
\end{align*}
Finally, for a compact set $K\subset \R$, dividing by $\inf_{z \in K}e ^{\nicefrac{V}{2}}>0$, we obtain \eqref{linfty-expdecay}. 
\end{proof}

\subsection{Pulled Front Dynamics in the Small Bias Case}
\label{subsection-pulled}

In the small bias case $\chi\leq 1$ (or when $\s>\s^*$), the inside dynamics of the wave is drastically different. In fact, we can start by observing that \Cref{thmpushed} does not apply to these cases. For $\s=\s_{\text{F/KPP}}$ and $z\geq 0$,
\begin{align}
    \frac{\dz \r^\sigma}{\r^\sigma}=\frac{1-\chi}{(1-\chi)z+1}-1.
\end{align}
Therefore, $ e^{V(z)}=((1-\chi)z+1)^2$ and $1\notin L^2(e^V)$.
Similarly, for $\s>\s^*$, $e^V$ behaves like $e^{(\s -2\mu_-(\s))z}$ for $z \gg 1$, and since $\s -2\mu_-(\s)=\sqrt{\s^2-4}>0$, we again have $1\notin L^2(e^V dz)$.

We will show that the traveling waves are pulled in these regimes:
neutral fractions converge to $0$, provided that their initial data
decay sufficiently rapidly as $z\to+\infty$. We first state a general
criterion for solutions of \eqref{neutralfraction}, under suitable
assumptions on $\beta$, which may be viewed as a generalization of
\cite[Theorem~1]{garnier2012}.

\begin{thm}[Weighted extinction criterion for neutral fractions]
\label{ThmPulled} 
Consider
\begin{align}\label{EntEqn}
\left\{ 
\begin{array}{ll}
\dt \nu -\dzz \nu -\beta(z)\dz \nu=0,\\
\nu(0,\cdot)=\nu_{\rm in},
\end{array}
\right. 
\end{align}
with initial datum $\nu_{\rm in}\in L^2(e^Vdz)$, where $V$ is given by \eqref{defV}.

Suppose that $\beta\in L^\infty(\R)$ satisfies the following conditions:
\begin{itemize}
\item There exists $K>0$, such that $\beta'\leq K$ in the sense of distributions.
\item  $\int_\R \beta_- (z)dz <+\infty$, where $(\beta)_-=\max\{ -\beta,0\}$.
\end{itemize}
Then the solution $\nu$ to \eqref{EntEqn} satisfies 
\begin{align}
\lim_{t\to +\infty}\norm{\nu(t)e^{\frac{V}{2}}}_{\infty} = 0.
\end{align}
Moreover, for every $a\in \R$,
\begin{align*}
\lim_{t\to+\infty}\sup_{z\in [a,+\infty)}|\nu(t,z)|=0.
\end{align*}
\end{thm}

Before proving \Cref{ThmPulled}, we apply it to establish the pulled nature of the waves in the two cases mentioned above: the small-bias
minimal-speed case $\chi\leq1$ with $\s=\s_{\text{F/KPP}}=2$, and the nonminimal-speed case
$\s>\s^*$.

\begin{cor}[Small bias or nonminimal speed waves are pulled]
\label{cor-pulled}
Let $\chi\leq 1$ and $\s=2$, or $\s >\s^*$.
Consider a neutral fraction $\nu$ with initial datum satisfying
\begin{align*}
    \nu_{\rm in}\in L^\infty(\R), \qquad \lim_{z\to+\infty}\nu_{\rm in}(z)=0.
\end{align*}
Then, for every $a\in \R$,
\begin{align}
\lim_{t\to+\infty}\sup_{z\in [a,+\infty)}|\nu(t,z)|=0.
\end{align}
\end{cor}

Before proving \Cref{cor-pulled}, let us briefly comment on its interpretation. The condition
\begin{align*}
\lim_{z\to+\infty}\nu_{\rm in}(z)=0
\end{align*}
means that the neutral fraction constitutes an asymptotically negligible part of the leading edge of the traveling wave.\Cref{cor-pulled} shows that any such neutral fraction eventually goes extinct in every region of the form $[a,+\infty)$.
We point out that, although \Cref{ThmPulled} is based on a weighted $L^2$ condition, the simple truncation argument used below allows us to replace this condition by the sole requirement that the neutral fraction vanishes at the leading edge. In particular, no prescribed rate of decay of $\nu_{\rm in}$ as $z\to+\infty$ is required. This is slightly more general than the type of weighted decay assumption commonly used in this context, such as in \cite{garnier2012}.

Notice also that evidently \Cref{cor-pulled} does not apply to $\nu_{\rm in}\equiv1$, consistently with the fact that $\nu\equiv1$ is a stationary solution. Finally, in the large-bias minimal-speed case $\chi>1$ and $\s=\s^*$, the function $\beta$ is bounded above by a negative constant for sufficiently large positive
$z$. Consequently,
\begin{align*}
\int_\R\beta_-(z) dz=+\infty,
\end{align*}
so that the assumptions of \Cref{ThmPulled} are not satisfied in that regime.

\begin{proof}[Proof of \Cref{cor-pulled}]
We first verify that the assumptions of \Cref{ThmPulled} are satisfied
in both cases.

We recall that $\beta(z)=\s-\chi \mathbbm{1}_{z\leq0}+2\frac{\dz\r^\s}{\r^\s}$.

Suppose first that $\chi\leq1$ and $\s=2$. By
\Cref{thmparwave},
\begin{align*}
    \r^\s(z)=\left\{\begin{array}{ll}
1 & \text{, if }z\leq0,\\
((1-\chi)z+1)e^{-z} & \text{, if }z>0,
\end{array} \right.
\end{align*} 
which leads to
\begin{align*}
\beta(z)= \s-\chi \mathbbm{1}_{z\leq0}+2\frac{\dz\r^\s}{\r^\s}=\left\{ \begin{array}{ll}
2-\chi & \text{, if }z \leq 0, \\
\frac{2(1-\chi)}{(1-\chi)z+1} & \text {, if } z> 0 .
\end{array} \right.
\end{align*}
In particular, $\beta\in L^\infty(\R)$ and $\beta\geq0$. Moreover, $\beta'$ is clearly bounded everywhere except at $z=0$, but $\beta(0^-)=2-\chi$ and $\beta(0^+)=2(1-\chi)$. This leads to $\beta'$ being bounded above in the sense of distributions and $\beta\geq 0$, hence \Cref{ThmPulled} applies. 

Suppose now that $\s>\s^*$. From \Cref{thmparwave}, we have
\begin{align*}
    \r^\s(z)= \left\{\begin{array}{ll}
1 & \text{, if }z\leq0,\\
\frac{\mu_+-\chi}{\mu_+-\mu_- }e^{-\mu_- z}-\frac{\mu_--\chi}{\mu_+-\mu_- }e^{-\mu_+ z}  & \text{, if }z>0,
\end{array} \right.
\end{align*}
where we recall that $\mu_\pm=\nicefrac{(\s\pm \sqrt{\s^2-4})}{2} $. Using again $\beta(z)= \s-\chi \mathbbm{1}_{z\leq0}+2\frac{\dz\r^\s}{\r^\s}$ then leads to
\begin{align*}
\beta(z)=\left\{ \begin{array}{ll}
\sigma-\chi & \text{, if }z \leq 0, \\
\s - 2\frac{(\mu_+-\chi)\mu_-e^{-\mu_-z}-(\mu_--\chi)\mu_+e^{-\mu_+z}}{(\mu_+-\chi)e^{-\mu_-z}-(\mu_--\chi)e^{-\mu_+z}} & \text {, if } z> 0.
\end{array} \right. 
\end{align*}
Clearly $\beta \in L^\infty(\R)$  and $\beta'$ is clearly bounded everywhere except at $z=0$. Moreover, $\beta(0^-)=\sigma-\chi$ and $\beta(0^+)=\s-2\chi$, thus $\beta'$ is bounded above. For $z\leq0, \beta(z)\geq0$ and $\lim_{z\to +\infty}\beta(z)= \s-2\mu_->0$, which establishes that $\int_\R \beta_-(z)dz<+\infty$ and hence \Cref{ThmPulled} applies. 

We now use a truncation argument. Fix $R\in\R$ and decompose
\begin{align*}
    \nu_{\rm in} =\nu_{\rm in}^R+r_{\rm in}^R,
\qquad
\nu_{\rm in}^R
:=
\nu_{\rm in}\mathbbm{1}_{z\leq R},
\qquad
r_{\rm in}^R
:=
\nu_{\rm in}\mathbbm{1}_{z>R}.
\end{align*}
Since $\nu_{\rm in}\in L^\infty(\R)$, we have that $\nu_{\rm in}^R\in L^2(e^V,dz)$.

Let $\nu^R$ and $r^R$ denote the solutions of \eqref{EntEqn} with
initial data $\nu_{\rm in}^R$ and $r_{\rm in}^R$, respectively. By
linearity,
\begin{align*}
\nu=\nu^R+r^R.
\end{align*}
By \Cref{ThmPulled}, for every $a\in\R$,
\begin{align*}
\lim_{t\to+\infty}\sup_{z\in[a,+\infty)}|\nu^R(t,z)|=0.
\end{align*}
On the other hand, the maximum principle gives the
$L^\infty$-contraction
\begin{align*}
\norm{r^R(t)}_\infty
\leq
\norm{r_{\rm in}^R}_\infty
\leq
\sup_{z>R}|\nu_{\rm in}(z)|.
\end{align*}
Therefore,
\begin{align*}
\limsup_{t\to+\infty}
\sup_{z\in[a,+\infty)}
|\nu(t,z)|
\leq
\sup_{z>R}|\nu_{\rm in}(z)|.
\end{align*}
Finally, since $\nu_{\rm in}(z)\to 0$ as $z\to+\infty$,
we have
\begin{align*}
\lim_{R\to+\infty}\sup_{z>R}|\nu_{\rm in}(z)|=0.
\end{align*}
Letting $R\to+\infty$ in the preceding inequality yields
\begin{align*}
\lim_{t\to+\infty}\sup_{z\in[a,+\infty)}|\nu(t,z)|=0,
\end{align*}
which concludes the proof.
\end{proof}

We now prove \Cref{ThmPulled}.

\begin{proof}[Proof of \Cref{ThmPulled}]
By approximation and parabolic regularization, it is enough to establish
the following estimates for smooth solutions.
Consider $ L:= -\dzz-\beta(z)\dz$ in the weighted space $L^2(e^Vdz)$ with domain $D({L})=H^2(e^Vdz)$, where $V$ is given by \eqref{defV}. As in the preceding subsection, $ L$ is self-adjoint and satisfies for $f\in D( L)$,
\begin{align*}
&\int_\R f ( L f) e^V dz\\
=&
\int_\R f (- f'' -\beta  f')e^V \dz \\
=& \int_\R \left(f'^2+\beta f f'  -\beta f f' \right)e^Vdz \\
=&\norm{f'}_{L^2(e^V dz)}^2.
\end{align*}
Consequently,
\begin{align}\label{DissipRate}
\frac{d}{dt}\left(\frac{1}{2}\norm{\nu}_{L^2(e^V dz)}^2 \right)= - \norm{\dz \nu}_{L^2(e^V dz)}^2 .
\end{align}

Set $p:=\dz\nu$. Differentiating \eqref{EntEqn} with respect to $z$
gives
\begin{align*}
\dt p-\dzz p-\beta\dz p=\beta'p.
\end{align*}
Taking the scalar product with $p$ in $L^2(e^Vdz)$ yields
\begin{align}
\label{DissipRateDerivative}
\frac12\frac{d}{dt}
\norm{\dz\nu(t)}_{L^2(e^Vdz)}^2
=
-\norm{\dzz\nu(t)}_{L^2(e^Vdz)}^2
+
\int_\R
\beta'(\dz\nu(t))^2e^Vdz.
\end{align}
Since
$\beta'\leq K$, combining \eqref{DissipRate} and
\eqref{DissipRateDerivative} gives
\begin{align*}
\frac{d}{dt}\left(\frac{K}{2}\norm{ \nu }_{L^2( e^V  dz)}^2+\frac{1}{2}\norm{\dz  \nu }_{L^2( e^V  dz)}^2 \right)
\leq - \norm{\dzz  \nu }_{L^2( e^V dz )}^2.
\end{align*}
Thus, the quantity on the left is nonincreasing. Moreover,
\eqref{DissipRate} shows that
$\norm{\nu(t)}_{L^2(e^Vdz)}^2$ is nonincreasing. It follows that
$\norm{\dz\nu(t)}_{L^2(e^Vdz)}^2$ admits a limit as
$t\to+\infty$. On the other hand, integrating \eqref{DissipRate} in
time gives
\begin{align*}
\int_0^{+\infty}
\norm{\dz\nu(t)}_{L^2(e^Vdz)}^2dt
\leq
\frac12
\norm{\nu_{\rm in}}_{L^2(e^Vdz)}^2.
\end{align*}
Consequently,
\begin{align}
\label{DerivativeDecayPulled}
\lim_{t\to+\infty}
\norm{\dz\nu(t)}_{L^2(e^Vdz)} = 0.
\end{align}

We now upgrade this convergence to a weighted uniform estimate. Set
\begin{align*}
c:=\int_\R\beta_-(z)dz,
\qquad
W(z):=\int_0^z\beta_+(y)dy-c,
\end{align*}
where $\beta_+:=\max\{\beta,0\}$. Since $V(z)=\int_0^z\beta(y)dy$, we have for any $z\in \R$
\begin{align*}
V(z)-2c \leq W(z)\leq V(z).
\end{align*}
For each $t>0$, the function $
z\mapsto\nu(t,z)^2e^{W(z)}$ belongs to $W^{1,1}(\R)$ and tends to $0$ at both ends of the real
line. Therefore,
\begin{align*}
e^{-2c}\norm{ \nu (t)^2 e^V }_{\infty}&\leq \norm{ \nu (t)^2e^W}_{\infty}\\
&\leq  \int_\R \left|\dz\left( \nu ^2 e^W\right)\right|dz\\
    & \leq 2\int_\R | \nu ||\dz  \nu | e^W dz + \int_\R \left(e^W\right)' \nu ^2dz \qquad \text{, since }(e^W)'\geq 0\\
& = 2\int_\R | \nu ||\dz  \nu | e^W dz -2\int_\R  \nu \dz  \nu  e^W dz\\
& \leq 4\int_\R | \nu ||\dz  \nu | e^W dz \\
&\leq 4 \norm{  \nu (t)}_{L^2(e^W dz)}\norm{\dz  \nu (t)}_{L^2(e^W dz)} \\
&\leq 4 \norm{  \nu (t)}_{L^2( e^V  dz)}\norm{\dz  \nu (t)}_{L^2( e^V  dz)},
\end{align*}
where at the end we used Cauchy-Schwarz inequality and 
$W\leq V$. Hence, 
\begin{align*}
    \lim_{t\to \infty} \sup_{z\in \R} | \nu (t,z)^2 e^V (z)| = 0.
\end{align*}
Finally, since $\int_\R \beta_-(z)dz<+\infty$, for every $a\in \R$, we have that \\
$\inf_{z\in[a,+\infty)}e^{V(z)}>0$ and therefore 
\begin{align*}
  \lim_{t\to+\infty}\sup_{z\in [a,+\infty)}|\nu(t,z)|=0 .
\end{align*}
\end{proof}

\section{Asymptotic Spreading Properties}
\label{Sect-Asympt}

In this Section, we work under the hypothesis that the solution $(\r,N)$ of System \eqref{diffusivemodel:main} is well-defined for all time, that $x \mapsto N(t,x)$ is an increasing function and we recall that $\bar{x}(t)$ is defined as the unique solution $N(t,\bar{x}(t))=N_\thresh$. As before, we consider the system in the moving frame of reference and study $(\r(t,z),N(t,z))$ solution of System \eqref{movingmodel:main}, droping the diacritical $\tilde{ }$ for the sake of concision. 

The aim of the Section is to describe the asymptotic behavior of $t\mapsto \bar{x}(t)$. In fact, we will show that under some assumptions on the initial datum, we have $\liminf_{t\to +\infty} \dot{\bar{x}}(t)\leq \s^* $ and $\limsup_{t\to +\infty} \dot{\bar{x}}(t)\geq \s^* $. The strategy of proof for both results is similar and is based on an argument by contradiction: if those properties were not satisfied, then this would first lead to an abnormal behavior of $\r$, which in turn contradicts the equation $N(t,0)=N_\thresh$. We have not been able to prove a stronger result, such as for instance $\lim_{t\to+\infty} \dot{\bar{x}}(t)=\s^* $ and we believe that in order to achieve such a result, one needs to study the behavior of $\r$ and $N$ simultaneously, which is much more involved than our present study. One major difficulty comes from the fact that System \eqref{diffusivemodel:main} does not have a comparison principle.

\subsection{The Spreading may not be too fast}

\begin{thm}
\label{ThmTransientSlow}
Suppose that $\dot{\bar{x}}\in L^\infty(\R_+)$ and that
\begin{align*}
 \frac{\r_{\rm in}}{\r^{\s^*}} \in  L^\infty.
\end{align*} 
Then $\liminf_{t\to +\infty} \dot{\bar{x}}(t)\leq \s^* $.
\end{thm}
\begin{proof}
We argue by contradiction and assume that there exist $t_0\geq 0, \delta>0$, such that for all $t\geq t_0, \dot{\bar{x}}(t)\geq  \s^*+\delta$. 

\textit{Step 1: $\r$ converges to $0$ uniformly on intervals of the form $[a,+\infty)$ for $a\in \R$.}

\medskip

\textit{Step 1. (a): We may suppose that $t_0=0$.}

We show that
\begin{align*}
\frac{\r(t_0,\cdot)}{\r^{\s^*}}\in L^\infty(\R).
\end{align*}
Set
\begin{align*}
v^{\s^*}(t,z):=\frac{\r(t,z)}{\r^{\s^*}(z)}.
\end{align*}
Then, $v^{\s^*}$ satisfies
\begin{align*}
\dt v^{\s^*}-\dzz v^{\s^*}-\beta_*(t,z)\dz v^{\s^*}-\gamma_*(t,z)v^{\s^*}=0,
\end{align*}
where
\begin{align*}
\beta_*(t,z)
:=
\dot{\bar{x}}(t)-\chi\mathbbm{1}_{z<0}
+2\frac{\dz\r^{\s^*}}{\r^{\s^*}},\qquad
\gamma_*(t,z)
:=
\left(\dot{\bar{x}}(t)-\s^*\right)
\frac{\dz\r^{\s^*}}{\r^{\s^*}}.
\end{align*}
Since $\nicefrac{\dz\r^{\s^*}}{\r^{\s^*}}\leq0$, we have
\begin{align*}
\gamma_*(t,z)
\leq
\left\|
\frac{\dz\r^{\s^*}}{\r^{\s^*}}
\right\|_\infty
\big(\s^*-\dot{\bar{x}}(t)\big)_+.
\end{align*}
Consequently, by the comparison principle, for $t>0$,
\begin{align*}
\left\|
\frac{\r(t)}{\r^{\s^*}}
\right\|_\infty \leq 
\left\|
\frac{\r_{\rm in}}{\r^{\s^*}}
\right\|_\infty
\exp\left(
\left\|
\frac{\dz\r^{\s^*}}{\r^{\s^*}}
\right\|_\infty
\int_0^t
\big(\s^*-\dot{\bar{x}}(s)\big)_+ ds
\right).
\end{align*}
After translating time, we may therefore assume without loss of generality that $t_0=0$, while retaining the assumption $\nicefrac{\r_{\rm in}}{\r^{\s^*}}\in L^\infty(\R)$.

\medskip

\textit{Step 1. (b): Construction of an upper bound for
$\nicefrac{\r}{\r^{\s^*+\delta}}$.} 

Set
\begin{align*}
v^{\s^*+\delta}(t,z):=\frac{\r(t,z)}{\r^{\s^*+\delta}(z)}.
\end{align*}
Similar to before, $v^{\s^*+\delta}$ satisfies
\begin{align*}
\dt v^{\s^*+\delta}-\dzz v^{\s^*+\delta}-\beta(t,z)\dz v^{\s^*+\delta}-\gamma(t,z)v^{\s^*+\delta}=0,
\end{align*}
where
\begin{align*}
\beta(t,z)
:=
\dot{\bar{x}}(t)-\chi\mathbbm{1}_{z<0}
+2\frac{\dz\r^{\s^*+\delta}}
{\r^{\s^*+\delta}},\qquad
\gamma(t,z)
:=
\left(\dot{\bar{x}}(t)-(\s^*+\delta)\right)
\frac{\dz\r^{\s^*+\delta}}{\r^{\s^*+\delta}}.
\end{align*}
Since $\dot{\bar{x}}(t)\geq\s^*+\delta$ and
$\dz\r^{\s^*+\delta}\leq0$, we have $\gamma \leq 0$. Hence, $v^{\s^*+\delta}$ is a subsolution of
\begin{align}
\label{EqBarV}
\dt u-\dzz u-\beta(t,z)\dz u=0.
\end{align}

Let
\begin{align*}
\eta := \frac{\mu_+(\s^*)-\mu_-(\s^*+\delta)}{2}>0, \qquad
g(z):=\frac{e^{-\eta z}}{1+e^{-\eta z}}.\end{align*}
Since for $z<0, \frac{\r^{\s^*}(z)}{\r^{\s^*+\delta}(z)} =1$ and by \Cref{thmparwave}, for $z\to+\infty$, 
\begin{align*}
    \frac{\r^{\s^*}(z)}{\r^{\s^*+\delta}(z)} = O\left(ze^{-(\mu_+(\s^*)-\mu_-(\s^*+\delta))z}\right)=O(g(z)),
\end{align*}
we can find a constant $K$ such that
\begin{align*}
    v^{\s^*+\delta}_{\rm in}(z) \leq K g(z).
\end{align*}

Let $\bar v$ denote the solution of \eqref{EqBarV} with initial datum $\bar v_{\rm in}(z)=Kg(z)$. 
By the comparison principle,
\begin{align}
\label{vBelowBarV}
v^{\s^*+\delta}\leq\bar v.
\end{align}

\medskip

\textit{Step 1. (c): $\dz \bar  v  \leq 0$.}

By construction, $\dz \bar  v _{\rm in}\leq 0$. 
Set $w:=\dz\bar v$. Differentiating \eqref{EqBarV} gives, in
distributional form,
\begin{align*}
\dt w-\dzz w-\dz(\beta w)=0.
\end{align*}
Proceeding as in the proof of \Cref{cor-existence}, we obtain
\begin{align*}
\frac12\frac{d}{dt}\int_\R w_+^2 dz
=
-\int_\R(\dz w_+)^2 dz
-\int_\R\beta w_+\dz w_+ dz\leq
\frac{\norm{\beta}_\infty^2}{2}
\int_\R w_+^2 dz.
\end{align*}
Since $w_+(0,\cdot)\equiv0$, Grönwall's lemma yields for all $t>0$, $w_+(t,\cdot)\equiv0$ and therefore $\dz\bar v\leq0.$

\medskip

\textit{Step 1. (d): $\bar v(t)$ converges to $0$ uniformly on every interval of the form $[a,+\infty)$.}

Define
\begin{align*}
\beta^{\s^*+\delta}(z)
:=
\s^*+\delta-\chi\mathbbm{1}_{z<0}
+
2\frac{\dz\r^{\s^*+\delta}}
{\r^{\s^*+\delta}},
\end{align*}
which satisfies for $t\geq 0,z\in \R$,
\begin{align*}
\beta^{\s^*+\delta}(z) \leq \beta(t,z).
\end{align*}
Together with $\dz\bar v(t,z)\leq0$, this gives
\begin{align*}
\dt\bar v-\dzz\bar v -\beta^{\s^*+\delta}(z)\dz\bar v
=\left(\beta(t,z)-\beta^{\s^*+\delta}(z)\right)
\dz\bar v \leq0.
\end{align*}
Thus $\bar v$ is a subsolution of
\begin{align*}
\dt u-\dzz u
-\beta^{\s^*+\delta}(z)\dz u=0
\end{align*}
with initial datum $Kg$.

Since $g\in L^\infty(\R)$ and $g(z)\to 0$ as $z\to+\infty$, \Cref{cor-pulled} implies that the corresponding solution $u$ satisfies, for every $a\in\R$,
\begin{align*}
\lim_{t\to+\infty}
\sup_{z\in[a,+\infty)}
u(t,z)=0.
\end{align*}
By the comparison principle,
\begin{align*}
0\leq v^{\s^*+\delta}(t,z)\leq\bar v(t,z)\leq u(t,z).
\end{align*}
Consequently, for every $a\in\R$,
\begin{align*}
\lim_{t\to+\infty}
\sup_{z\in[a,+\infty)}
v^{\s^*+\delta}(t,z)=0,
\end{align*}
and since $\r^{\s^*+\delta}$ is bounded on every interval
$[a,+\infty)$, we conclude that 
\begin{align*}
\lim_{t\to+\infty}
\sup_{z\in[a,+\infty)}
\r(t,z)=0,
\end{align*}

\begin{figure}
\begin{center}
\includegraphics[width=.65\textwidth]{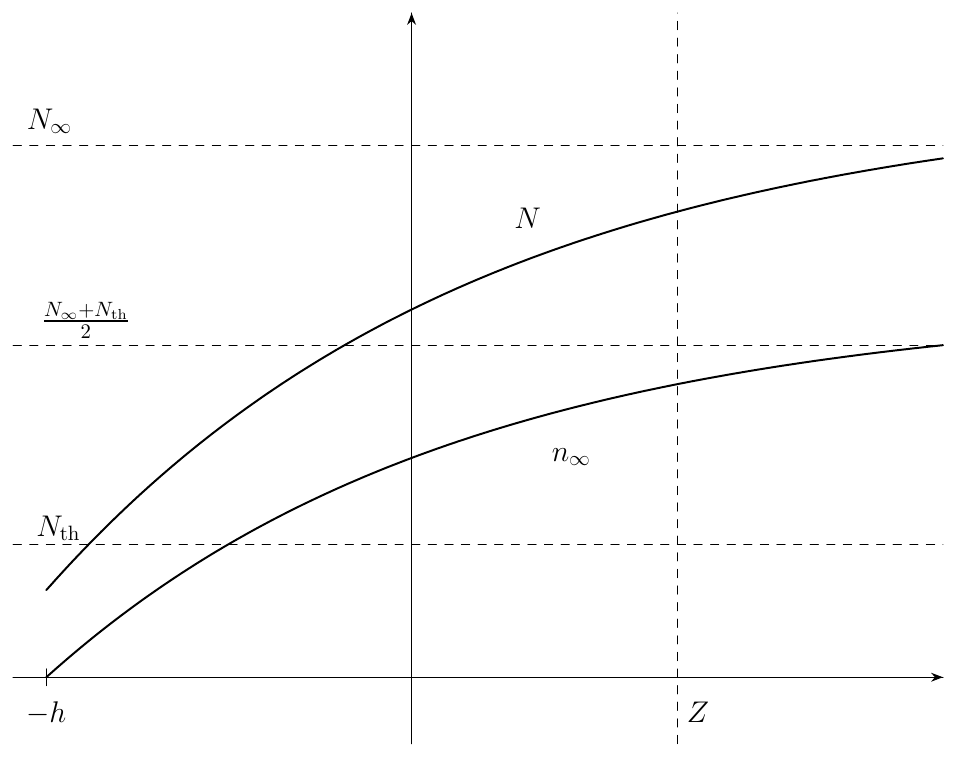}
\caption{In order to prove that condition $N(t,0)=N_\thresh$ cannot be satisfied, we consider solution $n(t,z)$ to \eqref{EqnNFast} with inital datum $n(t_1,z)=n_\infty(z)$ for $z\geq Z$ and $n(t_1,z)=0$ for $z\in[-h,Z)$. Then $n(t)\to n_\infty$ uniformly. Furthermore $N$ is a super-solution of \eqref{EqnNFast} and $N(t_1,\cdot)> n(t_1,\cdot) $, which leads to $N(t,\cdot) \geq n(t,\cdot)$ for all $t\geq t_1$. But $ n(t,0)\to n_\infty(0)>N_\thresh$ and thus the condition $N(t,0)=N_\thresh$ cannot be satisfied for all time.}
\end{center}
\end{figure}

\medskip

\textit{Step 2: Uniform convergence of $\r$ to $0$ uniformly on sets of the form $[a,+\infty)$ contradicts the condition $N(t,0)=N_\thresh$.} 

We start by introducing an array of parameters. Let $\mu>0$ and $\e>0$ such that
\begin{align}
\mu<\min\left(\frac{\s^*}{D},\mu_-(\s^*+\delta)\right),
\qquad
\label{choiceB}
\frac{N_\init+N_\thresh}{2}-\underbrace{\frac{\e}{D\mu\left(\frac{\s^*}{D}-\mu\right)}}_{B:=}>N_\thresh.
\end{align}
Then, choose $h>0$ big enough such that
\begin{align}
\label{choiceh}
\frac{N_\init+N_\thresh}{2}\left(1-e^{-\frac{\s^* h}{D}}\right)-B\left(1-e^{-\left(\frac{\s^*}{D}-\mu\right)h}\right)>N_\thresh,
\end{align}
which exists by noticing that in the limit $h\to +\infty$ the inequality is satisfied by the above. 

As $v^{\s^*+\delta}$ converges to $0$ uniformly on the set $[-h,+\infty)$, and by noticing that $\r^{\s^*+\delta}(z)=o\left(e^{-\mu z}\right)$, there exists $t_1$, such that for $t\geq t_1$ and $z\geq -h$, we have
\begin{align*}
 \r(t,z)N(t,z) \leq   v(t,z)\r^{\s^*+\delta}(z)N_\init \leq  \e e^{-\mu z}.
\end{align*}
Finally, let $Z>0$ be such that 
\begin{align*}
    N(t_1,Z)>\frac{N_\init+N_\thresh}{2},
\end{align*}
which exists as $\lim_{z\to+\infty}N(t,z)=N_\init>N_\thresh$.

\medskip

\textit{Step 2. (a): Construction of $n$ a lower bound of $N$.} 

We consider the parabolic equation in the domain $\Omega=(t_1,+\infty)\times(-h,+\infty)$
\begin{empheq}[left={\empheqlbrace\,}]{align}
\label{EqnNFast}
&\dt n- \s^*\dz n -D\dzz n +\e e^{-\mu z}=0\\
&n(t,-h)=0, n(t_1,\cdot)=n_{t_1}(\cdot), \nonumber
\end{empheq}
where the initial datum $n_{t_1}$ will be described below.

$N$ is a supersolution for \eqref{EqnNFast}. Indeed for $(t,z)\in \Omega$, using that $\dot\bar{x}(t)\geq \s^*$ and $\partial_z N\geq 0$,
\begin{align*}
\dt N-\s^* \dz N - D\dzz N + \e e^{-\mu z}
\geq & \dt N-\dot{\bar{x}}(t) \dz N - D\dzz N +   \r N
=  0.
\end{align*}
We will verify below that $N(t_1,\cdot)\geq n_{t_1}(\cdot)$. Under that condition and since $N(\cdot,-h)>0=n(\cdot,-h)$, the comparison principle will then yield $N(t,z)\geq n(t,z)$ for all $(t,z)\in \Omega$.

\medskip

\textit{Step 2. (b): Convergence of $n\to n_\infty$, where $n_\infty$ satisfies $n_\infty(0)>N_\thresh$.} 

By the specific choice of $B$ given by \eqref{choiceB}, we have that
\begin{align*}
n_\infty(z):=\frac{N_\init+N_\thresh}{2}-Be^{-\mu z}+Ce^{-\frac{\s^*}{D}z},
\end{align*}
is a stationary solution of \eqref{EqnNFast}. Moreover, setting $C:=Be^{-\left(\frac{\s^*}{D}-\mu	\right)h} -\frac{N_\init+N_\thresh}{2}e^{-\frac{\s^*}{D}h}   $, leads $n_\infty(-h)=0$. Finally, observe that
\begin{align*}
n_\infty(0)=\frac{N_\init+N_\thresh}{2}\left(1-e^{-\frac{\s^* h}{D}}\right)-B\left(1-e^{-\left(\frac{\s^*}{D}-\mu\right)h}\right)>N_\thresh ,
\end{align*}
as a consequence of the choice of $h$ given by \eqref{choiceh}. \\
Consider now the solution $n$ to \eqref{EqnNFast} with initial datum $n(t_1,\cdot)=n_{t_1}(\cdot)=n_\infty(\cdot) \mathbbm{1}_{\cdot\geq Z}$.

We will show that $n(t,\cdot)\to n_\infty$ uniformly and in particular for $z=0$, $n(t,0)\to n_\infty(0)>N_\thresh$. \\
Indeed set $w:=n_\infty-n$, then $w$ satisfies for $(t,z) \in \Omega$:
\begin{empheq}[left={\empheqlbrace\,}]{align*}
&\dt w - \s^*\dz w -D\dzz w=0\\
&w(t,-h)=0, w(t_1,z)=n_\infty(z)\mathbbm{1}_{-h<z<Z}.
\end{empheq}
It can be verified that for $(t,z)\in \Omega$, we have the following explicit expression:
\begin{align*}
&|w(t,z)|\\
=&\left|\frac{e^{-\frac{{\s^*}^2}{4}t-\frac{\s^*}{2}z}}{\sqrt{4\pi D(t-t_1)}}\int_{-h}^{+\infty}\left(e^{-\frac{(z+h-y)^2}{4D(t-t_1)}}-e^{-\frac{(z+h+y)^2}{4D(t-t_1)}}\right)e^{\frac{\s^*}{2}y}n_\infty(y)\mathbbm{1}_{y\leq Z}dy\right|\\
\leq & \frac{Ce^{\frac{\s^*}{2}h}e^{-\frac{{\s^*}^2}{4}t}}{\sqrt{(t-t_1)}} \norm{e^{\frac{\s^*}{2}y}n_\infty(y)}_{L^1([-h,Z])}.
\end{align*}
Hence, $\lim_{t\to +\infty} w(t,z)=0$ uniformly, which is equivalent to establishing that $\lim_{t\to+\infty} n(t,z)=n_\infty(z)$ uniformly. In particular there exist $t_2>t_1$, such that $n(t_2,0)>N_\thresh$.

Finally, let us verify that $n_{t_1}$ satisfies the ordering $N(t_1,\cdot)\geq n_{t_1}$ as claimed at the end of \textit{Step 2. (a)}. Indeed, $n_{t_1}=n_\infty \mathbbm{1}_{\cdot \geq Z}$ and for $z>Z$:
\begin{align*}
n_\infty(z)\leq \frac{N_\init+N_\thresh}{2}<N(t_1,Z) \leq N(t_1,z),
\end{align*}	
as $N$ is increasing. Hence $ N(t_1,\cdot)> n_{t_1}$ and thus $N>n$. In particular,
\begin{align*}
   N(t_2,0)\geq n(t_2,0)>N_\thresh, 
\end{align*}
which leads to a contradiction.

\end{proof}

\subsection{The Spreading may not be too slow}

In this Section, we will show that $\limsup_{t \to +\infty} \dot{\bar{x}}(t) \geq \s^*$. The proof is based on the observation that the norm of $\r$ increases exponentially in a certain weighted $L^1$-spaces, when $\dot{\bar{x}}<\s^*$. In fact, in the large bias case $\chi>1$, take for example the weight $e^{\frac{z}{\chi}}$. By noticing that it is an eigenfunction of the dual of the elliptic operator, one sees that:
\begin{align*}
\frac{d}{dt}\int_\R e^{\frac{z}{\chi}}\r(t,z)dz = \frac{\chi+\frac{1}{\chi}-\dot{\bar{x}}(t)}{\chi} \int_\R e^{\frac{z}{\chi}}\r(t,z)dz.
\end{align*} 
Then if $\limsup_{t\to+\infty}\dot{\bar{x}}(t)<\s^*$, we have that $\int_\R e^{\frac{z}{\chi}}\r(t,z)dz$ grows exponentially. Nevertheless, since $e^{\frac{z}{\chi}}$ is unbounded, this observation is not sufficiently instructive. Hence, instead of $e^{\frac{z}{\chi}}$ an exact eigenfunction of the dual of the elliptic operator, we consider a supersolution $e^u\in L^\infty(\R)$ of the dual of the elliptic operator, which enables us to show that the $L^1(\R,e^u dz)$-norm of $\r(t)$ diverges to $+\infty$. Then we show that a similar statement remains true for $\r$ in $L^1(\R_+,e^u dz)$ and finally this accumulation of mass on the half-line $\R_+$ leads to a contradiction on the condition $N(t,0)=N_\thresh$. 

\begin{thm}
\label{thm-slow-strong-bias}
Suppose that $\dot{\bar{x}}\in L^\infty(\R_+)$ and that $\r_{\rm in}\in L^\infty$. Then, we have that $\limsup_{t \to +\infty} \dot{\bar{x}}(t) \geq \s^*$.
\end{thm}

\begin{proof}
We argue by contradiction and assume that there exist $t_0\geq 0, \delta>0$, such that for all $t\geq t_0$,     $\dot{\bar{x}}(t)\leq  \s^*-\delta$. 

\medskip

\textit{Step 1: Construction of a function $e^u\in L^\infty(\R)$, such that $L^1(e^udz)$-norm of $\r$ tends to infinity.} 

Given a function $u\in W^{2,\infty}(\R)$, set $w(t,z):=e^{u(z)}\r(t,z)$, then:
\begin{subequations}\label{EQN-w:main}
\begin{empheq}[left={\empheqlbrace\,}]{align}
\label{EQN-w:a}
&\dt w -\dzz w-\dz\left(\beta(t,z) w \right) - \gamma(t,z)w=0 ,\\ 
&\beta(t,z) =\dot{\bar{x}}(t) - 2u'-	\chi \mathbbm{1}_{z\leq 0}, \label{EQN-w:b}\\ 
& \gamma(t,z) =u'^2+u''-\dot{\bar{x}}(t)u' + \mathbbm{1}_{z> 0} + \chi \mathbbm{1}_{z\leq 0} u'. \label{EQN-w:c}
\end{empheq}
\end{subequations}

\medskip

\textit{Step 1. (a): Construction of $u$ such that $\gamma$ given by \eqref{EQN-w:c} satisfies $\gamma\geq \eta>0$.}

Set
\begin{align*}
M:=\max\left\{2,\norm{\dot{\bar{x}}}_\infty\right\}.
\end{align*}
We seek $u$ of the form
\begin{align*}
u(z)=\mu(z)z,
\end{align*}
where $\mu(z)=\nicefrac{1}{\chi}$ for $z\leq0$, $\mu(z)=-1/M$ for sufficiently
large positive $z$, and $\mu$ decreases sufficiently slowly between
these two values. The construction is based on Bounds \eqref{AuxBd1},
\eqref{AuxBd2}, and \Cref{lemma-techical}. 

First, for $t\geq t_0$,
\begin{align}
\label{AuxBd1}
\left(\frac{1}{\chi}\right)^2-\dot{\bar{x}}(t)\left(\frac{1}{\chi}\right)+\chi \left(\frac{1}{\chi}\right)=\frac{\chi+\frac{1}{\chi}-\dot{\bar{x}}(t)}{\chi}\geq	\frac{\delta}{\chi}.
\end{align}
We choose $\eta>0$ sufficiently small so that
\begin{align*}
2\eta<\frac{\delta}{\chi}.
\end{align*}

Next, let
\begin{align*}
I:=
\left[-\frac1M,\frac1\chi\right],
\qquad
g_t(\mu):=\mu^2-\dot{\bar{x}}(t)\mu+1,
\text{ for }t\geq t_0, \mu\in I.
\end{align*}
The quantity $g_t(\mu)$ is precisely the
value of $\gamma$ on the positive half-line when $u(z)=\mu z$. We now show that $g_t$ is uniformly bounded from below by a positive constant on $I$.

The minimum of $g_t$ on $\R$ is attained at $ \mu=\nicefrac{\dot{\bar{x}}(t)}2$. 
If $\nicefrac{\dot{\bar{x}}(t)}2<-\nicefrac1M$, then the minimum of $g_t$ on $I$ is attained at the left endpoint and 
\begin{align*}
    g_t\left( -\frac{1}{M}\right)=\frac{1}{M^2}+\frac{\dot{\bar{x}}}{M}+1\geq \frac{1}{M^2}.
\end{align*}
If $\nicefrac{\dot{\bar{x}}(t)}{2}>\nicefrac{1}{\chi}$, then the minimum on $I$ is attained at the right endpoint and,
by \eqref{AuxBd1},
\begin{align*}
    g_t\left( \frac{1}{\chi}\right)
\geq
\frac{\delta}{\chi}.
\end{align*}
If $\nicefrac{\dot{\bar{x}}(t)}{2}\in I$, then $\min_{\mu\in I}g_t(\mu) =1-\nicefrac{\dot{\bar{x}}(t)^2}{4} $. In the large bias case $\chi>1$, we have $\left|\nicefrac{\dot{\bar{x}}(t)}2\right|
\leq
\max\left\{\nicefrac1M,\nicefrac1\chi\right\}<1$ and therefore 
\begin{align*}
g_t(\mu)
\geq
1-
\max\left\{\frac1{M^2},\frac1{\chi^2}\right\}>0.\end{align*}
In the small bias case $\chi\leq 1$, as $\s^*=2$, we have that $\dot{\bar{x}}(t)\leq 2-\delta$, which implies that $\nicefrac{\dot{\bar{x}}(t)}{2}\leq 1-\nicefrac{\delta}{2}$. Moreover since $\nicefrac{\dot{\bar{x}}(t)}{2}\in I$, we also have $\nicefrac{\dot{\bar{x}}(t)}{2}\geq  -\nicefrac{1}{M}$. Thus,  $\left|\nicefrac{\dot{\bar{x}}(t)}{2}\right|\leq  \max\left\{1-\nicefrac{\delta}{2},\nicefrac{1}{M}\right\} <1$ and the minimum of $g_t$ is again
 uniformly positive.\\
 Decreasing $\eta>0$ if necessary, we therefore obtain, for every
$t\geq t_0$,
\begin{align}
\label{AuxBd2}
\min_{\mu\in I}g_t(\mu)\geq2\eta.
\end{align}

We now introduce a technical lemma which allows us to interpolate slowly between the two endpoints of $I$.

\begin{lemma}\label{lemma-techical}
For every $\e\in \left(0,\nicefrac 12\right)$, there exists a nondecreasing function $\theta_\e\in W^{2,\infty}([0,1]; [0,1])$ such that 
\begin{align*}
\theta_\e(0)=0,\qquad
\theta_\e'(0)=0,\qquad
\theta_\e(1)=1,\qquad
\theta_\e'(1)=0,
\end{align*}and
\begin{align*}
\sup_{z\in[0,1]} \theta_\e'(z)z \leq \e.
\end{align*}
\end{lemma}

We momentarily postpone the proof of \Cref{lemma-techical}.
Using the function $\theta_\e$ from \Cref{lemma-techical}, for $\e,B>0$ to be determined later, 
set
\begin{align*}
\mu: z \mapsto \left\{ \begin{array}{ll}
\frac{1}{\chi} & \text{, if }z\leq 0, \\
\frac{1}{\chi}-\theta_\e\left(\frac{z}{B}\right)\left(\frac{1}{\chi}+\frac{ 1}{M}\right) & \text{, if }z\in(0,B], \\
-\frac{1}{M} & \text{, if }z>B, \\
\end{array}\right..
\end{align*}
and define $u(z):=\mu(z)z$. 

For $z\leq0$, we have $u'=1/\chi$ and $u''=0$. Thus, by
\eqref{AuxBd1},
\begin{align*}
\gamma(t,z)
=\left(\frac{1}{\chi}\right)^2-\dot{\bar{x}}(t)\left(\frac{1}{\chi}\right)+\chi \left(\frac{1}{\chi}\right)\geq \frac{\delta}{\chi}\geq 2\eta .
\end{align*}
For $z>B$, we have $u'=-1/M$ and $u''=0$, and therefore, by
\eqref{AuxBd2},
\begin{align*}
\gamma(t,z)
=g_t\left(-\frac1M\right)
\geq
2\eta.
\end{align*}
It remains to consider $z\in(0,B)$. In this interval
$\mu(z)\in I$, and
\begin{align*}
u'(z)=\mu(z)+z\mu'(z)=
\mu(z)-
\frac zB\theta_\e'\left(\frac zB\right)
\left(\frac1\chi+\frac1M\right).
\end{align*}
By \Cref{lemma-techical},
\begin{align*}
\sup_{z\in[0,B]}
\frac zB\theta_\e'\left(\frac zB\right)
\leq\e,
\end{align*}
and hence
\begin{align*}
\sup_{z\in[0,B]}|u'(z)-\mu(z)|
\leq
\e\left(\frac1\chi+\frac1M\right).
\end{align*}
Moreover,
\begin{align*}
u''(z)=
-\frac1B
\left(\frac1\chi+\frac1M\right)
\left[
\frac zB\theta_\e''\left(\frac zB\right)
+
2\theta_\e'\left(\frac zB\right)
\right].
\end{align*}
Thus, once $\e$ has been fixed,
\begin{align*}
\norm{u''}_{L^\infty(0,B)}
\leq
\frac{C_\e}{B},
\end{align*}
for some $C_\e>0$ independent of $B$.

Finally, on $(0,B)$,
\begin{align*}
\gamma(t,z)-g_t(\mu(z))=
\left(u'(z)-\mu(z)\right)
\left(u'(z)+\mu(z)-\dot{\bar{x}}(t)\right)+u''(z).
\end{align*}
Since $u'$, $\mu$, and $\dot{\bar{x}}$ are uniformly bounded, we may first choose $\e>0$ sufficiently small and then $B>0$ sufficiently large so that
\begin{align*}
\left| \gamma(t,z)-g_t(\mu(z)) \right| \leq\eta
\end{align*}
for every $t\geq t_0$ and $z\in(0,B)$. By \eqref{AuxBd2},
\begin{align*}
\gamma(t,z)\geq\eta.
\end{align*}

We have therefore constructed $u$ such that
\begin{align*}
\gamma(t,z)\geq\eta>0
\end{align*}
for every $t\geq t_0$ and almost every $z\in\R$.

\medskip

\textit{Step 1. (b): Construction of $\omega$ such that $\omega e^{\eta (t-t_0)}$ is a lower bound of $w=e^u\r$ for $t\geq t_0$.} 

Let $\omega$ be the solution, for $t\geq t_0$, of
\begin{empheq}[left={\empheqlbrace\,}]{align}
\label{EQN-omega}
&\dt \omega -\dzz \omega -\dz\left(\beta(t,z)\omega \right)=0\\
&\omega(t_0,\cdot)=e^{u(\cdot)}\r(t_0,\cdot). \nonumber
\end{empheq}
Since $\gamma\geq\eta$, we have that $\omega(t,z)e^{\eta (t-t_0)}$ is a subsolution of \eqref{EQN-w:a}, which implies that for $t\geq t_0, z\in \R$:
\begin{align}
\label{wOmegaBd}
w(t,z)=e^{u(z)}\r(t,z)\geq \omega(t,z)e^{\eta (t-t_0)}
\end{align}
By construction of $u$, we have that $\omega(t_0,\cdot)\in L^1(\R)$. Moreover, \eqref{EQN-omega} is under conservative form and the mass of $\omega$ is conserved. Combining the conservation of mass of $\omega(t)$ with \eqref{wOmegaBd}, we obtain for $t\geq t_0$
\begin{align*}
\int_\R e^{u(z)}\r(t,z)dz\geq \int \omega(t,z)e^{\eta (t-t_0)} dz= e^{\eta (t-t_0)} \int \omega(t_0,z) dz .
\end{align*}
In conclusion, the $L^1(e^udz)$-norm of $\r$ tends to infinity.

\medskip

\textit{Step 2: A uniformly positive amount of the mass of $\omega$ reaches $\R_+$.} 

We show that
\begin{align}
\label{OmegaPositiveMass}
\liminf_{t\to+\infty}
\int_{t-4}^t\int_{\R_+}\omega(s,z) dz ds>0.
\end{align}

Set
\begin{align*}
M_0:=\int_\R\omega(t_0,z)dz,
\end{align*}
and by the conservation of mass, we have that for all $t\geq t_0, M_0=\int_\R\omega(t,z)dz$.

\medskip

\textit{Step 2. (a): There exists $h>0$ such that the mass of $\omega$ in $(-\infty,-h]$ is at most $\nicefrac{M_0}{2}$.} 

Set
\begin{align*}
I(t):=\int_{\R_-} \left(e^{-\frac{z}{2\chi}}-1+\frac{z}{2\chi} \right)\omega(t,z) dz.
\end{align*}
For $z<0$, by construction of $u$,
\begin{align*}
    \beta(t,z)=\dot{\bar{x}}(t)-\frac{2}{\chi}-\chi=\dot{\bar{x}}(t) - \left(\chi+\frac{1}{\chi}\right) - \frac{1}{\chi}\leq\dot{\bar{x}}(t) - \sigma^* - \frac{1}{\chi}\leq -\delta	- \frac{1}{\chi}.
\end{align*}
Moreover, $I(t_0)<+\infty$, since on $\R_-$, $\omega(t_0,z)=e^{z/\chi}\r(t_0,z)$
and $\r(t_0,\cdot)\in L^\infty(\R)$. The following integrations by parts may be justified by first introducing a cutoff at $-\infty$ and then letting the cutoff tend to infinity.

Using \eqref{EQN-omega}, we obtain
\begin{align*}
&\dot{I}(t)
\\
=&\int_{\R_-} \left(e^{-\frac{z}{2\chi}}-1+\frac{z}{2\chi} \right)(\dzz\omega+\dz\left(\beta\omega\right)) dz \\
=&\left[\left(e^{-\frac{z}{2\chi}}-1+\frac{z}{2\chi} \right)(\dz\omega(t,z)+\beta(t,z)\omega(t,z)) \right]_{z=-\infty}^{z=0}\\
&+\frac{1}{2\chi}\int_{\R_-} \left(e^{-\frac{z}{2\chi}}-1 \right)(\dz\omega+\beta\omega) dz \\
=&\frac{1}{4\chi^2}\int_{\R_-} e^{-\frac{z}{2\chi}}\omega dz +\frac{1}{2\chi}\int_{\R_-} \left(e^{-\frac{z}{2\chi}}-1 \right)\beta\omega dz\\
=& \frac{1}{2\chi}\int_{\R_-}\left(\frac{1}{2\chi}+\beta \right)\left( e^{-\frac{z}{2\chi}}-1\right)\omega dz+ \frac{1}{4\chi^2}\int_{\R_-} \omega dz \\
=&\frac{1}{2\chi}\int_{\R_-}\underbrace{\left(\frac{1}{2\chi}+\beta \right)}_{\leq -\delta	-\frac{1}{2\chi}}\left( e^{-\frac{z}{2\chi}}-1+\frac{z}{2\chi}\right)\omega dz\\
&+ \frac{1}{4\chi^2}\underbrace{\int_{\R_-} \omega dz}_{\leq M_0} \underbrace{ -\frac{1}{2\chi}\int_{\R_-}\left(\frac{1}{2\chi}+\beta \right)\frac{z}{2\chi}\omega dz}_{\leq 0}\\
\leq& -\underbrace{\frac{1}{2\chi}\left(\delta+\frac{1}{2\chi	} \right)}_{a:=}I(t) + \frac{1}{4\chi^2}M_0.
\end{align*}
Thus by Grönwall's Lemma, and noticing that $a\geq \nicefrac{1}{4\chi^2}$
\begin{align*}
I(t)\leq I(t_0)e^{-a(t-t_0)} + \nicefrac{\left(\frac{1}{4\chi^2} M_0\right)}{a}\leq I(t_0) + M_0.
\end{align*}
Furthermore $z\mapsto e^{-\frac{z}{2\chi}}-1+\frac{z}{2\chi}$ is decreasing on $\R_-$, hence by a Markov inequality, 
\begin{align*}
\int_{-\infty}^{-h} \omega(t,z)dz \leq \frac{I(t)}{e^{\frac{h}{2\chi} }-1-\frac{h}{2\chi}  }  \leq \frac{8\chi^2}{h^2} \left( I(t_0) + M_0\right),
\end{align*}
where we used that $e^{\frac{h}{2\chi} }-1-\frac{h}{2\chi} \geq \frac{h ^2}{8\chi^2}$. Hence, if we choose $h$ sufficiently large then for $t\geq t_0$, we have
\begin{align}
\label{OmegaLeftTail}
\int_{-\infty}^{-h} \omega(t,z)dz \leq \frac{1}{2}M_0.
\end{align}

\medskip

\textit{Step 2. (b): The mass of $\omega$ on $\R_+$ is lower bounded by the mass of $\omega$ on $[-h,0]$.} 

We use the parabolic Harnack inequality,  as stated for instance in \cite[Theorem 1.1]{trudinger1968}:
\begin{lemma}
Let $Q_1:=(-1,0)\times (-h,h)$ and $Q_2:=(-4,-2)\times (-h,h)$, then there exists a constant $C_0>0$, such that for every $t_2>t_0+4$, we have the following inequality:
\begin{align}
\label{harnack}
\sup_{(t_2,0)+Q_2} \omega \leq C_0\inf_{(t_2,0)+Q_1} \omega.
\end{align}
\end{lemma}
We use \eqref{harnack} to establish that for every $t_2>t_0+4$,
\begin{align}
\label{csq-harnack}
\int_{t_2-4}^{t_2} \int_{\R_+} \omega(s,z)dzds \geq \frac{1}{1+2C_0}M_0.
\end{align}

Suppose first that, for every $s\in(0,1)$,
\begin{align*}
\int_{\R_+}\omega(t_2-s,z)dz
\geq
\frac{1}{1+2C_0}M_0.
\end{align*}
Integrating over $s\in(0,1)$ immediately gives
\eqref{csq-harnack}.

Suppose now that there exists $s\in(0,1)$ such that
\begin{align*}
\int_{\R_+}\omega(t_2-s,z)dz <\frac{1}{1+2C_0}M_0.
\end{align*}
Since $(t_2-s,z)\in(t_2,0)+Q_1$ for $z\in(0,h)$, we obtain
\begin{align*}
\frac{1}{1+2C_0}M_0>\int_{\R_+} \omega(t_2-s,z)dz\geq \int_0^h \omega(t_2-s,z)dz \geq h \inf_{(t_2,0)+Q_1} \omega.
\end{align*}
By \eqref{harnack}, it follows that
\begin{align*}
h\sup_{(t_2,0)+Q_2}\omega
\leq
C_0h\inf_{(t_2,0)+Q_1}\omega
\leq
\frac{C_0}{1+2C_0}M_0.
\end{align*}
Therefore, for every
$\tau\in(t_2-4,t_2-2)$,
\begin{align*}
\int_{-h}^0\omega(\tau,z)dz
\leq
\frac{C_0}{1+2C_0}M_0.
\end{align*}
Then, using conservation of the mass of $\omega$ and \eqref{OmegaLeftTail}, we deduce that\begin{align*}
&\int_{\R_+}\omega(\tau,z) dz\\
=&
M_0 - \int_{-\infty}^{-h}\omega(\tau,z)dz
-\int_{-h}^0\omega(\tau,z)dz\\
\geq&
\frac12M_0 -\frac{C_0}{1+2C_0}M_0 =
\frac{1}{2(1+2C_0)}M_0.
\end{align*}
Integrating this inequality over
$\tau\in(t_2-4,t_2-2)$ gives
\begin{align*}
\int_{t_2-4}^{t_2-2}
\int_{\R_+}\omega(\tau,z)dzd\tau
\geq
\frac{1}{1+2C_0}M_0.
\end{align*}
This implies \eqref{csq-harnack} and concludes \textit{Step 2}.

\medskip

\textit{Step 3: Inequality \eqref{csq-harnack} contradicts $N(t,0)=N_\thresh$. } 

We multiply \eqref{movingmodel:b} by $e^u$ and integrate over $\R_+$. Using $u(0)=0$ and $u'(0^+)=1/\chi$, we obtain
\begin{align*}
&\frac{d}{dt}\int_{\R_+}e^u N dz \\
=& \dot{\bar{x}}\int_{\R_+}e^u\dz N dz +D \int_{\R_+}e^u\dzz N -  \int_{\R_+}\r e^u N 
\\
=&-\dot{\bar{x}} \left(N(t,0)+\int_{\R_+}N u'e^u dz \right)\\
&+D\left(\underbrace{-\dz N(t,0)}_{\leq 0} +
\frac{  N(t,0)}{\chi}+\int_{\R_+}\left(u''+u'^2\right)e^u N dz\right)-  \int_{\R_+}  \r e^u Ndz
\\
\leq& C(1+\norm{\dot{\bar{x}}}_\infty) \norm{N}_\infty-   N_\thresh \int_{\R_+} \r e^u dz\\
\leq& C(1+\norm{\dot{\bar{x}}}_\infty)\norm{N}_\infty  -   N_\thresh \int_{\R_+} \omega e^{\eta (t-t_0)} dz,
\end{align*}
where we used the fact that $e^u$ and its derivatives are integrable, $N$ is bounded below by $N_\thresh$ on $\R_+$ and $\r e^u$ is bounded below by the subsolution $\omega e^{\eta (t-t_0)}$. Finally, we integrate between $[t-4,t]$ and obtain
\begin{align*}
&\int_{\R_+}e^u N(t,z)dz \\ \leq& \int_{\R_+}e^u N(t-4,z)dz +4C(1+\norm{\dot{\bar{x}}}_\infty) -   N_\thresh  \int_{t-4}^t\int_{\R_+}e^{\eta (s-t_0)} \omega  dzds\\
\leq& \norm{N}_{\infty}\int_{\R_+}e^u dz +C(1+\norm{\dot{\bar{x}}}_\infty) -   N_\thresh e^{\eta(t-4-t_0)} \int_{t-4}^t\int_{\R_+} \omega  dz\\
\leq & C(1+\norm{N}_{\infty}+\norm{\dot{\bar{x}}}_\infty)  -\frac{    N_\thresh e^{\eta(t-4-t_0)}}{1+2C_0}M_0,
\end{align*}
where $C_0$ is the constant appearing in \eqref{harnack} from \textit{Step 2. (b)}. 
Since $\omega(t_0,\cdot)$ is nonnegative and nontrivial, $M_0>0$ and the right-hand side
tends to $-\infty$ as $t\to+\infty$. This contradicts
\begin{align*}
\int_{\R_+}e^uN(t,z)dz\geq0,
\end{align*}
and concludes the proof.
\end{proof}

It remains to prove \Cref{lemma-techical}.

\begin{proof}[Proof of \Cref{lemma-techical}]
Consider the function $f (z)=z^\e$. Then $f'(z)z=\e z^{\e}$. Hence $\sup_{z\in[0,1]} f'(z)z \leq \e$. Set
\begin{align*}
\theta_\e:z\mapsto \left\{ 
\begin{array}{ll}
\e^{(\e-3)\left(1+\frac{1}{\e}\right)}\left(\e^{1+\frac{1}{\e}}(3-\e)z^2-(2-\e)z^3 \right) & \text{if } z\in\left[0,\e^{1+\frac{1}{\e}}\right) \\
 f(z) & \text{if } z\in[\e^{1+\frac{1}{\e}},1-\e) \\
g(z)& \text{if } z\in[1-\e,1]
\end{array}
\right.,
\end{align*}
where $g$ is any concave $C^1$ function such that $g(1-\e)=f(1-\e)$, $g'(1-\e)=f'(1-\e)$, $g(1)=1$ and $g'(1)=0$. Notice that $g'(1-\e)\geq g'(1)$, hence such a concave function $g$ exists.

First notice that by construction $\theta_\e(0)=\theta_\e'(0)=0$, $\theta_\e\left(\left(\e^{1+\frac{1}{\e}}\right)^-\right)=f\left(\e^{1+\frac{1}{\e}}\right)$ and $\theta_\e'\left(\left(\e^{1+\frac{1}{\e}}\right)^-\right)=f'\left(\e^{1+\frac{1}{\e}}\right)$. By straightforward computations, we show that $\theta_\e$ is increasing on the interval $\left[0,\e^{1+\frac{1}{\e}}\right)$ and that $\theta_\e'(z)z$ reaches its maximum on the interval $\left[0,\e^{1+\frac{1}{\e}}\right)$ at point $z_\e:=\left(\frac{2}{3}\right)^2\frac{3-\e}{2-\e}\e^{1+\frac{1}{\e}}$, with value:
\begin{align*}
\theta_\e'(z_\e)z_\e &=\left(\frac{2}{3}\right)^5\frac{(3-\e)^3}{(2-\e)^2}\e^{\e+1} \\
&\leq \left(\frac{2}{3}\right)^5\frac{3^3}{1^2}\e^{\e}\e \\
&\leq \frac{32}{9}\e.
\end{align*}
For $z\in [1-\e,1]$, $\theta_\e$ is increasing, as by concavity of $g$, we have that $g'(z)\geq g'(1)=0$. In addition,
\begin{align*}
\theta_\e'(z)z=g'(z)z\leq g'(1-\e)z\leq g'(1-\e)=\frac{(1-\e)f'(1-\e)}{1-\e}\leq \frac{\e}{1-\e}.
\end{align*}
Hence $\sup_{z\in [0,1]}\theta'_\e(z)z\leq K \e$ for a constant $K>0$ independent of $\e$, which up to relabeling $\e$ proves the claimed bound. Moreover, by construction, we have that $\theta_\e\in W^2([0,1];[0,1])$
\end{proof}

\section*{Acknowledgements}

The author is extremely grateful to Vincent Calvez, whose insightful guidance has played a crucial role in the work presented here. The author also thanks Christophe Anjard, Olivier Cochet-Escartin and Jean-Paul Rieu for the fruitful collaboration, which has been the starting point of the present work. Finally, the author wishes to thank Jimmy Garnier, L{\'e}o Girardin and Fran{\c c}ois Hamel for valuable explanations and suggestions.\\
This project has received funding from the European Research Council (ERC) under the European Union’s Horizon 2020 research and innovation programme (grant agreement No 865711).

\printbibliography 

\end{document}